\documentclass[11pt]{article}
\title{Airy limit for the Jack processes and topological expansion}
\author{Jiaming Xu}
\date{}
\AtBeginDocument{\maketitle}

\usepackage[utf8]{inputenc}
\usepackage[margin=1in]{geometry}
\usepackage{amssymb, amsmath, amsthm}
\usepackage{mathtools}
\usepackage{graphicx}
\usepackage{enumitem}
\usepackage{tikz}
\usepackage{mathrsfs}
\usepackage{comment}
\usepackage{xcolor}
\usepackage[T1]{fontenc}
\usepackage{hyperref}
\usepackage{stmaryrd}

\theoremstyle{plain}
\newtheorem{lemma}{Lemma}[section]
\newtheorem{claim}{Claim}[section]
\newtheorem{proposition-definition}[lemma]{Proposition-Definition}

\newtheorem{theorem}[lemma]{Theorem}
\newtheorem{proposition}[lemma]{Proposition}\newtheorem{corollary}[lemma]{Corollary}
\newtheorem{theorem-definition}[lemma]{Theorem-Definition}

\theoremstyle{definition}
\newtheorem{definition}[lemma]{Definition}

\theoremstyle{remark}
\newtheorem{remark}[lemma]{Remark}

\allowdisplaybreaks

\begin{document}

\begin{abstract}
For every $\beta>0$, we establish multi-time soft-edge moment convergence for a class of Jack processes, including the Jack--Plancherel process, a discrete $\beta$-analogue of Dyson Brownian motion.  The limiting moments admit an absolutely convergent expansion in terms of nonnegative Brownian bridges decorated by pairs of equal and opposite jumps.  At equal times, we identify this limit with the joint Laplace statistics of the Airy$_\beta$ process for $\beta\geq1$.  The same Brownian expansion yields an asymptotic $\beta$-topological expansion of the marginal $b$-conjecture and, at $\beta=2$, an explicit nonnegative formula for the Witten--Kontsevich intersection numbers.
\end{abstract}


\section{Introduction}

\subsection{Overview}

Since Wigner's seminal work~\cite{Wigner}, random matrices have been used to model energy levels in nuclear physics and have revealed many surprising connections with statistics, geometry, number theory, and quantum field theory.  Two of the most classical examples are the Gaussian orthogonal and unitary ensembles (GOE and GUE), the Hermitian Gaussian matrices with i.i.d. real/complex entries above the diagonal.  They play a fundamental role in the study of general random matrices: their eigenvalues admit explicit joint densities, with several complementary descriptions.

Among these interpretations, one of the classical discoveries was inspired by 't Hooft~\cite{tHooft} and Cicuta~\cite{CicutaTopological}, who wrote the large-$N$ expansion of the Gaussian matrix moments as a weighted sum over closed surfaces. The mechanism is wonderfully concrete.  After a product of traces of degrees $(k_1,\ldots,k_m)$ is expanded into matrix entries, Wick's formula contracts the entries pairwise. An entry indexed by $(i,j)$ can contract only with one indexed by $(j,i)$ or $(i,j)$.  If $k_1+\cdots+k_m=2E$ (otherwise the moment vanishes), each Wick pairing may be identified with a gluing of $m$ polygons having $k_1,\ldots,k_m$ sides into a closed, possibly disconnected surface.  For GUE, the covariance of entries forces the polygons to be glued into an orientable surface, while for GOE non-orientable surfaces also occur.

Equivalently, each pairing defines a map, namely a possibly twisted ribbon graph: vertices are thickened to discs, edges to ribbons, and the cyclic orders inherited from the traces specify how the ribbons meet the discs. Recall that a connected closed orientable surface is classified by its number $g$ of handles and has Euler characteristic $2-2g$, and a connected closed non-orientable surface is classified by its number $c$ of crosscaps and has Euler characteristic $2-c=2-2g$, now with $g\in\frac12\mathbb Z_{\geq0}$. For $\beta=1,2$, let $X_N^{(\beta)}$ be the $N\times N$
GOE matrix for $\beta=1$ and GUE matrix for $\beta=2$, normalized in a way that its diagonal entries have variance $2/\beta$. Let $\mathfrak G_\beta(\boldsymbol k)$ denote the corresponding Wick gluings of the $m$ labeled polygons, and let $\chi(\mathfrak M)$ be the sum of the Euler characteristics of the connected components of $\mathfrak M$.  Then we have the exact moment identity
\[
  \mathbb E\!\left[
    \prod_{p=1}^m
    \operatorname{Tr}\!\left((X_N^{(\beta)})^{k_p}\right)
  \right]
  =
  \sum_{\mathfrak M\in\mathfrak G_\beta(\boldsymbol k)}
  N^{E-m+\chi(\mathfrak M)}.
\]

This dictionary soon acquired striking geometric consequences: it was used by Harer and Zagier~\cite{HarerZagier} to compute the orbifold Euler characteristics of moduli spaces of curves. A generalization of such a connection was later given by Kontsevich, who characterized the intersection numbers by a ribbon-graph cell decomposition on a certain matrix model, and consequently proved Witten's conjecture~\cite{Witten,Kontsevich}.  These ideas have since become central in two-dimensional quantum gravity, topological recursion, and the study of matrix-group integrals and mapping class groups~\cite{EO,MageePuderI,MageePuderII}.

A useful way to pass beyond the classical real and complex matrix models is
to regard a beta ensemble as an $N$-particle system on $\mathbb R$ (or, in
its two-dimensional analogue, on $\mathbb C$), with $\beta$ playing the role
of inverse temperature.  The classical eigenvalue laws occur at $\beta=1,2$, while the particle-system formulation extends naturally to
every $\beta>0$.  For general $\beta$, the loop equations of Chekhov and Eynard \cite{ChekhovEynardBeta}, together with
their analytic realization by Borot and Guionnet \cite{BorotGuionnetOneCut}, provide a recursive characterization of the joint moments of the linear statistic.  In general, such descriptions do not come with
combinatorially explicit formulas.

The Gaussian beta ensemble is exceptional in this respect.  Goulden and
Jackson~\cite{GouldenJackson} conjectured an explicit formula for
Jack-polynomial expectations in this ensemble, thereby linking the Gaussian
beta ensemble directly to Jack symmetric functions; their conjecture was
proved by Okounkov~\cite{OkounkovGouldenJackson}.  More broadly, the
Goulden--Jackson $b$-conjecture predicts that the power-sum coefficients of the logarithmic derivative of the Jack series
\[
  \Phi(\boldsymbol x,\boldsymbol y,\boldsymbol z;t,b)
  =\sum_{\lambda}t^{|\lambda|}
  \frac{
    J_\lambda(\boldsymbol x;1+b)
    J_\lambda(\boldsymbol y;1+b)
    J_\lambda(\boldsymbol z;1+b)}
  {j_\lambda(1+b)\,
   [p_1^{|\lambda|}]J_\lambda(\boldsymbol p;1+b)}
\]
are polynomials in \begin{equation*}
b=\frac{2}{\beta}-1
\end{equation*}
with nonnegative integer coefficients admitting a map interpretation. It is a
central problem in algebraic combinatorics. A simplified version, the marginal $b$-conjecture, is obtained by the
specialization
\[
  p_k(\boldsymbol x)=N,\qquad
  p_k(\boldsymbol z)=\delta_{k,2},
\]
while keeping $\boldsymbol y$ arbitrary. This was done by La Croix~\cite{LaCroix}, and his proof identifies the marginal map series with moments
of the Gaussian beta ensemble. The $\beta$-map series naturally continues the classical cases: at $\beta=2$, $b=0$ selects orientable maps, while at $\beta=1$,
$b=1$ counts orientable and non-orientable surfaces without distinguishing
their degree of non-orientability.  For general $\beta$, the map series consists of a coefficient factor $b^{\eta(\mathfrak M)}$, where $\eta(\mathfrak M)$ is a
recursive measure of non-orientability of the map.  Substantial extensions of the
Jack--map correspondence were obtained by Do\l\k{e}ga, F\'eray, Chapuy, and
their collaborators~\cite{DolegaFeray,ChapuyDolega}.  

From the probabilistic viewpoint, one naturally takes $N\to\infty$ and
studies the asymptotic behavior of the eigenvalues.  Loop equations, in both
their continuous and discrete forms, may be viewed as analytic counterparts
of topological recursion and have become central tools for global laws,
fluctuations, and rigidity in beta ensembles
\cite{JohanssonFluctuations,LebleSerfaty,BGG,GH,BorotGorinGuionnet}.
Their use at microscopic scales is more recent.
Huang--McKenzie--Yau~\cite{HMY} employ microscopic loop equations in their
proof of edge universality for random regular graphs, while
Bourgade--Huang~\cite{BourgadeHuangLoop} show, for rational $\beta>0$, that
the universal static bulk and edge point processes are characterized by their local
loop-equation hierarchies.  At the soft edge, however, the hierarchy still
presents the limiting law recursively. 
A complementary route is available for the classical matrix models.
Okounkov~\cite{OkounkovRandomMatrices} related asymptotic map 
expansions of the GUE to the Airy point process and intersection
theory, while the moment method and its ribbon-graph interpretation were generalized by
Soshnikov~\cite{SoshnikovEdge}, Sodin~\cite{SodinMoments},
Liu--Zou~\cite{LiuZou}, and others in the proofs of edge universality for a general class of real and complex matrices. Their results provide strong asymptotic
information, but are not usually organized into closed
formulas for individual map coefficients and do not naturally
interpolate in $\beta$.

As shown in \cite{BEY}, the universal object at a regular soft edge is the Airy$_\beta$ point
process.  Write
$\mathcal A_1^{(\beta)}>\mathcal A_2^{(\beta)}>\cdots$ for its points.  It
may be defined from the stochastic Airy operator
\[
  -\frac{\mathrm d^2}{\mathrm dx^2}+x+
  \frac2{\sqrt\beta}\,\dot W(x)
\]
on $\mathbb R_{\geq0}$ with Dirichlet boundary condition: if its eigenvalues
are $\Lambda_1<\Lambda_2<\cdots$, then
$\mathcal A_i^{(\beta)}=-\Lambda_i$ \cite{RRV}.  A natural dynamical extension of the Airy$_\beta$ process is the
Airy$_\beta$ line ensemble, which is expected to be the universal soft-edge
scaling limit for a broad class of time-dependent beta ensembles and related
models; see, for example,~\cite{GXZ,HZ}.  At $\beta=2$, it reduces to the
classical Airy line ensemble, a central object in the KPZ universality class,
with additional determinantal structure and Brownian-Gibbs property.  For general
$\beta$, however, much less is known about the dynamical object than about
either this classical case or its one-time marginals.  Existing universality
frameworks are largely based on stochastic differential equations~\cite{Landon,HZ}. Explicit descriptions are already scarce at one time, and are still rarer in the multi-time setting.  Gorin--Shkolnikov~\cite{GorinShkolnikovAiry}
obtained a Feynman--Kac representation through the stochastic Airy semigroup,
and Gorin--Xu--Zhang~\cite{GXZ} obtained an explicit integral formula for the
multi-time Laplace statistics.
To our knowledge, no comparable direct analogue of the Airy kernel is
available for the correlation functions at general $\beta$.

The main probabilistic object studied in the present paper is the Jack--Plancherel process. It is a
continuous-time Markov process on
\[
  \mathbb Y(N):=\{\lambda:\ell(\lambda)\leq N\},
\]
the set of partitions with at most $N$ rows, and a fundamental element in the broader class of Jack processes. Its one-time marginal is the
Poissonized Jack--Plancherel measure $\mathbb P_{N,s}^{(\theta)}$ defined
in~\eqref{eq:Jack-Plancherel-measure}, which degenerates to  the classical Poissonized
Plancherel measure when $\theta=1$, equivalently $\beta=2$.  The latter and its asymptotic behavior have played a
central role in asymptotic representation theory and integrable
probability; see, among many others,
\cite{LoganShepp,VershikKerov,BDJ,BOO,OkounkovRandomMatrices}.
When $\theta=1$, the shifted rows of
$\lambda^{(N)}(s)$ form a system of nonintersecting Poisson random walks. For general $\beta=2\theta$, the
Jack--Plancherel process is a natural discrete $\beta$-analogue of Dyson
Brownian motion; see \cite{GS,Huang,GorinHuangDynamical} and the references
therein.

The present paper has two main outputs. First, we
establish multi-time soft-edge moment convergence for the
Jack--Plancherel process at time displacements of order $N^{2/3}$, and later extend it to a general class of Jack processes. The limit is
given by an explicit Brownian functional, which we expect to describe the
corresponding Airy$_\beta$ line-ensemble statistics. Second, the one-time specialization yields a new explicit formula for the Laplace statistics of the Airy$_\beta$ point process for $\beta\ge1$, whose structure is directly connected with the topological expansion discussed above.

At a single time, and suppressing the fixed normalization and rescaling
constants, our formula reads
\[
\begin{aligned}
  \mathbb E\left[\prod_{p=1}^m\mathcal Z_\beta(s_p)\right]
  ={}&
  \sum_{r\geq0}\frac{(1+b)^r}{r!}
  \sum_{\substack{\text{multigraph}\\\text{assignments}}}
  \int
  \bigl(\text{killed Brownian-bridge densities}\bigr)
  \\
  &\qquad{}\times
  \exp\left\{
    b\cdot \text{const}\sum\operatorname{Area}
  \right\}
  \prod_{e=1}^r
  h_e\,\mathrm d\bigl(\text{times and heights}\bigr).
\end{aligned}
\]
The vertices of the multigraph are the $m$ excursion blocks; its edges are
the jump pairs; and $r$ is the number of edges.  An edge $e$ consists of a
downward and an upward jump of the same height $h_e$.  The integral ranges
over their jump times and heights and over the concatenated nonnegative
Brownian bridges; $\sum\operatorname{Area}$ denotes the total area under all
bridge pieces.  Upon expanding the area exponential, the exponent of $b$
records the corresponding measure of non-orientability. The multi-time functional has the same structure, with an additional exponential weight accounting for the time separations.

\begin{figure}[t]
\centering
\begin{tikzpicture}[x=.57cm,y=.70cm,>=stealth]

  \path[fill=cyan!8,draw=none]
    plot coordinates {
      (0,0) (.03,.14) (.05,.2) (.08,.43) (.1,.62) (.13,.7) (.15,.99) (.18,.87)
      (.2,.8) (.22,.77) (.25,.8) (.27,.92) (.3,.95) (.33,1.12) (.35,1.33) (.38,1.56)
      (.4,1.66) (.43,1.72) (.45,1.94) (.48,2.01) (.5,2.14) (.52,2.15) (.55,2.13) (.58,2.15)
      (.6,2.15) (.63,1.89) (.65,2) (.68,2.19) (.7,2.28) (.73,2.23) (.75,2.24) (.78,2.43)
      (.8,2.39) (.83,2.5) (.85,2.39) (.88,2.3) (.9,2.31) (.93,2.28) (.95,2.39) (.98,2.19)
      (1,2.07) (1.02,2.05) (1.05,2.28) (1.08,2.11) (1.1,2.37) (1.13,2.37) (1.15,2.23) (1.18,2.09)
      (1.2,2.13) (1.23,2.17) (1.25,2.39) (1.27,2.63) (1.3,2.63) (1.33,2.35) (1.35,2.37) (1.38,2.39)
      (1.4,2.49) (1.42,2.62) (1.45,2.96) (1.48,2.94) (1.5,2.88) (1.53,2.86) (1.55,2.7) (1.58,2.82)
      (1.6,2.76) (1.63,2.67) (1.65,2.71) (1.68,2.82) (1.7,2.93) (1.73,2.89) (1.75,2.95) (1.78,3.01)
      (1.8,2.84) (1.82,2.56) (1.85,2.53) (1.88,2.36) (1.9,2.48) (1.93,2.5) (1.95,2.47) (1.98,2.57)
      (2,2.46) (2.03,2.64) (2.05,2.69) (2.08,2.87) (2.1,2.98) (2.13,2.95) (2.15,3) (2.17,3.16)
      (2.2,3.01) (2.23,3.21) (2.25,3.19) (2.28,3.12) (2.3,3.07) (2.33,3.06) (2.35,3.04) (2.38,3.02)
      (2.4,3.26) (2.42,3.43) (2.45,3.31) (2.47,3.18) (2.5,3.16) (2.53,3.15) (2.55,2.88) (2.58,2.71)
      (2.6,3) (2.63,3.08) (2.65,2.93) (2.68,2.87) (2.7,2.79) (2.73,3) (2.75,2.79) (2.78,2.94)
      (2.8,2.99) (2.83,2.82) (2.85,2.95) (2.88,2.95) (2.9,2.99) (2.93,3.01) (2.95,3.1) (2.98,3.39)
      (3,3.3) (3,1.7) (3.03,1.77) (3.05,1.83) (3.08,1.78) (3.1,1.93) (3.13,2.05) (3.15,2.22)
      (3.18,2.33) (3.2,2.4) (3.23,2.44) (3.25,2.7) (3.28,2.6) (3.3,2.66) (3.33,2.9) (3.35,3.06)
      (3.38,3.13) (3.4,2.99) (3.43,3.05) (3.45,2.97) (3.48,2.73) (3.5,2.73) (3.53,2.97) (3.55,3.22)
      (3.58,3.25) (3.6,3.23) (3.63,3.09) (3.65,3.41) (3.68,3.23) (3.7,2.89) (3.73,2.95) (3.75,2.88)
      (3.78,3.07) (3.8,3.2) (3.83,3.21) (3.85,3.34) (3.88,3.21) (3.9,3.46) (3.93,3.56) (3.95,3.48)
      (3.98,3.38) (4,3.34) (4.03,3.06) (4.05,3.04) (4.08,2.84) (4.1,2.87) (4.13,2.95) (4.15,2.98)
      (4.18,3.01) (4.2,2.87) (4.22,2.97) (4.25,2.68) (4.28,2.64) (4.3,2.97) (4.33,3.08) (4.35,3)
      (4.38,3.09) (4.4,3.07) (4.43,2.86) (4.45,2.67) (4.47,2.61) (4.5,2.44) (4.53,2.4) (4.55,2.49)
      (4.58,2.52) (4.6,2.39) (4.63,2.37) (4.65,2.35) (4.68,2.19) (4.7,2.14) (4.72,1.96) (4.75,1.88)
      (4.78,1.62) (4.8,1.63) (4.83,1.78) (4.85,1.96) (4.88,1.94) (4.9,2.13) (4.93,2.22) (4.95,2.4)
      (4.97,2.21) (5,2.1) (5,3.7) (5.03,3.22) (5.05,2.98) (5.08,2.97) (5.1,2.85) (5.13,2.73)
      (5.15,2.55) (5.18,2.31) (5.2,2.28) (5.23,2.3) (5.25,2.42) (5.28,2.51) (5.3,2.37) (5.33,2.76)
      (5.35,2.8) (5.38,2.73) (5.4,2.82) (5.43,2.86) (5.45,2.46) (5.48,2.35) (5.5,2.17) (5.53,2.21)
      (5.55,2.07) (5.58,2.12) (5.6,2.1) (5.63,2.37) (5.65,2.31) (5.68,2.32) (5.7,2.31) (5.73,2.22)
      (5.75,2.23) (5.78,2.38) (5.8,2.2) (5.83,1.89) (5.85,2.11) (5.88,2.34) (5.9,2.5) (5.93,2.55)
      (5.95,2.68) (5.98,2.74) (6,2.95) (6.03,3.12) (6.05,3.01) (6.08,2.96) (6.1,3.26) (6.13,2.99)
      (6.15,2.96) (6.18,2.94) (6.2,2.72) (6.23,2.71) (6.25,2.7) (6.28,2.98) (6.3,2.61) (6.33,2.57)
      (6.35,2.52) (6.38,2.38) (6.4,2.48) (6.43,2.72) (6.45,2.73) (6.48,2.93) (6.5,2.99) (6.53,3.01)
      (6.55,3.21) (6.58,3.44) (6.6,3.51) (6.63,3.92) (6.65,4.01) (6.68,3.74) (6.7,3.85) (6.73,3.67)
      (6.75,3.72) (6.78,3.66) (6.8,3.61) (6.83,3.53) (6.85,3.55) (6.88,3.31) (6.9,3.44) (6.93,3.16)
      (6.95,2.71) (6.98,2.69) (7,2.8) (7,1.3) (7.03,1.15) (7.05,1.13) (7.08,.99) (7.1,1.06)
      (7.13,1.18) (7.15,1.14) (7.18,1.12) (7.2,1.16) (7.23,1.21) (7.25,1.1) (7.28,.62) (7.3,.61)
      (7.33,.81) (7.35,.81) (7.38,1.03) (7.4,1.2) (7.43,1.03) (7.45,1.14) (7.48,1.33) (7.5,1.29)
      (7.53,1.71) (7.55,1.77) (7.58,1.82) (7.6,1.96) (7.63,1.9) (7.65,1.84) (7.68,1.88) (7.7,2.15)
      (7.73,2.12) (7.75,2.3) (7.78,2.14) (7.8,2.06) (7.83,1.99) (7.85,2.19) (7.88,2.08) (7.9,2.02)
      (7.93,2.18) (7.95,2.5) (7.98,2.44) (8,2.43) (8.03,2.31) (8.05,2.19) (8.07,2.28) (8.1,2.09)
      (8.13,1.85) (8.15,1.65) (8.18,1.69) (8.2,1.8) (8.23,1.76) (8.25,1.74) (8.28,1.79) (8.3,1.61)
      (8.32,1.49) (8.35,1.38) (8.38,1.31) (8.4,1.4) (8.43,1.48) (8.45,1.45) (8.48,1.56) (8.5,1.34)
      (8.53,1.42) (8.55,1.38) (8.57,1.5) (8.6,1.57) (8.63,1.36) (8.65,1.22) (8.68,1.26) (8.7,1.39)
      (8.73,1.09) (8.75,1.01) (8.78,.91) (8.8,.78) (8.82,.5) (8.85,.61) (8.88,.22) (8.9,.34)
      (8.93,.37) (8.95,.44) (8.98,.28) (9,0)
    } -- cycle;
  \path[fill=cyan!8,draw=none]
    plot coordinates {
      (9,0) (9.03,.27) (9.05,.3) (9.07,.27) (9.1,.38) (9.13,.4) (9.15,.44) (9.18,.22)
      (9.2,.31) (9.23,.14) (9.25,.22) (9.28,.37) (9.3,.41) (9.32,.58) (9.35,.72) (9.38,.49)
      (9.4,.39) (9.43,.33) (9.45,.59) (9.48,.7) (9.5,.68) (9.53,.75) (9.55,.48) (9.57,.56)
      (9.6,.5) (9.63,.49) (9.65,.76) (9.68,.73) (9.7,.69) (9.73,.58) (9.75,.79) (9.78,.78)
      (9.8,1.08) (9.82,1.08) (9.85,1.13) (9.88,1.04) (9.9,1.08) (9.93,.85) (9.95,1) (9.98,1.14)
      (10,1.18) (10.03,1.02) (10.05,1.05) (10.07,.85) (10.1,.67) (10.13,.36) (10.15,.34) (10.18,.47)
      (10.2,.26) (10.23,.38) (10.25,.22) (10.28,.31) (10.3,.56) (10.33,.37) (10.35,.34) (10.38,.24)
      (10.4,.3) (10.43,.32) (10.45,.22) (10.48,.3) (10.5,.41) (10.53,.43) (10.55,.25) (10.58,.44)
      (10.6,.7) (10.63,.71) (10.65,.72) (10.68,.78) (10.7,.72) (10.73,.47) (10.75,.3) (10.78,.5)
      (10.8,.86) (10.83,.93) (10.85,.85) (10.88,.78) (10.9,.95) (10.93,.94) (10.95,.92) (10.98,1.29)
      (11,1.2) (11,2.7) (11.03,2.59) (11.05,2.47) (11.08,2.38) (11.1,2.32) (11.13,2.25) (11.15,2.13)
      (11.18,2.04) (11.2,2) (11.23,1.96) (11.25,2.25) (11.28,2.13) (11.3,2.1) (11.33,2.1) (11.35,2.01)
      (11.38,2.16) (11.4,2.01) (11.43,1.61) (11.45,1.69) (11.48,1.79) (11.5,1.83) (11.53,1.73) (11.55,1.67)
      (11.58,1.79) (11.6,1.64) (11.63,1.74) (11.65,2.1) (11.68,1.9) (11.7,1.81) (11.73,1.89) (11.75,1.92)
      (11.78,1.83) (11.8,1.85) (11.83,1.93) (11.85,1.87) (11.88,1.99) (11.9,2.2) (11.93,2.15) (11.95,2.18)
      (11.98,2.23) (12,1.8) (12.03,1.87) (12.05,1.96) (12.08,1.96) (12.1,1.96) (12.13,1.65) (12.15,1.59)
      (12.18,1.91) (12.2,1.79) (12.23,1.83) (12.25,1.89) (12.28,1.8) (12.3,1.38) (12.33,1.36) (12.35,1.12)
      (12.38,.89) (12.4,1.1) (12.43,1.23) (12.45,1.18) (12.48,1.28) (12.5,1.06) (12.53,.93) (12.55,.83)
      (12.58,.97) (12.6,.65) (12.63,.66) (12.65,.52) (12.68,.74) (12.7,.79) (12.73,1.13) (12.75,1.22)
      (12.78,1.09) (12.8,.93) (12.83,1.11) (12.85,1.14) (12.88,1.26) (12.9,1.46) (12.93,1.44) (12.95,1.32)
      (12.98,1.36) (13,1.35) (13.03,1.04) (13.05,1.24) (13.08,1.05) (13.1,1.06) (13.13,1.05) (13.15,1.41)
      (13.18,1.38) (13.2,1.36) (13.23,1.72) (13.25,1.89) (13.28,1.79) (13.3,1.91) (13.33,2.05) (13.35,2.3)
      (13.38,2.23) (13.4,2.07) (13.43,2.08) (13.45,2.04) (13.48,2.06) (13.5,2.05) (13.53,1.98) (13.55,1.82)
      (13.58,1.73) (13.6,1.81) (13.63,1.99) (13.65,2.12) (13.68,2) (13.7,1.95) (13.73,1.77) (13.75,1.77)
      (13.78,1.49) (13.8,1.66) (13.83,1.39) (13.85,1.59) (13.88,1.87) (13.9,1.91) (13.93,1.7) (13.95,1.44)
      (13.98,1.36) (14,1.17) (14.03,1.49) (14.05,1.55) (14.08,1.62) (14.1,1.44) (14.13,1.52) (14.15,1.55)
      (14.18,1.68) (14.2,1.54) (14.23,1.37) (14.25,1.25) (14.28,1.04) (14.3,.94) (14.33,.93) (14.35,.79)
      (14.38,.87) (14.4,1.05) (14.43,1.16) (14.45,1.3) (14.48,1.29) (14.5,1.28) (14.53,1.2) (14.55,1.15)
      (14.58,1.2) (14.6,1.25) (14.63,1.23) (14.65,1.53) (14.68,1.44) (14.7,1.27) (14.73,1.35) (14.75,1.44)
      (14.78,1.5) (14.8,1.52) (14.83,1.67) (14.85,1.83) (14.88,1.86) (14.9,1.97) (14.93,1.94) (14.95,1.93)
      (14.98,1.95) (15,2.01) (15.02,2.18) (15.05,2.32) (15.08,2.35) (15.1,2.1) (15.13,2.29) (15.15,2.49)
      (15.18,2.33) (15.2,2.03) (15.23,2.07) (15.25,2.27) (15.28,2.14) (15.3,2.08) (15.33,2.25) (15.35,1.87)
      (15.38,1.98) (15.4,1.7) (15.43,1.74) (15.45,1.9) (15.48,1.99) (15.5,1.97) (15.53,2.23) (15.55,2.27)
      (15.58,2.25) (15.6,2.26) (15.63,2.14) (15.65,2.09) (15.68,2.15) (15.7,1.89) (15.73,1.84) (15.75,2.07)
      (15.78,2.12) (15.8,2.08) (15.83,2.31) (15.85,2.25) (15.88,2.45) (15.9,2.54) (15.93,2.53) (15.95,2.58)
      (15.98,2.47) (16,2.62) (16.02,2.54) (16.05,2.79) (16.08,2.79) (16.1,2.86) (16.13,2.64) (16.15,2.56)
      (16.18,2.45) (16.2,2.44) (16.23,2.15) (16.25,2.34) (16.27,2.1) (16.3,2.17) (16.33,2.03) (16.35,1.83)
      (16.38,1.91) (16.4,1.96) (16.43,1.7) (16.45,1.52) (16.48,1.3) (16.5,1.37) (16.52,1.41) (16.55,1.33)
      (16.58,1.41) (16.6,1.5) (16.63,1.7) (16.65,1.6) (16.68,2.12) (16.7,1.96) (16.73,1.96) (16.75,1.82)
      (16.77,1.62) (16.8,1.34) (16.83,1.38) (16.85,1.13) (16.88,.72) (16.9,.88) (16.93,.72) (16.95,.92)
      (16.98,.49) (17,.52) (17.02,.76) (17.05,1.05) (17.08,1.37) (17.1,1.23) (17.13,1.16) (17.15,1.21)
      (17.18,1.34) (17.2,1.61) (17.23,1.38) (17.25,1.27) (17.27,1.11) (17.3,.99) (17.33,.69) (17.35,.52)
      (17.38,.91) (17.4,.94) (17.43,.8) (17.45,.82) (17.48,.63) (17.5,.77) (17.52,.65) (17.55,.5)
      (17.58,.31) (17.6,.47) (17.63,.62) (17.65,.64) (17.68,.55) (17.7,.52) (17.73,.43) (17.75,.51)
      (17.77,.61) (17.8,.74) (17.83,.66) (17.85,.59) (17.88,.28) (17.9,.36) (17.93,.47) (17.95,.52)
      (17.98,.28) (18,0)
    } -- cycle;

  \draw[->] (0,0) -- (18.8,0) node[right] {$t$};
  \draw[->] (0,0) -- (0,4.7) node[above] {$B(t)$};
  \draw[gray,dashed] (9,0) -- (9,4.35);
  \node[below=3pt] at (0,0) {$Q_0$};
  \node[below=3pt] at (9,0) {$Q_1$};
  \node[below=3pt] at (18,0) {$Q_2$};
  \node[below=12pt] at (4.5,0) {excursion block $1$};
  \node[below=12pt] at (13.5,0) {excursion block $2$};

  \draw[black,line width=.75pt,line join=round,line cap=round]
    plot coordinates {
      (0,0) (.03,.14) (.05,.2) (.08,.43) (.1,.62) (.13,.7) (.15,.99) (.18,.87)
      (.2,.8) (.22,.77) (.25,.8) (.27,.92) (.3,.95) (.33,1.12) (.35,1.33) (.38,1.56)
      (.4,1.66) (.43,1.72) (.45,1.94) (.48,2.01) (.5,2.14) (.52,2.15) (.55,2.13) (.58,2.15)
      (.6,2.15) (.63,1.89) (.65,2) (.68,2.19) (.7,2.28) (.73,2.23) (.75,2.24) (.78,2.43)
      (.8,2.39) (.83,2.5) (.85,2.39) (.88,2.3) (.9,2.31) (.93,2.28) (.95,2.39) (.98,2.19)
      (1,2.07) (1.02,2.05) (1.05,2.28) (1.08,2.11) (1.1,2.37) (1.13,2.37) (1.15,2.23) (1.18,2.09)
      (1.2,2.13) (1.23,2.17) (1.25,2.39) (1.27,2.63) (1.3,2.63) (1.33,2.35) (1.35,2.37) (1.38,2.39)
      (1.4,2.49) (1.42,2.62) (1.45,2.96) (1.48,2.94) (1.5,2.88) (1.53,2.86) (1.55,2.7) (1.58,2.82)
      (1.6,2.76) (1.63,2.67) (1.65,2.71) (1.68,2.82) (1.7,2.93) (1.73,2.89) (1.75,2.95) (1.78,3.01)
      (1.8,2.84) (1.82,2.56) (1.85,2.53) (1.88,2.36) (1.9,2.48) (1.93,2.5) (1.95,2.47) (1.98,2.57)
      (2,2.46) (2.03,2.64) (2.05,2.69) (2.08,2.87) (2.1,2.98) (2.13,2.95) (2.15,3) (2.17,3.16)
      (2.2,3.01) (2.23,3.21) (2.25,3.19) (2.28,3.12) (2.3,3.07) (2.33,3.06) (2.35,3.04) (2.38,3.02)
      (2.4,3.26) (2.42,3.43) (2.45,3.31) (2.47,3.18) (2.5,3.16) (2.53,3.15) (2.55,2.88) (2.58,2.71)
      (2.6,3) (2.63,3.08) (2.65,2.93) (2.68,2.87) (2.7,2.79) (2.73,3) (2.75,2.79) (2.78,2.94)
      (2.8,2.99) (2.83,2.82) (2.85,2.95) (2.88,2.95) (2.9,2.99) (2.93,3.01) (2.95,3.1) (2.98,3.39)
      (3,3.3)
    };
  \draw[black,line width=.75pt,line join=round,line cap=round]
    plot coordinates {
      (3,1.7) (3.03,1.77) (3.05,1.83) (3.08,1.78) (3.1,1.93) (3.13,2.05) (3.15,2.22) (3.18,2.33)
      (3.2,2.4) (3.23,2.44) (3.25,2.7) (3.28,2.6) (3.3,2.66) (3.33,2.9) (3.35,3.06) (3.38,3.13)
      (3.4,2.99) (3.43,3.05) (3.45,2.97) (3.48,2.73) (3.5,2.73) (3.53,2.97) (3.55,3.22) (3.58,3.25)
      (3.6,3.23) (3.63,3.09) (3.65,3.41) (3.68,3.23) (3.7,2.89) (3.73,2.95) (3.75,2.88) (3.78,3.07)
      (3.8,3.2) (3.83,3.21) (3.85,3.34) (3.88,3.21) (3.9,3.46) (3.93,3.56) (3.95,3.48) (3.98,3.38)
      (4,3.34) (4.03,3.06) (4.05,3.04) (4.08,2.84) (4.1,2.87) (4.13,2.95) (4.15,2.98) (4.18,3.01)
      (4.2,2.87) (4.22,2.97) (4.25,2.68) (4.28,2.64) (4.3,2.97) (4.33,3.08) (4.35,3) (4.38,3.09)
      (4.4,3.07) (4.43,2.86) (4.45,2.67) (4.47,2.61) (4.5,2.44) (4.53,2.4) (4.55,2.49) (4.58,2.52)
      (4.6,2.39) (4.63,2.37) (4.65,2.35) (4.68,2.19) (4.7,2.14) (4.72,1.96) (4.75,1.88) (4.78,1.62)
      (4.8,1.63) (4.83,1.78) (4.85,1.96) (4.88,1.94) (4.9,2.13) (4.93,2.22) (4.95,2.4) (4.97,2.21)
      (5,2.1)
    };
  \draw[black,line width=.75pt,line join=round,line cap=round]
    plot coordinates {
      (5,3.7) (5.03,3.22) (5.05,2.98) (5.08,2.97) (5.1,2.85) (5.13,2.73) (5.15,2.55) (5.18,2.31)
      (5.2,2.28) (5.23,2.3) (5.25,2.42) (5.28,2.51) (5.3,2.37) (5.33,2.76) (5.35,2.8) (5.38,2.73)
      (5.4,2.82) (5.43,2.86) (5.45,2.46) (5.48,2.35) (5.5,2.17) (5.53,2.21) (5.55,2.07) (5.58,2.12)
      (5.6,2.1) (5.63,2.37) (5.65,2.31) (5.68,2.32) (5.7,2.31) (5.73,2.22) (5.75,2.23) (5.78,2.38)
      (5.8,2.2) (5.83,1.89) (5.85,2.11) (5.88,2.34) (5.9,2.5) (5.93,2.55) (5.95,2.68) (5.98,2.74)
      (6,2.95) (6.03,3.12) (6.05,3.01) (6.08,2.96) (6.1,3.26) (6.13,2.99) (6.15,2.96) (6.18,2.94)
      (6.2,2.72) (6.23,2.71) (6.25,2.7) (6.28,2.98) (6.3,2.61) (6.33,2.57) (6.35,2.52) (6.38,2.38)
      (6.4,2.48) (6.43,2.72) (6.45,2.73) (6.48,2.93) (6.5,2.99) (6.53,3.01) (6.55,3.21) (6.58,3.44)
      (6.6,3.51) (6.63,3.92) (6.65,4.01) (6.68,3.74) (6.7,3.85) (6.73,3.67) (6.75,3.72) (6.78,3.66)
      (6.8,3.61) (6.83,3.53) (6.85,3.55) (6.88,3.31) (6.9,3.44) (6.93,3.16) (6.95,2.71) (6.98,2.69)
      (7,2.8)
    };
  \draw[black,line width=.75pt,line join=round,line cap=round]
    plot coordinates {
      (7,1.3) (7.03,1.15) (7.05,1.13) (7.08,.99) (7.1,1.06) (7.13,1.18) (7.15,1.14) (7.18,1.12)
      (7.2,1.16) (7.23,1.21) (7.25,1.1) (7.28,.62) (7.3,.61) (7.33,.81) (7.35,.81) (7.38,1.03)
      (7.4,1.2) (7.43,1.03) (7.45,1.14) (7.48,1.33) (7.5,1.29) (7.53,1.71) (7.55,1.77) (7.58,1.82)
      (7.6,1.96) (7.63,1.9) (7.65,1.84) (7.68,1.88) (7.7,2.15) (7.73,2.12) (7.75,2.3) (7.78,2.14)
      (7.8,2.06) (7.83,1.99) (7.85,2.19) (7.88,2.08) (7.9,2.02) (7.93,2.18) (7.95,2.5) (7.98,2.44)
      (8,2.43) (8.03,2.31) (8.05,2.19) (8.07,2.28) (8.1,2.09) (8.13,1.85) (8.15,1.65) (8.18,1.69)
      (8.2,1.8) (8.23,1.76) (8.25,1.74) (8.28,1.79) (8.3,1.61) (8.32,1.49) (8.35,1.38) (8.38,1.31)
      (8.4,1.4) (8.43,1.48) (8.45,1.45) (8.48,1.56) (8.5,1.34) (8.53,1.42) (8.55,1.38) (8.57,1.5)
      (8.6,1.57) (8.63,1.36) (8.65,1.22) (8.68,1.26) (8.7,1.39) (8.73,1.09) (8.75,1.01) (8.78,.91)
      (8.8,.78) (8.82,.5) (8.85,.61) (8.88,.22) (8.9,.34) (8.93,.37) (8.95,.44) (8.98,.28)
      (9,0)
    };
  \draw[black,line width=.75pt,line join=round,line cap=round]
    plot coordinates {
      (9,0) (9.03,.27) (9.05,.3) (9.07,.27) (9.1,.38) (9.13,.4) (9.15,.44) (9.18,.22)
      (9.2,.31) (9.23,.14) (9.25,.22) (9.28,.37) (9.3,.41) (9.32,.58) (9.35,.72) (9.38,.49)
      (9.4,.39) (9.43,.33) (9.45,.59) (9.48,.7) (9.5,.68) (9.53,.75) (9.55,.48) (9.57,.56)
      (9.6,.5) (9.63,.49) (9.65,.76) (9.68,.73) (9.7,.69) (9.73,.58) (9.75,.79) (9.78,.78)
      (9.8,1.08) (9.82,1.08) (9.85,1.13) (9.88,1.04) (9.9,1.08) (9.93,.85) (9.95,1) (9.98,1.14)
      (10,1.18) (10.03,1.02) (10.05,1.05) (10.07,.85) (10.1,.67) (10.13,.36) (10.15,.34) (10.18,.47)
      (10.2,.26) (10.23,.38) (10.25,.22) (10.28,.31) (10.3,.56) (10.33,.37) (10.35,.34) (10.38,.24)
      (10.4,.3) (10.43,.32) (10.45,.22) (10.48,.3) (10.5,.41) (10.53,.43) (10.55,.25) (10.58,.44)
      (10.6,.7) (10.63,.71) (10.65,.72) (10.68,.78) (10.7,.72) (10.73,.47) (10.75,.3) (10.78,.5)
      (10.8,.86) (10.83,.93) (10.85,.85) (10.88,.78) (10.9,.95) (10.93,.94) (10.95,.92) (10.98,1.29)
      (11,1.2)
    };
  \draw[black,line width=.75pt,line join=round,line cap=round]
    plot coordinates {
      (11,2.7) (11.03,2.59) (11.05,2.47) (11.08,2.38) (11.1,2.32) (11.13,2.25) (11.15,2.13) (11.18,2.04)
      (11.2,2) (11.23,1.96) (11.25,2.25) (11.28,2.13) (11.3,2.1) (11.33,2.1) (11.35,2.01) (11.38,2.16)
      (11.4,2.01) (11.43,1.61) (11.45,1.69) (11.48,1.79) (11.5,1.83) (11.53,1.73) (11.55,1.67) (11.58,1.79)
      (11.6,1.64) (11.63,1.74) (11.65,2.1) (11.68,1.9) (11.7,1.81) (11.73,1.89) (11.75,1.92) (11.78,1.83)
      (11.8,1.85) (11.83,1.93) (11.85,1.87) (11.88,1.99) (11.9,2.2) (11.93,2.15) (11.95,2.18) (11.98,2.23)
      (12,1.8) (12.03,1.87) (12.05,1.96) (12.08,1.96) (12.1,1.96) (12.13,1.65) (12.15,1.59) (12.18,1.91)
      (12.2,1.79) (12.23,1.83) (12.25,1.89) (12.28,1.8) (12.3,1.38) (12.33,1.36) (12.35,1.12) (12.38,.89)
      (12.4,1.1) (12.43,1.23) (12.45,1.18) (12.48,1.28) (12.5,1.06) (12.53,.93) (12.55,.83) (12.58,.97)
      (12.6,.65) (12.63,.66) (12.65,.52) (12.68,.74) (12.7,.79) (12.73,1.13) (12.75,1.22) (12.78,1.09)
      (12.8,.93) (12.83,1.11) (12.85,1.14) (12.88,1.26) (12.9,1.46) (12.93,1.44) (12.95,1.32) (12.98,1.36)
      (13,1.35) (13.03,1.04) (13.05,1.24) (13.08,1.05) (13.1,1.06) (13.13,1.05) (13.15,1.41) (13.18,1.38)
      (13.2,1.36) (13.23,1.72) (13.25,1.89) (13.28,1.79) (13.3,1.91) (13.33,2.05) (13.35,2.3) (13.38,2.23)
      (13.4,2.07) (13.43,2.08) (13.45,2.04) (13.48,2.06) (13.5,2.05) (13.53,1.98) (13.55,1.82) (13.58,1.73)
      (13.6,1.81) (13.63,1.99) (13.65,2.12) (13.68,2) (13.7,1.95) (13.73,1.77) (13.75,1.77) (13.78,1.49)
      (13.8,1.66) (13.83,1.39) (13.85,1.59) (13.88,1.87) (13.9,1.91) (13.93,1.7) (13.95,1.44) (13.98,1.36)
      (14,1.17) (14.03,1.49) (14.05,1.55) (14.08,1.62) (14.1,1.44) (14.13,1.52) (14.15,1.55) (14.18,1.68)
      (14.2,1.54) (14.23,1.37) (14.25,1.25) (14.28,1.04) (14.3,.94) (14.33,.93) (14.35,.79) (14.38,.87)
      (14.4,1.05) (14.43,1.16) (14.45,1.3) (14.48,1.29) (14.5,1.28) (14.53,1.2) (14.55,1.15) (14.58,1.2)
      (14.6,1.25) (14.63,1.23) (14.65,1.53) (14.68,1.44) (14.7,1.27) (14.73,1.35) (14.75,1.44) (14.78,1.5)
      (14.8,1.52) (14.83,1.67) (14.85,1.83) (14.88,1.86) (14.9,1.97) (14.93,1.94) (14.95,1.93) (14.98,1.95)
      (15,2.01) (15.02,2.18) (15.05,2.32) (15.08,2.35) (15.1,2.1) (15.13,2.29) (15.15,2.49) (15.18,2.33)
      (15.2,2.03) (15.23,2.07) (15.25,2.27) (15.28,2.14) (15.3,2.08) (15.33,2.25) (15.35,1.87) (15.38,1.98)
      (15.4,1.7) (15.43,1.74) (15.45,1.9) (15.48,1.99) (15.5,1.97) (15.53,2.23) (15.55,2.27) (15.58,2.25)
      (15.6,2.26) (15.63,2.14) (15.65,2.09) (15.68,2.15) (15.7,1.89) (15.73,1.84) (15.75,2.07) (15.78,2.12)
      (15.8,2.08) (15.83,2.31) (15.85,2.25) (15.88,2.45) (15.9,2.54) (15.93,2.53) (15.95,2.58) (15.98,2.47)
      (16,2.62) (16.02,2.54) (16.05,2.79) (16.08,2.79) (16.1,2.86) (16.13,2.64) (16.15,2.56) (16.18,2.45)
      (16.2,2.44) (16.23,2.15) (16.25,2.34) (16.27,2.1) (16.3,2.17) (16.33,2.03) (16.35,1.83) (16.38,1.91)
      (16.4,1.96) (16.43,1.7) (16.45,1.52) (16.48,1.3) (16.5,1.37) (16.52,1.41) (16.55,1.33) (16.58,1.41)
      (16.6,1.5) (16.63,1.7) (16.65,1.6) (16.68,2.12) (16.7,1.96) (16.73,1.96) (16.75,1.82) (16.77,1.62)
      (16.8,1.34) (16.83,1.38) (16.85,1.13) (16.88,.72) (16.9,.88) (16.93,.72) (16.95,.92) (16.98,.49)
      (17,.52) (17.02,.76) (17.05,1.05) (17.08,1.37) (17.1,1.23) (17.13,1.16) (17.15,1.21) (17.18,1.34)
      (17.2,1.61) (17.23,1.38) (17.25,1.27) (17.27,1.11) (17.3,.99) (17.33,.69) (17.35,.52) (17.38,.91)
      (17.4,.94) (17.43,.8) (17.45,.82) (17.48,.63) (17.5,.77) (17.52,.65) (17.55,.5) (17.58,.31)
      (17.6,.47) (17.63,.62) (17.65,.64) (17.68,.55) (17.7,.52) (17.73,.43) (17.75,.51) (17.77,.61)
      (17.8,.74) (17.83,.66) (17.85,.59) (17.88,.28) (17.9,.36) (17.93,.47) (17.95,.52) (17.98,.28)
      (18,0)
    };

  \draw[blue,very thick,dashed] (3,3.3)--(3,1.7);
  \draw[blue,very thick,dashed] (5,2.1)--(5,3.7);
  \draw[violet,very thick,dashed] (7,2.8)--(7,1.3);
  \draw[violet,very thick,dashed] (11,1.2)--(11,2.7);
  \node[blue,anchor=east,fill=white,inner sep=.8pt] at (2.92,2.5) {$-h_1$};
  \node[blue,anchor=west,fill=white,inner sep=.8pt] at (5.08,4.22) {$+h_1$};
  \node[violet,anchor=east,fill=white,inner sep=.8pt] at (6.92,2.05) {$-h_2$};
  \node[violet,anchor=west,fill=white,inner sep=.8pt] at (11.08,3.05) {$+h_2$};
  \node[cyan!50!black,fill=white,inner sep=1pt] at (14.1,.85)
    {$\sum\operatorname{Area}$};
\end{tikzpicture}
\caption{A simulated Brownian decorated configuration for $m=2$,
corresponding to the lattice configuration in
Figure~\ref{fig:decorated-configuration-m2}.  The black pieces are
independent nonnegative Brownian bridges conditioned on their jump data, and
the shaded region represents their total area.  The blue dashed jumps form a
same-excursion pair, while the violet dashed jumps form a cross-excursion
pair; within each pair the negative and positive jumps have equal height.}
\label{fig:intro-brownian-decorated-configuration}
\end{figure}
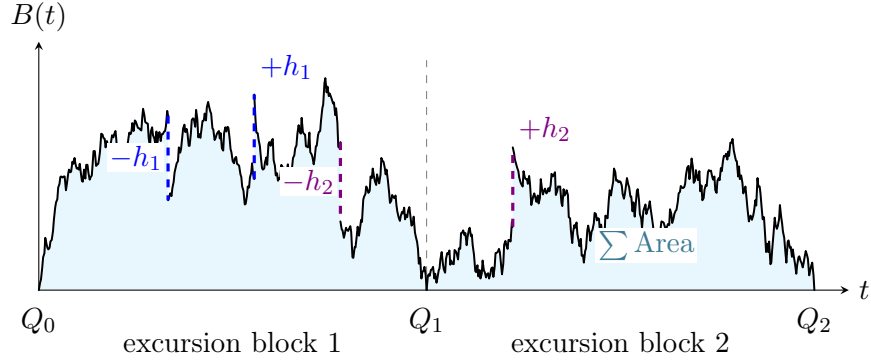

The absence of a classical matrix model is often viewed as a principal
obstacle at general $\beta$.  Here, however, treating $\beta$ as a genuine
parameter is what reveals the organization of the formula.  The factors $b$
and $1+b$ separate two topological mechanisms that are difficult to see from
a calculation performed directly in one concrete matrix model.  Specializing
the general formula gives new representations even at $\beta=1,2$.  At these
classical values, the joint Laplace moments admit Pfaffian and determinantal
descriptions, respectively; expanding those formulas produces signed linear
combinations, whereas the Brownian representation above is termwise
nonnegative.  Thus the classical values appear as transparent
specializations of a common symmetric-function construction.

This coefficient structure is intrinsic to the formula and has two
manifestations, one combinatorial and one geometric.  On the combinatorial
side, it gives a soft-edge realization of La Croix's handle basis and an
explicit formula for its asymptotic map coefficients.  On the geometric
side, intersection numbers on moduli spaces have broad connections with
two-dimensional quantum gravity and Weil--Petersson volumes, and their
large-genus asymptotics have also attracted sustained attention
\cite{Witten,Kontsevich,MirzakhaniWP,AggarwalLargeGenus,GuoYangZagier}.
At $\beta=2$, combining our formula with Okounkov's Airy representation
\cite{OkounkovIntersection} yields a finite positive Brownian multigraph
formula for the Witten--Kontsevich intersection correlators.  The precise statements are given in
Section~\ref{subsec:intro-applications}.

Our expansion is organized around $b=0$, equivalently $\beta=2$.  This
should be contrasted with the more commonly used Brownian organization
centered at $1/\beta=0$, which is adapted to the zero-temperature limit where
the process crystallizes at the zeros of the Airy function
\cite{RRV,GorinShkolnikovAiry,GXZ}.  The two descriptions expose different
structures: the latter is naturally analytic and spectral, whereas the
paired-jump expansion is combinatorial and topological.  This distinction is
what makes positivity, connectedness, and coefficient extraction transparent
in the present formula.

Our proof combines tools from symmetric function theory, estimates for conditioned random walks, and existing soft-edge results for discrete beta ensembles.  Nazarov--Sklyanin operators~\cite{NS} were previously used by Moll~\cite{MollThesis,Moll}, Huang~\cite{Huang}, and Cuenca--Do{\l}\k{e}ga--Moll~\cite{CDM} to study global limits of certain Jack measures.  The present paper is the first to use them for an edge-limit calculation: they convert its
moments into decorated Lukasiewicz paths; after diffusive scaling, the paths
become nonnegative Brownian bridges and the decorations become paired jumps
and area weights.  This moment calculation and its Brownian limit hold for
every $\beta>0$.  The rigidity and edge-universality theorem of
Guionnet--Huang~\cite{GH} is then used in the final step to identify the
limiting series with the Airy$_\beta$ point process for $\beta\geq1$.  

\subsection{The Brownian functional and the main result}
\label{subsec:intro-main-result}

We first define the functional appearing in the answer.  Put
\begin{equation}
  z_c=2-\sqrt2,\qquad
  \mu_+=2+2\sqrt2,\qquad
  P_{-1}=\frac1{\sqrt2},\qquad
  \sigma^2=2+\sqrt2.
  \label{eq:intro-functional-constants}
\end{equation}
These constants are not additional parameters: $z_c$ is the critical
exponential tilt of the canonical walk introduced in Section~\ref{subsec:decorated-path-expansion}, $\mu_+=V(z_c)$ is the right
macroscopic edge, and $P_{-1}$ and $\sigma^2$ are respectively the
probability of a $-1$ step and the variance under that centered tilt.
Let $b\in(-1,\infty)$.  Fix $m\geq1$ and
$\boldsymbol k=(k_1,\ldots,k_m)\in\mathbb R_{>0}^m$.  Write
\[
  Q_p=\sum_{q=1}^p k_q,\qquad Q_0=0.
\]
The interval $[Q_{p-1},Q_p]$ will carry the $p^{\mathrm{th}}$ excursion.

For $r\geq0$, let
\begin{equation*}
  \mathfrak A_{m,r}
  =\left\{\boldsymbol\ell=
  ((\ell_-^j,\ell_+^j))_{j=1}^r:
  1\leq \ell_-^j\leq\ell_+^j\leq m\right\}.
\end{equation*}
For a fixed assignment $\boldsymbol\ell$, introduce, for every
$j=1,\ldots,r$, two jump times and a common jump height,
\[
  t_-^j<t_+^j,\qquad h^j>0,\qquad
  Q_{\ell_\varepsilon^j-1}<t_\varepsilon^j<Q_{\ell_\varepsilon^j},
  \quad \varepsilon\in\{-,+\},
\]
together with $H_-^j,H_+^j>0$. We require all jump times belonging to the same excursion interval
$(Q_{p-1},Q_p)$ to be distinct. The domain of all these variables is
denoted by $\mathcal D_{\boldsymbol k}(\boldsymbol\ell)$.  The jump-pair
measure is
\begin{equation}
  \mathrm d\boldsymbol\nu_r
  =\prod_{j=1}^r
  \mathrm dt_-^j\,\mathrm dt_+^j\,
  \frac{(1+b)\sigma^2}{\mu_+^2}h^j\,\mathrm dh^j\,
  \mathrm dH_-^j\,\mathrm dH_+^j.
  \label{eq:intro-jump-measure}
\end{equation}

To define the integrand, order the jump times on the $p^{\mathrm{th}}$ excursion as
\[
  0=s_{p,0}<s_{p,1}<\cdots<s_{p,n_p}<s_{p,n_p+1}=k_p.
\]
If the $a^{\mathrm{th}}$ event is the negative jump of pair $j$, set
\[
  y_{p,a}=H_-^j+h^j,\qquad x_{p,a}=H_-^j;
\]
if it is the positive jump, set
\[
  y_{p,a}=H_+^j,\qquad x_{p,a}=H_+^j+h^j.
\]
Thus $y_{p,a}$ and $x_{p,a}$ are respectively the left and right limits at
the jump time.  We also set $x_{p,0}=y_{p,n_p+1}=0$.  Between consecutive
jumps we use Brownian bridges from $x_{p,a-1}$ to $y_{p,a}$ conditioned to
stay nonnegative.

The corresponding killed-Brownian kernels are
\begin{align*}
  \mathbf F(t;u,v)
  &=\frac1{\sqrt{2\pi t}}
    \left(e^{-(u-v)^2/(2t)}-e^{-(u+v)^2/(2t)}\right),
    \\
  \mathbf F_0(t;u)
  &=\frac{2u}{\sqrt{2\pi t^3}}e^{-u^2/(2t)},\qquad
  \mathbf F_{0,0}(t)=\frac2{\sqrt{2\pi t^3}}.
\end{align*}
We use the convention that \(\mathbf F(t;u,v)\) is the 
transition density of a Brownian motion from \(v\) to \(u\) at time $t$, killed upon hitting zero. Then the transition-density factor of excursion $p$ is
\begin{equation}
  \mathbf D_p=
  \begin{cases}
    \mathbf F_{0,0}(k_p),&n_p=0,\\[1mm]
    \mathbf F_0(s_{p,1};y_{p,1})
    \displaystyle\prod_{a=2}^{n_p}
      \mathbf F(s_{p,a}-s_{p,a-1};y_{p,a},x_{p,a-1})
    \mathbf F_0(k_p-s_{p,n_p};x_{p,n_p}),&n_p\geq1.
  \end{cases}
  \label{eq:intro-transition-factor}
\end{equation}
Conditionally on the boundary data, let $B_{p,a}$ be the corresponding
nonnegative bridge on $[s_{p,a-1},s_{p,a}]$.  Define its area factor by
\begin{equation}
  \mathbf A_{b,p}
  =\prod_{a=1}^{n_p+1}
  \mathbb E\left[
    \exp\left(\frac{b\sigma}{\mu_+}\int_{s_{p,a-1}}^{s_{p,a}}B_{p,a}(u)\,\mathrm du\right)
  \right].
  \label{eq:intro-area-factor}
\end{equation}
All expectations here are under the conditioned bridge laws just
specified.

Let
\[\tau_1\geq\tau_2\geq\cdots\geq\tau_m
  \]
be the time parameters in the multi-time formula.
Let
\begin{equation*}
  \tilde{c}
  :=\frac{\sigma}{4(1+\sqrt2)}.
\end{equation*}
For $\boldsymbol\ell\in\mathfrak A_{m,r}$, define the time factor
\begin{equation}
  \mathbf T_{\boldsymbol\tau}(\boldsymbol\ell,\boldsymbol h)
  :=\prod_{j=1}^r
  \exp\left\{\tilde{c}
    (\tau_{\ell_+^j}-\tau_{\ell_-^j})h^j\right\}.
  \label{eq:intro-multitime-factor}
\end{equation}
Thus $0<\mathbf T_{\boldsymbol\tau}\leq1$, and a same-excursion pair has
time factor one.  We now set
\begin{align}
  \mathbf L_{b,r}^{\boldsymbol\ell}
    [\boldsymbol k;\boldsymbol\tau]
  &:=\frac1{r!}(2P_{-1}\sigma)^{-m}
  \int_{\mathcal D_{\boldsymbol k}(\boldsymbol\ell)}
  \mathbf T_{\boldsymbol\tau}(\boldsymbol\ell,\boldsymbol h)
  \prod_{p=1}^m\mathbf D_p\mathbf A_{b,p}\,
  \mathrm d\boldsymbol\nu_r,
  \label{eq:intro-multitime-fixed-functional}\\
  \mathbf L_{b,r}[\boldsymbol k;\boldsymbol\tau]
  &:=\sum_{\boldsymbol\ell\in\mathfrak A_{m,r}}
  \mathbf L_{b,r}^{\boldsymbol\ell}
    [\boldsymbol k;\boldsymbol\tau],
  &
  \mathbf L_b[\boldsymbol k;\boldsymbol\tau]
  &:=\sum_{r=0}^{\infty}
  \mathbf L_{b,r}[\boldsymbol k;\boldsymbol\tau].\notag
\end{align}
When $\tau_1=\cdots=\tau_m$, the factor
\eqref{eq:intro-multitime-factor} is identically one and we denote
\begin{equation}
  \mathbf L_b[\boldsymbol k;\tau_1=\cdots=\tau_m]
  =\mathbf L_b[\boldsymbol k].
  \label{eq:intro-fixed-functional}
\end{equation}

The factor $1/r!$ removes the labeling of the jump pairs.  Each assignment $\boldsymbol\ell$ is uniquely encoded by the edge-labeled
multigraph $G(\boldsymbol\ell)$ with vertex set $\{1,\ldots,m\}$, whose
edge labeled $j$ is $\{\ell_-^j,\ell_+^j\}$.  Consequently,
$\mathbf L_b[\boldsymbol k]$ may be viewed as a weighted sum over
multigraphs.  We develop this viewpoint further in
Sections~\ref{sec:beta-topological} and~\ref{sec:intersection-numbers}.

Let $(\lambda^{(N)}(s))_{s\geq0}$ be the Jack--Plancherel process; see Section~\ref{subsec:Jack-Plancherel-process} for its definition. Instead of the usual empirical measure, we take our observables to be the Perelomov--Popov measure. More precisely, for
$\lambda\in\mathbb Y(N)$, set
\[
  y_i=y_i(\lambda)
  :=\frac{\lambda_i-\theta(i-1)}{\theta N},
  \qquad i=1,\ldots,N,
\]
and define
\begin{equation*}
  \mu_{\mathrm{PP}}[\lambda]
  :=\frac1N\sum_{i=1}^N
    \prod_{j\neq i}
    \frac{y_i-y_j+N^{-1}}{y_i-y_j}\,
    \delta_{y_i}.
\end{equation*}
This is a probability measure introduced by Bufetov--Gorin~\cite{BG15}. The nonuniform weighting is based on formulas of
Perelomov--Popov~\cite{PP} and has a
representation-theoretic origin, and was applied in the study of random tiling and $\beta$-nonintersecting random walks in~\cite{BG15,Huang,HuangPoisson}. 

We consider the soft-edge regime, in which both the time displacements and
the moment degrees are of order $N^{2/3}$. More precisely, for $\tau_1\ge\cdots\ge\tau_m$, we set
\begin{equation*}
  s_{p,N}:=2N+\tau_pN^{2/3},\qquad
  a_{p,N}:=\frac{s_{p,N}}N,\qquad
  \mu_{p,N}:=a_{p,N}+2\sqrt{a_{p,N}},
\end{equation*}
and take
\[
  K_p=\lfloor k_pN^{2/3}\rfloor.
\]
Under this scaling, we obtain the following joint moment convergence for
the Jack--Plancherel process.

\begin{theorem}[Multi-time  moment convergence]
\label{thm:intro-Jack-Plancherel-process-moments}
Fix $\theta>0$, $m\geq1$, $\boldsymbol k\in\mathbb R_{>0}^m$, and
$\tau_1\geq\cdots\geq\tau_m$.  With $b=\theta^{-1}-1$, the series
$\mathbf L_b[\boldsymbol k;\boldsymbol\tau]$ converges absolutely, and
\begin{align*}
 \lim_{N\to\infty}\mathbb E\left[
   \prod_{p=1}^m
   N\int
   \left(\frac{x}{\mu_{p,N}}\right)^{K_p}
   \mathrm d\mu_{\mathrm{PP}}[\lambda^{(N)}(s_{p,N})](x)
   \right]
   \notag=(1+\mu_+)^m
   \mathbf L_b[\boldsymbol k;\boldsymbol\tau].
\end{align*}
The convergence is locally uniform in the ordered time vector
$\boldsymbol\tau$.
\end{theorem}

We prove Theorem \ref{thm:intro-Jack-Plancherel-process-moments} in Section \ref{sec:combinatorial-asymptotics}. To the best of our knowledge,  such multi-time soft-edge asymptotics have not previously been established for Jack processes\footnote{This should be distinguished from the edge dynamics studied by Dimitrov--Lolas \cite{DimitrovLolas}, which concerns microscopic gaps between the rightmost particles on adjacent levels of a multilevel Jack process and leads to a zero-range interacting particle system.}. Nonetheless, the approach developed in this paper is not restricted to the Jack–Plancherel process. The probabilistic estimates in Section \ref{subsec:conditional-walk-bridges} are established for general integer-valued random walks, and the asymptotic analysis in Section \ref{sec:combinatorial-asymptotics} extends to a broad class of Jack processes with general specializations. We formulate the corresponding multi-time moment convergence theorem in Section \ref{sec:general-Jack-processes}, while retaining the Plancherel specialization in the main calculations for clarity of exposition. We believe that the limiting process is the Airy$_\beta$ line ensemble, and in this text we do the identification for the single-time, see Section \ref{sec:airy-identification}. The identification provides simultaneously a new explicit formula for the Airy$_\beta$ process as follows.

Let $\mathcal A_1^{(\beta)}>\mathcal A_2^{(\beta)}>\cdots$ denote the
Airy$_\beta$ point process and
\[
  \mathcal Z_\beta(s)=\sum_{i\geq1}e^{s\mathcal A_i^{(\beta)}},
  \qquad s>0.
\]
Finally, put
\begin{equation}
  \mathfrak c_2=2^{-1/6}(1+\sqrt2)^{-2/3},
  \qquad C_2=z_c(1+\mu_+)=2+\sqrt2.
  \label{eq:intro-edge-conversion-constants}
\end{equation}
\begin{theorem}[Airy$_\beta$ Laplace transform]
\label{thm:main-Airy-Laplace}
For every $\beta\ge 1$, $m\geq1$, and $k_1,\ldots,k_m>0$, set
$b=2/\beta-1$. We have
\begin{equation}
  \mathbb E\left[
    \prod_{p=1}^m
    \mathcal Z_\beta\left(\frac{k_p}{\mu_+\mathfrak c_2}\right)
  \right]
  =C_2^m\mathbf L_b[\boldsymbol k].
  \label{eq:intro-main-Laplace-formula}
\end{equation}
Equivalently, for $s_1,\ldots,s_m>0$,
\begin{equation*}
  \mathbb E\left[\prod_{p=1}^m\mathcal Z_\beta(s_p)\right]
  =C_2^m\mathbf L_b[\mu_+\mathfrak c_2\boldsymbol s].
\end{equation*}
\end{theorem}

The identification of the one-time limiting process uses the comparison of the soft-edge limits for the discrete and continuous
beta ensembles established in~\cite{GH}.  While the technical condition $\beta\ge 1$ is essential in \cite{GH}, and the dynamical edge comparison between discrete and continuous ensembles is not yet established, we expect that for all $\beta>0$, our formula $C_2^m\mathbf L_b[\mu_+\mathfrak c_2\boldsymbol s;\boldsymbol\tau]$ gives the second multi-time Laplace transform formula for the Airy$_\beta$ line ensemble after Gorin-Xu-Zhang. 

The zero-temperature boundary\footnote{The uniform estimate in
Section~\ref{subsec:uniform-summability} justifies passage to the limit
$b\downarrow-1$.} $b=-1$ provides a useful consistency
check.  Let $0>a_1>a_2>\cdots$ be the zeros of the Airy function.
We write $\mathcal Z_\infty(s):=\sum_{j\geq1}e^{sa_j}$ for $s>0$.
Every jump-pair measure in~\eqref{eq:intro-jump-measure} contains the
factor $1+b$, so for $m=1$ only the $r=0$ term remains.  If $e$ is a
standard Brownian excursion on $[0,1]$, the classical excursion-area
identity~\cite[Section~13, (80)]{JansonBrownianAreas} states that
\[
  \mathbb E\!\left[
    e^{-u\int_0^1e(t)\,\mathrm dt}
  \right]
  =\sqrt{2\pi}\,u
    \sum_{j\geq1}
    \exp\!\left\{-2^{-1/3}|a_j|u^{2/3}\right\}.
\]
Brownian scaling, together with
$P_{-1}\mu_+=C_2$ and
$2^{-1/3}(\sigma/\mu_+)^{2/3}=(\mu_+\mathfrak c_2)^{-1}$, therefore
gives
\[
  C_2\mathbf L_{-1}[k]
  =\sum_{j\geq1}
    \exp\!\left\{\frac{k a_j}{\mu_+\mathfrak c_2}\right\}
  =\mathcal Z_\infty\!\left(\frac{k}{\mu_+\mathfrak c_2}\right).
\]
Thus the boundary value of our formula recovers the deterministic Airy
spectrum, as expected from the stochastic Airy operator when
$\beta\to\infty$.

At $\beta=2$ one has $b=0$, and all area factors in
\eqref{eq:intro-area-factor} equal one.  The formula then contains only
killed-Brownian transition densities and paired jumps.  This is the
probabilistic shadow of the determinantal structure at $\beta=2$.

One can compare $\mathbf L_b[\boldsymbol k;\boldsymbol\tau]$ with the formula of
\cite{GXZ}.  Both descriptions use Brownian bridges with
local decorations, but their organizations are different.  First, the
expansion in~\cite{GXZ} is naturally organized in powers of $1/\beta$,
whereas ours is organized around $b=0$ and makes the determinantal
simplification visible.  Second, a block in~\cite{GXZ} may contain
arbitrarily many upward and downward jumps; every block here consists of
exactly two jumps, of opposite signs and equal height.  Third, each of our
excursions returns to zero at $Q_p$, while the construction in~\cite{GXZ}
also requires bridges with positive boundary heights.  Finally, because of
the latter two simplifications, every fixed term in our functional is
nonnegative for every $b>-1$, while the formula in~\cite{GXZ} involves extensive
cancellation between positive and negative terms.

\subsection{Combinatorial and geometric manifestations}
\label{subsec:intro-applications}

The Airy identity in Theorem~\ref{thm:main-Airy-Laplace} and the two
results below are three manifestations of the same Brownian
expansion.  On the combinatorial side, let $G(\boldsymbol\ell)$ be a connected assignment multigraph as above, with
$m$ vertices, $r$ edges, and cycle rank
\[
  i=b_1\bigl(G(\boldsymbol\ell)\bigr)=r-m+1.
\]
If $\rho$ is the total degree selected from the area exponentials, the
corresponding Brownian term carries $b^\rho(1+b)^r$.  After normalizing by $(1+b)^{1-m}$, this becomes
\[
  b^{2g-2i}(1+b)^i,\qquad 2g=2i+\rho.
\]
The polynomial $b^{2g-2i}(1+b)^i$ is the corresponding element of the
handle basis for the marginal map series in \cite{LaCroix}: its coefficient is $2^{-i}$ times
the number of rooted cellular maps of genus $g$ whose root-edge deletion
encounters exactly $i$ handles.  Here $g$ is
defined by $\chi=2-2g$ and may be half-integral in the non-orientable case.  Thus the cycle rank records the handle contribution,
while the area degree records the remaining non-orientability.  In Theorem~\ref{thm:asymptotic-LaCroix-coefficients}, a coefficient
extraction therefore leaves only finitely many Brownian integrals, and identifies their sum with the
asymptotic number of maps with genus $g$ and $i$ handles. This gives a nonnegative Brownian integral formula of the latter.

On the geometric side, take $\beta=2$ so that $b=0$.  All area decorations disappear and only the
orientable terms remain.  Passing from moments to cumulants removes the
disconnected assignments: the alternating cancellations in the
moment--cumulant formula are absorbed into the single requirement that the
assignment multigraph be connected.  Its cycle rank is then exactly the
genus.  With $\mathfrak I_{\boldsymbol\ell}(\boldsymbol x)$ denoting the
corresponding unnormalized nonnegative Brownian integral,
Theorem~\ref{thm:intersection-Brownian-multigraph} gives the genus-$g$
Witten--Kontsevich $m$-point correlator in the form
\begin{equation}
  F_{g,m}(\boldsymbol x)
  =
  \frac{(\pi/2)^{m/2}}
  {(g+m-1)!\sqrt{x_1\cdots x_m}}
  \sum_{\boldsymbol\ell\in
  \mathfrak A_{m,g+m-1}^{\mathrm{conn}}}
  \mathfrak I_{\boldsymbol\ell}(\boldsymbol x).
  \label{eq:intro-positive-intersection-formula}
\end{equation}
Equivalently, the sum may be organized by ordinary connected multigraphs on
$[m]$, with loops and multiple edges allowed and with no ribbon structure
retained.  Formula~\eqref{eq:intro-positive-intersection-formula} therefore
reorganizes the alternating determinantal expression in Okounkov's cumulant
as a finite sum of nonnegative Brownian weights, with geometric genus
represented by ordinary graph cycle rank.

\subsection{Further directions}
While Theorem \ref{thm:main-Airy-Laplace} identifies the one-time limit with the Airy$_\beta$ process, it remains to identify the multi-time limit with the Airy$_\beta$ line ensemble. The present framework gives a concrete route for it, which we briefly sketch as follows. It is well known that for fixed $N\le L$, the Jack measure with specialization $\rho_{N,L,\delta}:\alpha_1=\ldots=\alpha_L=\delta<1$, $\alpha_{L+1}=\alpha_{L+2}=\cdots=0$ converges weakly to the Laguerre $\beta$ ensemble under the scaling $\delta=\exp(-\epsilon)$, $y_i=\theta\epsilon \lambda_i$, $\epsilon\to0$, see e.g. \cite[Remark 1.10]{BorodinOlshanskiScaling}. Indeed, the same convergence holds for the Jack process indexed by $\rho_{N,1,\delta}$ in finite-time distribution. The continuum limit is the \emph{Laguerre $\beta$-corners process}, a discrete time Markov process defined in \cite[Section 2.1]{LernerBrecher2025}. Take a sequence of $\rho_{N,1,\delta}$ parameterized by $(\epsilon_N,N)$ for $\epsilon_N$ small enough, so that the first $k$ particles of the corresponding Jack process are close to those of Laguerre corners process under edge scaling. On the one hand, it is plausible to prove the moment convergence under edge scaling of this sequence to the same functional, using the similar approach in Sections \ref{sec:preliminaries} and \ref{sec:combinatorial-asymptotics} (when $\epsilon_N$ is fixed this is covered by Theorem \ref{thm:general-Jack-moment-convergence}). On the other hand, combining the proofs of \cite{KX} and \cite{GXZ}, one can directly identify the edge limit of the Laguerre corners process with Airy$_\beta$ line ensemble.

We are also working on a direct description of the Airy$_\beta$ correlation
functions based on the same expansion.  A subsequent bulk limit may lead to
analogous formulas for sine$_\beta$ correlations.

Finally, the positive multigraph formula may offer a probabilistic perspective
on the large-genus asymptotics of intersection numbers.  Since
\eqref{eq:intro-positive-intersection-formula} is an explicit positive sum
over connected multigraphs, it may allow one to identify directly which graph
geometries dominate as the genus tends to infinity, complementing the
existing large-genus analyses
\cite{AggarwalLargeGenus,GuoYangZagier}.  A suitable decomposition of these
multigraphs may also provide a probabilistic interpretation of the classical
topological recursion.

The paper is organized as follows.  Section~\ref{sec:preliminaries}
introduces the Jack--Plancherel process, Nazarov--Sklyanin operators, and the conditioned
random-walk estimates needed in Section~\ref{sec:combinatorial-asymptotics}.  Section~\ref{sec:combinatorial-asymptotics} derives
the multi-time Brownian moment expansion. Section~\ref{sec:general-Jack-processes} extends the moment convergence to a general class of Jack processes. Section~\ref{sec:airy-identification} removes the
Perelomov--Popov deformation and identifies the Airy edge limit.
Sections~\ref{sec:beta-topological} and
\ref{sec:intersection-numbers} develop respectively the asymptotic
topological expansion and the positive Brownian formula for intersection
correlators.
\subsection*{Acknowledgements} I am grateful to Maciej Do\l{}\k{e}ga for kindly explaining his work and the related background, which sparked my interest in topological expansion.
 I thank IPAM for hosting the program ``Geometry, Statistical Mechanics, and Integrability'' in spring 2024, where this discussion took place.
 I thank Vadim Gorin for helpful comments, and his support throughout the past seven years. I dedicate this work to Stockholm.

\subsection*{Statement of AI}
All the math ingredients in this text are given by the human author. During the preparation of the text, we used ChatGPT 5.6 for literature search, error detection, improving the exposition and drawing the figures. The human author checked the whole text line by line, and takes full responsibility for the contents.

\section{Preliminaries}
\label{sec:preliminaries}

Throughout the paper, $\theta>0$ is the Jack parameter, and
\begin{equation*}
  \beta=2\theta,
  \qquad
  b=\theta^{-1}-1=\frac2\beta-1.
\end{equation*}
A partition is a nonincreasing sequence
$\lambda=(\lambda_1,\lambda_2,\ldots)$ of nonnegative integers with
finite size $|\lambda|=\sum_i\lambda_i$; its length
$\ell(\lambda)$ is the number of nonzero parts.  We work in the algebra
of symmetric functions
\[
  \Lambda=\mathbb Q(\theta)[p_1,p_2,\ldots],
\]
where the power sums $p_d$, $d\geq1$, are algebraically independent
generators.  Thus, unless a specialization is explicitly made, the
$p_d$ are formal variables rather than the power sums of a fixed
collection $x_1,\ldots,x_N$.  If $m_d(\lambda)$ is the multiplicity
of the part $d$, we write
\[
  p_\lambda=\prod_i p_{\lambda_i},
  \qquad
  z_\lambda=\prod_{d\geq1}d^{m_d(\lambda)}m_d(\lambda)!.
\]

For partitions of the same size, write $\mu<\lambda$ for strict
dominance order.  The scalar product on $\Lambda$ is determined by
\begin{equation}
  \langle p_\lambda,p_\mu\rangle_\theta
  =\mathbf 1_{\{\lambda=\mu\}}z_\lambda\theta^{-\ell(\lambda)}.
  \label{eq:Jack-scalar-product}
\end{equation}
The monic Jack symmetric functions
$J_\lambda(\,\cdot\,;\theta)$ are uniquely characterized by
\[
  J_\lambda
  =m_\lambda+\sum_{\mu<\lambda}u_{\lambda\mu}(\theta)m_\mu,
  \qquad
  \langle J_\lambda,J_\mu\rangle_\theta=0
  \quad(\lambda\neq\mu);
\]
in particular, the coefficient of $m_\lambda$ is one.  For a box
$\square=(i,j)\in\lambda$, let $a(\square),l(\square)$ and
$a'(\square),l'(\square)$ denote its arm, leg, coarm, and coleg,
respectively.  In this normalization the squared norm is
\begin{equation}
  j_\lambda(\theta)
  :=\langle J_\lambda,J_\lambda\rangle_\theta
  =\prod_{\square\in\lambda}
    \frac{a(\square)+\theta l(\square)+1}
         {a(\square)+\theta l(\square)+\theta}.
  \label{eq:Jack-squared-norm}
\end{equation}
These conventions agree with the monic normalization in
\cite[Chapter~VI, Section~10]{Macdonald} after setting the usual Jack
parameter equal to $\theta^{-1}$.

A specialization $\rho$ is an algebra homomorphism
$\Lambda\to\mathbb C$ and is specified by the numbers $p_d(\rho)$.
It is called Jack-positive if
$J_{\lambda}(\rho;\theta)\geq0$ for every partition
$\lambda$.  The Jack Cauchy
identity is
\begin{equation}
  \sum_\lambda
  \frac{J_\lambda(\rho;\theta)J_\lambda(\rho';\theta)}
       {j_\lambda(\theta)}
  =\exp\left(
    \theta\sum_{d\geq1}\frac{p_d(\rho)p_d(\rho')}{d}
  \right).
  \label{eq:Jack-Cauchy-identity}
\end{equation}
Equivalently, the right-hand side is the reproducing kernel for
\eqref{eq:Jack-scalar-product}; expanding that kernel in the orthogonal
Jack basis gives the left-hand side.  The identity may be read formally,
or analytically whenever the series in the exponential converges.  We
refer to \cite{Macdonald,KOO,BC} for further background.

\subsection{The Jack--Plancherel measure}
\label{subsec:Jack-Plancherel-measure}

Let $\rho_1,\rho_2$ be Jack-positive specializations for which
\[
  \mathcal H_\theta(\rho_1,\rho_2)
  :=\exp\left(
    \theta\sum_{d\geq1}\frac{p_d(\rho_1)p_d(\rho_2)}d
  \right)
\]
is finite.  The Jack measure associated with $(\rho_1,\rho_2)$ is
\begin{equation}
  \mathbb P_{\rho_1,\rho_2}^{(\theta)}(\lambda)
  :=\frac{1}{\mathcal H_\theta(\rho_1,\rho_2)}
    \frac{J_\lambda(\rho_1;\theta)
          J_\lambda(\rho_2;\theta)}
         {j_\lambda(\theta)}.
\label{eq_jackmeasure}\end{equation}
By Jack positivity and \eqref{eq:Jack-Cauchy-identity}, it is a probability measure on partitions $\lambda$.

In this text we use the following two specializations.  Evaluation at $N$ variables
equal to one, denoted by $1^N$, is characterized by
$p_d(1^N)=N$ for every $d\geq1$.  The Plancherel specialization
$\mathrm{Pl}_s$, $s\geq0$, is characterized by
\begin{equation*}
  p_1(\mathrm{Pl}_s)=s,
  \qquad
  p_d(\mathrm{Pl}_s)=0,\quad d\geq2.
\end{equation*}
Both are Jack-positive.  The evaluation formulas
\begin{align*}
  J_\lambda(1^N;\theta)
  =\prod_{\square\in\lambda}
    \frac{N\theta+a'(\square)-\theta l'(\square)}
         {a(\square)+\theta l(\square)+\theta},\qquad
  J_\lambda(\mathrm{Pl}_s;\theta)
  =(s\theta)^{|\lambda|}
    \prod_{\square\in\lambda}
    \frac{1}{a(\square)+\theta l(\square)+\theta}
\end{align*}
are the Jack specialization formulas of
\cite[Chapter~VI, Section~10]{Macdonald}.  In particular,
$J_\lambda(1^N;\theta)=0$ if $\ell(\lambda)>N$.

\begin{definition}
For $N\geq1$ and $s\geq0$, the Jack--Plancherel measure considered in
this paper is the Jack measure corresponding to
$(1^N,\mathrm{Pl}_s)$:
\begin{align}
  \mathbb P_{N,s}^{(\theta)}(\lambda)
  &:=e^{-\theta sN}
    \frac{J_\lambda(1^N;\theta)
          J_\lambda(\mathrm{Pl}_s;\theta)}
         {j_\lambda(\theta)}
  \notag\\
  &=e^{-\theta sN}(s\theta)^{|\lambda|}
    \prod_{\square\in\lambda}
    \frac{N\theta+a'(\square)-\theta l'(\square)}
    {\bigl(a(\square)+\theta l(\square)+\theta\bigr)
     \bigl(a(\square)+\theta l(\square)+1\bigr)}.
  \label{eq:Jack-Plancherel-measure}
\end{align}
It is supported on partitions with at most $N$ rows.
\end{definition}

This is the Poissonized Jack--Plancherel measure of
\cite[Proposition~2.9]{GS}.  It is a discrete counterpart of the
Gaussian $\beta$-ensemble: under the diffusive scaling of
\cite[Proposition~2.10]{GS}, its shifted rows converge to G$\beta$E with $\beta=2\theta$.  Other objects also called
Jack--Plancherel measures are obtained by fixing the size or by using two
infinite specializations; see, for example, \cite{KOO,Fulman,Moll}.
Those measures need not have a deterministic bound on the number of rows.

For a probability measure $M_N$ on partitions of length at most $N$, its
Jack generating function is
\begin{equation}
  \mathcal F_{M_N}(\boldsymbol p;\theta)
  :=\sum_{\ell(\lambda)\leq N}M_N(\lambda)
    \frac{J_\lambda(\boldsymbol p;\theta)}
         {J_\lambda(1^N;\theta)}.
  \label{eq:Jack-generating-function-definition}
\end{equation}
As a formal element of the degree completion of $\Lambda$, this
definition requires no tail assumption.  When analytic evaluation and
termwise differentiation near $1^N$ are needed, we use the regularity
condition of \cite[Assumption~2.1]{Huang}; the Plancherel specialization
below has the required analytic exponential representative.

\begin{proposition}[Jack generating function of the Plancherel measure]
\label{prop:Jack-Plancherel-generating-function}
For $M_N=\mathbb P_{N,s}^{(\theta)}$,
\begin{equation*}
  \mathcal F_{M_N}(\boldsymbol p;\theta)
  =\exp\bigl(\theta s(p_1-N)\bigr)-R_{N,s}(\boldsymbol p;\theta),
\end{equation*}
where
\begin{equation}
  R_{N,s}(\boldsymbol p;\theta)
  :=e^{-\theta sN}
    \sum_{\ell(\lambda)>N}
    \frac{J_\lambda(\mathrm{Pl}_s;\theta)
          J_\lambda(\boldsymbol p;\theta)}
         {j_\lambda(\theta)}
  \label{eq:Jack-Plancherel-tail}
\end{equation}
is a linear combination of Jack functions indexed by partitions with
more than $N$ rows.
\end{proposition}

\begin{proof}
Substituting \eqref{eq:Jack-Plancherel-measure} into
\eqref{eq:Jack-generating-function-definition} cancels
$J_\lambda(1^N;\theta)$.  Completing the resulting sum from
$\ell(\lambda)\leq N$ to all partitions and applying
\eqref{eq:Jack-Cauchy-identity} with
$(\rho,\rho')=(\mathrm{Pl}_s,\boldsymbol p)$ gives
$e^{-\theta sN}e^{\theta sp_1}$; the omitted terms are exactly
\eqref{eq:Jack-Plancherel-tail}.
\end{proof}

The correction $R_{N,s}$ will not enter any moment calculation below.
Indeed, the Nazarov--Sklyanin operators preserve each Jack eigenspace,
whereas $J_\lambda(1^N;\theta)=0$ for $\ell(\lambda)>N$.

\subsection{Nazarov--Sklyanin operators}
\label{subsec:NS-operator}

The Nazarov--Sklyanin operator is an operator on $\Lambda$, introduced
in \cite{NS}, for which the Jack functions are eigenfunctions.  We
briefly recall its definition and the moment-extraction identity, using
the conventions of \cite[Section~3]{Huang}.  Set
\begin{equation*}
  p_{-d}:=\frac d\theta\frac{\partial}{\partial p_d},
  \qquad d\geq1,
  \qquad p_0:=0,
\end{equation*}
and introduce the semi-infinite matrix $L=(L_{ij})_{i,j\geq0}$,
\begin{equation*}
  L_{ij}=jb\,\mathbf 1_{\{i=j\}}+p_{j-i}.
\end{equation*}
Explicitly,
\begin{equation*}
  L=
  \begin{pmatrix}
    0      &p_1    &p_2    &p_3    &\cdots\\
    p_{-1} &b      &p_1    &p_2    &\cdots\\
    p_{-2} &p_{-1} &2b     &p_1    &\cdots\\
    p_{-3} &p_{-2} &p_{-1} &3b     &\cdots\\
    \vdots &\vdots &\vdots &\vdots &\ddots
  \end{pmatrix}.
\end{equation*}
The commuting Nazarov--Sklyanin operators are the coefficients of
\begin{equation*}
  I(u;\theta):=(u-L)^{-1}_{00}
  =\frac1u+\sum_{k\geq1}\frac{I^{(k)}}{u^{k+1}},
  \qquad I^{(k)}=(L^k)_{00}.
\end{equation*}
Here the inverse is a formal Laurent series at $u=\infty$; every
coefficient is a well-defined operator on a symmetric polynomial.
Equivalently,
\begin{equation*}
  I^{(k)}
  =\sum_{i_1,\ldots,i_{k-1}\geq0}
    L_{0i_1}L_{i_1i_2}\cdots
    L_{i_{k-2}i_{k-1}}L_{i_{k-1}0}.
\end{equation*}

For a partition $\lambda$ of length at most $N$, introduce the shifted
particles
\begin{equation}
  y_i=y_i(\lambda)
  :=\frac{\lambda_i-\theta(i-1)}{\theta N},
  \qquad i=1,\ldots,N.
  \label{eq:shifted-particle-coordinates}
\end{equation}
Its Perelomov--Popov measure is
\begin{equation}
  \mu_{\mathrm{PP}}[\lambda]
  :=\frac1N\sum_{i=1}^N
    \prod_{j\neq i}\frac{y_i-y_j+N^{-1}}{y_i-y_j}\,\delta_{y_i}.
  \label{eq:PP-measure}
\end{equation}
This is a positive measure.  Indeed,
$y_i-y_{i+1}\geq N^{-1}$, and every factor in each atom weight in
\eqref{eq:PP-measure} is nonnegative.

One crucial fact about Nazarov--Sklyanin operators is that they act diagonally on Jack functions, and the
eigenvalues give the moments of \eqref{eq:PP-measure}. We use the following form derived in \cite{Huang}.
\begin{theorem}\cite[Theorem~3.7]{Huang}
\label{thm:Huang-moment-extraction}
For $m\geq1$ and
$k_1,\ldots,k_m\geq1$,
\begin{align*}
  \prod_{p=1}^m
  \left(
    \frac{I^{(k_p)}}{N^{k_p}}
    +\frac{I^{(k_p+1)}}{N^{k_p+1}}
  \right)
  \mathcal J_{\lambda}(\boldsymbol p;\theta)
  \notag=
  \left[
    \prod_{p=1}^m
    \int x^{k_p}\,\mathrm d\mu_{\mathrm{PP}}[\lambda](x)
  \right]J_{\lambda}(\boldsymbol{p};\theta).
\end{align*}
\end{theorem}

\subsection{The Jack--Plancherel process and multi-time moment formula}
\label{subsec:Jack-Plancherel-process}
We upgrade the materials in Sections \ref{subsec:Jack-Plancherel-measure} and \ref{subsec:NS-operator} to their multi-time counterpart as follows. Denote
$\mathbb Y(N):=\{\lambda:\ell(\lambda)\leq N\}$. 
Let the \emph{skew Jack function} $J_{\mu/\lambda}(\boldsymbol{p};\theta)$ be the unique polynomial of $\boldsymbol{p}$ that satisfies
\[
J_{\mu}(\boldsymbol{p}+\boldsymbol{q};\theta)=\sum_{\lambda\subseteq\mu}J_{\mu/\lambda}(\boldsymbol{p};\theta)\,J_{\lambda}(\boldsymbol{q};\theta),
\qquad (\boldsymbol{p}+\boldsymbol{q})_k:=\boldsymbol{p}_k+\boldsymbol{q}_k .
\]
It
vanishes unless $\lambda\subseteq\mu$, and $J_{\lambda/\lambda}=1$.
Also recall the constant term $j_\lambda(\theta)$ from
\eqref{eq:Jack-squared-norm}.  For $u\ge 0$, let
\begin{equation}
  \mathcal G_{N,u}(\boldsymbol p;\theta)
  :=\exp\bigl(\theta u(p_1-N)\bigr)
  =\frac{\mathcal H_\theta(\mathrm{Pl}_u,\boldsymbol p)}
         {\mathcal H_\theta(\mathrm{Pl}_u,1^N)}.
  \label{eq:process-increment-generating-factor}
\end{equation}
For $\lambda,\mu\in\mathbb Y(N)$,  we introduce 
\begin{equation*}
  \mathsf P_{N,u}^{(\theta)}(\lambda,\mu)
  :=e^{-\theta Nu}
    \frac{j_\lambda(\theta)}{j_\mu(\theta)}
    \frac{J_\mu(1^N;\theta)}{J_\lambda(1^N;\theta)}
    J_{\mu/\lambda}(\mathrm{Pl}_u;\theta)
\end{equation*}
as a transition kernel from $\lambda$ to $\mu$.
Let
\begin{equation*}
  \mathcal I_N:=\operatorname{span}\{J_\eta(\boldsymbol p;\theta):
  \ell(\eta)>N\}.
\end{equation*}
By the skew Cauchy identity, in the form of e.g.
\cite[Proposition~A.5]{Huang}, multiplication by
$\mathcal G_{N,u}$ acts on the normalized Jack basis as
\begin{equation}
  \mathcal G_{N,u}(\boldsymbol p;\theta)
  \frac{J_\lambda(\boldsymbol p;\theta)}{J_\lambda(1^N;\theta)}
  =
  \sum_{\mu\in\mathbb Y(N)}
  \mathsf P_{N,u}^{(\theta)}(\lambda,\mu)
  \frac{J_\mu(\boldsymbol p;\theta)}{J_\mu(1^N;\theta)}
  +E_{\lambda,u}^{(N)}(\boldsymbol p),
  \qquad E_{\lambda,u}^{(N)}\in\mathcal I_N.
  \label{eq:process-kernel-action-on-Jack-generating-functions}
\end{equation}
Indeed, the same identity without the restriction $\ell(\mu)\leq N$
produces the left-hand side exactly; the discarded terms are precisely
those with more than $N$ rows.  Evaluating
\eqref{eq:process-kernel-action-on-Jack-generating-functions} at $\boldsymbol{p}=1^N$
shows that
\begin{equation*}
  \sum_{\mu\in\mathbb Y(N)}
    \mathsf P_{N,u}^{(\theta)}(\lambda,\mu)=1.
\end{equation*}
Moreover, by \cite[Proposition A.4]{Huang}
\begin{equation*}
  \sum_{\nu\in\mathbb Y(N)}
  \mathsf P_{N,u}^{(\theta)}(\lambda,\nu)
  \mathsf P_{N,v}^{(\theta)}(\nu,\mu)
  =\mathsf P_{N,u+v}^{(\theta)}(\lambda,\mu).
\end{equation*}
Thus $\mathsf P_{N,u}^{(\theta)}$ is the transition kernel of a
continuous-parameter Markov process on $\mathbb Y(N)$.

\begin{definition}
The \emph{Jack--Plancherel process}
$(\lambda^{(N)}(s))_{s\geq0}$ starts from the empty partition
$\varnothing$ and has transition kernel
\begin{equation*}
  \mathbb P\bigl(
    \lambda^{(N)}(t)=\mu\mid\lambda^{(N)}(s)=\lambda
  \bigr)
  =\mathsf P_{N,t-s}^{(\theta)}(\lambda,\mu),
  \qquad 0\leq s\leq t.
\end{equation*}
\end{definition}
Taking $\lambda=\varnothing$ shows that the time-$s$ marginal is exactly
$\mathbb P_{N,s}^{(\theta)}$ in \eqref{eq:Jack-Plancherel-measure}.

We next record the multi-time version of
Theorem~\ref{thm:Huang-moment-extraction}.  Let
\begin{align*}
  \mathscr D_{N,K}
  :=\frac{I^{(K)}}{N^K}+\frac{I^{(K+1)}}{N^{K+1}},
  \qquad\Phi_K(\lambda):=
  \int x^K\,\mathrm d\mu_{\mathrm{PP}}[\lambda](x).
\end{align*}
By Theorem~\ref{thm:Huang-moment-extraction}, $\mathscr D_{N,K}$ acts diagonally on $J_\lambda$ with eigenvalue $\Phi_K(\lambda)$.  Fix
$s_1\geq s_2\geq\cdots\geq s_m\geq0$, set $s_{m+1}=0$, and define
recursively
\begin{align*}
  \mathcal H_{m,N}
  &:=\mathscr D_{N,K_m}\mathcal G_{N,s_m},
  \notag\\
  \mathcal H_{p,N}
  &:=\mathscr D_{N,K_p}
    \left(\mathcal G_{N,s_p-s_{p+1}}\mathcal H_{p+1,N}\right),
  \qquad p=m-1,\ldots,1.
\end{align*}
The following identity will be the key of our moment extraction.
\begin{theorem}\label{thm:multitimemoment}
    \begin{equation*}
  \mathbb E\left[
    \prod_{p=1}^m
    \int x^{K_p}\,
      \mathrm d\mu_{\mathrm{PP}}[
        \lambda^{(N)}(s_p)](x)
  \right]
  =
  \left.\mathcal H_{1,N}(\boldsymbol p;\theta)
   \right|_{\boldsymbol p=1^N}.
\end{equation*}
\end{theorem}
\begin{proof}
Iterating \eqref{eq:process-kernel-action-on-Jack-generating-functions}
gives the identity modulo $\mathcal I_N$
\begin{align*}
  \mathcal H_{1,N}(\boldsymbol p;\theta)
  &\equiv
  \sum_{\lambda_1,\lambda_{2},\ldots,\lambda_m\in\mathbb Y(N)}
  \left[
    \prod_{q=1}^m
    \mathsf P_{N,s_q-s_{q+1}}^{(\theta)}
      (\lambda_{q+1},\lambda_q)
  \right]
  \left[
    \prod_{q=1}^m\Phi_{K_q}(\lambda_q)
  \right]
  \frac{J_{\lambda_1}(\boldsymbol p;\theta)}
       {J_{\lambda_1}(1^N;\theta)},
\end{align*}
where $\lambda_{m+1}=\varnothing$.
  The ideal $\mathcal I_N$ is preserved by $\mathscr D_{N,K}$ by \cite[Theorem 3.7]{Huang}, and vanishes
after evaluation at $1^N$.  The result then follows from taking 
$\boldsymbol p=1^N$.
\end{proof}

\subsection{Conditional random walk bridges}
\label{subsec:conditional-walk-bridges}

For $N=1,2,\ldots$, let $X_{N,1},X_{N,2},\ldots$ be i.i.d. random variables
with law $\nu_N$ on\footnote{For the Jack Plancherel process, $\nu_N$ is on $\mathbb Z_{\ge -1}$. We consider a wider class of walks so that the framework covers a class of general Jack processes.} $\mathbb Z$ and finite variance
$\sigma_N^2>0$. We impose the following hypotheses.
\begin{enumerate}[label=\textup{(A\arabic*)},ref=\textup{(A\arabic*)}]
\item \label{ass:centered}
      \(\mathbb E X_{N,1}=0\) for every \(N\), and
      \(\sigma_N^2\to\sigma^2\in(0,\infty)\) as $N\to\infty$.
\item \label{ass:exponential}
      There is \(\eta_0>0\) such that
      \[
        \sup_N \mathbb E e^{\eta_0|X_{N,1}|}<\infty .
      \]
\item \label{ass:weak}
      \(\nu_N\Rightarrow\nu\) as $N\to\infty$, where \(\nu\) is an integer-valued
      probability law.
\item \label{ass:aperiodic}
      The array is uniformly strongly aperiodic: for every
      \(\delta\in(0,\pi)\), there is \(\rho_\delta<1\) such that,
      for all sufficiently large \(N\),
      \[
        \sup_{\delta\le |t|\le\pi}
        \left|\sum_{z\in\mathbb Z}e^{itz}\nu_N(z)\right|
        \le \rho_\delta .
      \]
\end{enumerate}
We consider random walk bridges with these increments and collect the
three inputs used later:  one-sided maximum bounds, lattice local limits, and a conditional functional
limit theorem. When $\nu_N=\nu$, these problems were well studied in the random walk literature, see e.g.~\cite{ABR} and~\cite{CaravennaChaumontBridge}. The particular step law
obtained from the operator expansion in Section~\ref{sec:combinatorial-asymptotics} gives  sequences $\nu_N$ in the above class. 
A similar type of random walks was studied in~\cite{KX}, in particular, the Lukasiewicz feature that the only negative increment is $-1$ is crucial in the combinatorial argument there. Nonetheless, in this section we mainly refer to the note \cite{Xu}, which gives a detailed exposition of these results for a more general class of walk increments.

For $H_1,H_2\in\mathbb Z_{\geq0}$ and $L\geq0$, let
$\mathcal W(H_1,H_2,L)$ be the set of paths
$W:\{0,1,\ldots,L\}\to\mathbb Z_{\geq0}$ with
$W(0)=H_1$, $W(L)=H_2$, and increments distributed as $X_{N,1}$.
For a path $W$, set
\[
  \operatorname{wt}(W)
  :=\prod_{q=1}^L
    \mathbb P\bigl(X_{N,1}=W(q)-W(q-1)\bigr),
\]
and, for any collection $\mathcal A$ of such paths, write
$|\mathcal A|:=\sum_{W\in\mathcal A}\operatorname{wt}(W)$.  We shall use
\begin{equation}
  Z_L(A,B):=|\mathcal W(A,B,L)|.
  \label{eq:walk-partition-function}
\end{equation}
For $L=0$, the path is constant and has weight one, so
$Z_0(A,B)=\mathbf 1_{\{A=B\}}$.

\begin{proposition}\cite[Proposition 2.4]{Xu}
\label{prop:KX-one-sided-walk-estimates}
There exist constants $C,c>0$  
such that the following estimates hold for all sufficiently large
\(N\) and all \(L,H,R\in\mathbb Z_{\ge0}\).  Write
\[
  Z_L^{\rightarrow}(H):=Z_L(0,H),\qquad
  Z_L^{\leftarrow}(H):=Z_L(H,0),
\]
and, for any path of length $L$, write
$\max W:=\max_{0\leq t\leq L}W(t)$.  The forward and reverse bridges satisfy the same estimate
\begin{align}
  &\sum_{W\in\mathcal W(0,H,L)}
      \operatorname{wt}(W)
      \mathbf 1_{\{\max W-H\geq R\}}
  \notag\\
  &\qquad\leq
  C\frac{H+1}{(L+1)^{3/2}}
  \exp\!\left\{-c\left(
    \frac{H^2}{L+H+1}
    +\frac{R^2}{L+R+1}\right)\right\}.
  \label{eq:forward-relative-maximum-tail}
\end{align}
\begin{align}
  &\sum_{W\in\mathcal W(H,0,L)}
      \operatorname{wt}(W)
      \mathbf 1_{\{\max W-H\geq R\}}
  \notag\\
  &\qquad\leq
  C\frac{H+1}{(L+1)^{3/2}}
  \exp\!\left\{-c\left(
    \frac{H^2}{L+H+1}
    +\frac{R^2}{L+R+1}\right)\right\}.
  \label{eq:right-bridge-relative-maximum-tail}
\end{align}

Finally, after summing the initial height of a bridge ending at zero,
\begin{align}
  \sum_{H\geq0}\sum_{W\in\mathcal W(H,0,L)}
      \operatorname{wt}(W)
      \mathbf 1_{\{\max W\geq R\}}
  &\leq
  C(L+1)^{-1/2}
  \exp\!\left\{-c\frac{R^2}{L+R+1}\right\},
  \label{eq:KX-free-to-zero-tail}\\
  \sum_{W\in\mathcal W(0,0,L)}
      \operatorname{wt}(W)
      \mathbf 1_{\{\max W\geq R\}}
  &\leq
  C(L+1)^{-3/2}
  \exp\!\left\{-c\frac{R^2}{L+R+1}\right\}.
  \label{eq:KX-zero-to-zero-tail}
\end{align}
\end{proposition}

We also record an estimate of unconditional walks. For $n\geq0$, set
\[
  S_0:=0,\qquad S_n:=\sum_{j=1}^nX_{N,j},\qquad
  p_n(x):=\mathbb P(S_n=x).
\]

\begin{lemma}\cite[Lemma 2.1]{Xu}
\label{lem:unconditioned-local-maximum}There exist constants $C,c>0$, such
that for all sufficiently large $N$ and all $L,R\in\mathbb Z_{\geq0}$,
\begin{equation}
  \sup_{x\in\mathbb Z}
  \mathbb P\!\left(
    S_L=x,\ \max_{0\leq j\leq L}|S_j|\geq R
  \right)
  \leq C(L+1)^{-1/2}
  \exp\!\left\{-\frac{cR^2}{L+R+1}\right\}.
  \label{eq:unconditioned-maximum-bound}
\end{equation}
\end{lemma}

Below is a local limit result for the conditional bridges.

\begin{proposition}
\label{prop:conditional-bridge-local-limits}
Fix $x>0$ and $h_0,h_1>0$, and set
\[
  L_N=\lfloor xN^{2/3}\rfloor,\qquad
  H_{a,N}=\lfloor h_a\sigma N^{1/3}\rfloor,
  \quad a\in\{0,1\}.
\]
Then as $N\to\infty$, there exist $a_1,a_2>0$, such that
\begin{align}
  N^{1/3}|\mathcal W(H_{0,N},H_{1,N},L_N)|
  &\longrightarrow \sigma^{-1}\mathbf F(x;h_1,h_0),
  \label{eq:positive-positive-local-limit}\\
  N^{2/3}|\mathcal W(H_{0,N},0,L_N)|
  &\longrightarrow a_2\mathbf F_0(x;h_0),
  \label{eq:positive-zero-local-limit}\\
  N^{2/3}|\mathcal W(0,H_{1,N},L_N)|
  &\longrightarrow a_1\mathbf F_0(x;h_1),
  \label{eq:zero-positive-local-limit}\\
  N|\mathcal W(0,0,L_N)|
  &\longrightarrow \sigma a_1a_2\mathbf F_{0,0}(x).
  \label{eq:zero-zero-local-limit}
\end{align}

In the case $\operatorname{supp}\nu_N\subseteq
\mathbb Z_{\geq-1}$, these constants are
\[
  a_1=\sigma^{-2},\qquad a_2=(2P_{-1})^{-1},\qquad
  P_{-1}:=\lim_{N\to\infty}\nu_N(-1)>0.
\]
The convergence above is proved in \cite[Theorem~3.2]{Xu}, while the calculations in the proof of \cite[Lemma~4.5]{KX} give the explicit factor for the Lukasiewicz walks.
\end{proposition}

We end this section by the following conditional functional CLT.

\begin{proposition}\cite[Theorem 4.1]{Xu}
\label{prop:conditional-bridge-functional-limit}
Fix $x>0$ and $h_0,h_1\geq0$, and set
\[
  L_N=\lfloor xN^{2/3}\rfloor,
  \qquad
  H_{a,N}=\lfloor h_a\sigma N^{1/3}\rfloor,
  \quad a\in\{0,1\}.
\]
Normalize $\operatorname{wt}$ on
$\mathcal W(H_{0,N},H_{1,N},L_N)$ to a probability measure, and define
the right-continuous step process
\begin{equation*}
  \widehat W_N(u)
  :=\frac{W(\lfloor uN^{2/3}\rfloor)}{\sigma N^{1/3}},
  \qquad 0\leq u\leq x.
\end{equation*}
Then as $N\to\infty$, $\widehat W_N$ converges weakly in the Skorokhod space
$D([0,x])$, equipped with the $J_1$ topology, to the nonnegative
Brownian bridge of length $x$ from $h_0$ to $h_1$.  
If exactly one of $h_0,h_1$ is zero, the limit is the
corresponding three-dimensional Bessel bridge, interpreted by time reversal
when $h_1=0$; if both are zero, it is the normalized Brownian excursion.  
\end{proposition}

\section{Combinatorial expansion and asymptotics}
\label{sec:combinatorial-asymptotics}

In this section, we turn the Nazarov--Sklyanin moment formula for the
Jack--Plancherel process introduced in
Section~\ref{subsec:Jack-Plancherel-process} into a weighted sum over
decorated nonnegative walk bridges, and then pass to its Brownian limit.

Fix $m\geq1$ and let $\boldsymbol\tau$ range over a compact
subset of the Weyl chamber $\{x_1\geq\cdots\geq x_m\}\subset \mathbb R^m$. Set
\begin{align*}
  s_{p,N}&:=2N+\tau_pN^{2/3},
  &a_{p,N}&:=\frac{s_{p,N}}N=2+\tau_pN^{-1/3},
  \\
  z_{p,N}&:=\frac{\sqrt{a_{p,N}}}{1+\sqrt{a_{p,N}}}<1,
  &\mu_{p,N}&:=a_{p,N}+2\sqrt{a_{p,N}}.
\end{align*}
For all sufficiently large $N$, these quantities are positive and
\begin{equation*}
  z_{p,N}\longrightarrow z_c:=2-\sqrt{2},\qquad
  \mu_{p,N}\longrightarrow\mu_+:=2+2\sqrt{2}
\end{equation*}
uniformly in $p$ and locally uniformly in $\boldsymbol\tau$.

Fix
$\boldsymbol k=(k_1,\ldots,k_m)\in\mathbb R_{>0}^m$, and set
\begin{equation*}
  K_p:=\lfloor k_pN^{2/3}\rfloor,
  \qquad
  Q_p:=\sum_{q=1}^pK_q,
  \qquad Q_0:=0.
\end{equation*}
Although suppressed from the notation, $K_p$ and $Q_p$ depend on $N$.

Let $\mu_{\mathrm{PP}}[\lambda]$ be the Perelomov--Popov measure from
Section~\ref{subsec:NS-operator}.  
With expectation taken over the Jack--Plancherel process,
define
\begin{equation}
  \mathcal M_N(\boldsymbol k;\boldsymbol\tau)
  :=\mathbb E\!\left[
      \prod_{p=1}^m
      N\int\left(\frac{x}{\mu_{p,N}}\right)^{K_p}
      \mathrm d\mu_{\mathrm{PP}}
        [\lambda^{(N)}(s_{p,N})](x)
    \right].
  \label{eq:deformed-moment-definition}
\end{equation}

For $\boldsymbol\varepsilon=(\varepsilon_1,\ldots,\varepsilon_m)
\in\{0,1\}^m$, put
\[
  \widetilde K_p:=K_p+\varepsilon_p.
  \qquad
  |\boldsymbol\varepsilon|:=\sum_{p=1}^m\varepsilon_p,
\]

For the process action, recall $\mathcal G_{N,u}$ from \eqref{eq:process-increment-generating-factor} and define recursively
\begin{align*}
  \mathcal B_{m,N}^{\boldsymbol\varepsilon}
  &:=\frac{I^{(\widetilde K_m)}}
       {\mu_{m,N}^{\widetilde K_m}N^{\widetilde K_m}}
       \mathcal G_{N,s_{m,N}},
  \notag\\
  \mathcal B_{p,N}^{\boldsymbol\varepsilon}
  &:=\frac{I^{(\widetilde K_p)}}
       {\mu_{p,N}^{\widetilde K_p}N^{\widetilde K_p}}
    \left(
      \mathcal G_{N,s_{p,N}-s_{p+1,N}}
      \mathcal B_{p+1,N}^{\boldsymbol\varepsilon}
    \right),
  \qquad p=m-1,\ldots,1,
\end{align*}
then let
\begin{equation}
  \mathcal A_N(\boldsymbol\varepsilon;
    \boldsymbol k,\boldsymbol\tau)
  :=\left.N^m\mathcal B_{1,N}^{\boldsymbol\varepsilon}
    \right|_{\boldsymbol p=1^N}.
  \label{eq:single-product-action}
\end{equation} 
Then Theorem \ref{thm:multitimemoment} is rewritten as follows.

\begin{lemma}
\label{lem:exponential-representative}
The deformed moment admits the exact finite decomposition
\begin{equation}
  \mathcal M_N(\boldsymbol k;\boldsymbol\tau)
  =\sum_{\boldsymbol\varepsilon\in\{0,1\}^m}
    \left(\prod_{p=1}^m\mu_{p,N}^{\varepsilon_p}\right)
    \mathcal A_N(\boldsymbol\varepsilon;
      \boldsymbol k,\boldsymbol\tau).
  \label{eq:deformed-moment-action}
\end{equation}
\end{lemma}

For later reference, retain the notation
$\mathbf L_{b,r}[\boldsymbol k;\boldsymbol\tau]$ for the $r$-jump-pair functional
defined in the introduction.  Each jump-pair measure contains the factor
$\theta^{-1}=1+b$.  Also put
\begin{equation*}
  \mathbf L_b[\boldsymbol k;\boldsymbol\tau]
  :=\sum_{r\geq0}\mathbf L_{b,r}[\boldsymbol k;\boldsymbol\tau].
\end{equation*}
The proof below gives the following result, which is a rephrase of Theorem \ref{thm:intro-Jack-Plancherel-process-moments}.

\begin{theorem}[Convergence of the deformed moments]
\label{thm:deformed-moment-convergence}
For every $\theta>0$, $\boldsymbol k\in\mathbb R_{>0}^m$, and
$\tau_1\geq\cdots\geq\tau_m$, the series
$\mathbf L_b[\boldsymbol k;\boldsymbol\tau]$ is absolutely convergent.  Moreover, for
every $\boldsymbol\varepsilon\in\{0,1\}^m$,
\begin{equation}
  \lim_{N\to\infty}
  \mathcal A_N(\boldsymbol\varepsilon;\boldsymbol k
    ,\boldsymbol\tau)
  =\mathbf L_b[\boldsymbol k;\boldsymbol\tau],
  \qquad b=\theta^{-1}-1.
  \label{eq:single-product-action-limit}
\end{equation}
Consequently,
\begin{equation}
  \lim_{N\to\infty}\mathcal M_N(\boldsymbol k
    ;\boldsymbol\tau)
  =(1+\mu_+)^m\mathbf L_b[\boldsymbol k
    ;\boldsymbol\tau].
  \label{eq:deformed-moment-limit}
\end{equation}
Both convergence statements are locally uniform in the
ordered vector $\boldsymbol\tau$.
\end{theorem}

\subsection{The decorated-path expansion}
\label{subsec:decorated-path-expansion}

The remainder of Section~\ref{sec:combinatorial-asymptotics} is devoted to the study of the action
$\mathcal A_N(\boldsymbol\varepsilon;
\boldsymbol k,\boldsymbol\tau)$ in~\eqref{eq:single-product-action}.
The nesting orders the excursion blocks from the latest physical time
$s_{1,N}$ to the earliest physical time $s_{m,N}$; this is also their
left-to-right order below.
Recall that the matrix indices of the Nazarov--Sklyanin Lax matrix start at zero, and
\begin{equation*}
  L_{ij}=b j\,\mathbf 1_{\{i=j\}}+p_{j-i},
  \qquad
  p_{-d}=\frac d\theta\frac{\partial}{\partial p_d},
  \qquad p_0=0.
\end{equation*}
In particular,
\begin{equation}
  I^{(K)}
  =\sum_{i_1,\ldots,i_{K-1}\geq0}
    L_{0i_1}L_{i_1i_2}\cdots L_{i_{K-2}i_{K-1}}L_{i_{K-1}0}.
  \label{eq:Ik-as-positive-path}
\end{equation}
We read the indices in~\eqref{eq:Ik-as-positive-path} from left to right, although the corresponding operators act on the function to their right from right to left. Writing $i_0=i_K=0$, every summand is encoded by a lattice path
\[
 \Gamma:= (i_0,i_1,\ldots,i_K)\in \mathbb Z_{\ge0}^{K+1}
\]
which starts at $0$, stays nonnegative at the intermediate times, and returns to $0$ at time $K$. We call such a nonnegative bridge an \emph{excursion block}; it may return to zero at intermediate times. A matrix entry at the $s^{\mathrm{th}}$ position has one of the following four effects:
\begin{enumerate}[label=(\roman*)]
  \item if $i_s=i_{s-1}$, it contributes the diagonal factor $b i_s$;
  \item if $i_s-i_{s-1}=d>0$, it multiplies by $p_d$;
  \item if $i_s-i_{s-1}=-d<0$, the derivative $p_{-d}$ may remove one of the copies of $p_d$ already created to its right in the operator word; choosing such a copy contributes $d/\theta$;
  \item if $i_s-i_{s-1}=-1$, the derivative may instead act on the
  current exponential in block $p$.  The product of
  the initial representative and all transition increments to its right is
  $\mathcal G_{N,s_{p,N}}=\exp(\theta s_{p,N}(p_1-N))$, so this contributes
  $s_{p,N}=a_{p,N}N$.
\end{enumerate}
After all operators have acted, every multiplication operator in (ii) that is not removed in (iii) contributes $p_d|_{\boldsymbol p=1^N}=N$. Hence an unpaired positive jump and a $-1$ jump of type (iv), after the factor $N^{-1}$ attached to each matrix entry, have weights $1$ and $a_{p,N}$ in block $p$, respectively. These are the \emph{canonical steps}. The diagonal entry in (i) records the current height together with a factor $b$, while case (iii) creates a pair consisting of a negative jump of size $d$ and a later positive jump of the same size.

It is useful to make the pairing in (iii) explicit. Expand every derivative by the Leibniz rule. Each time a derivative $p_{-d}$ removes a copy of $p_d$, join their two positions by an arc with index $d$. The multiplication lies to the right in the operator word and is therefore applied earlier; in the left-to-right path parametrization the negative jump occurs first. Distinct derivatives remove distinct copies of the polynomial variables. Conversely, the path and the collection of arcs determine one term in the iterated Leibniz expansion. The factor $\deg(p_d)$ that appears when $p_{-d}$ acts is precisely the number of admissible copies of $p_d$, so it is accounted for by the choice of the partner of the negative jump.

For the nested action of $m$ operators, concatenate the $m$ index paths and place the $p^{\mathrm{th}}$ path on the integer interval $\llbracket Q_{p-1},Q_p\rrbracket$. Each path starts and ends at zero and is nonnegative in its interior. A \emph{decorated configuration} consists of this concatenated path together with
\begin{itemize}
  \item $r$ arcs joining a negative jump to a positive jump on its right of the same size, and
  \item $\rho$ marked zero increments, at each of which the current height is recorded.
\end{itemize}
An arc may join two positions in the same excursion block or in two different excursion blocks. In the latter case its negative jump position necessarily belongs to the excursion block with smaller index. This is exactly the finite set of index assignments used in the definition of $\mathbf L_{b,r}$. We call the two jumps connected by an arc a \emph{jump pair}, consisting of a positive element $p_d$ and a negative element $p_{-d}$ with common height $d$.
\begin{figure}[t]
\centering
\begin{tikzpicture}[x=.57cm,y=.68cm,>=stealth]
  \draw[step=1,gray!22,very thin,dotted] (0,0) grid (18,5);
  \draw[->] (0,0) -- (18.8,0) node[right] {$t$};
  \draw[->] (0,0) -- (0,5.75) node[above] {$i_{\lfloor t\rfloor}$};
  \draw[gray,dashed] (9,0) -- (9,5.35);
  \node[below=3pt] at (0,0) {$Q_0$};
  \node[below=3pt] at (9,0) {$Q_1$};
  \node[below=3pt] at (18,0) {$Q_2$};

  \foreach \xa/\xb/\y in {
      0/1/0,1/2/2,2/3/4,3/4/2,4/5/3,5/6/5,6/7/4,7/8/2,8/9/1,
      9/10/0,10/11/1,12/13/3,13/14/5,14/15/4,15/16/3,16/17/2,17/18/1}
    \draw[black,very thick] (\xa,\y)--(\xb,\y);
  \draw[red,very thick] (11,3)--(12,3);

  \draw[blue,very thick,dashed] (3,4)--(3,2);
  \draw[blue,very thick,dashed] (5,3)--(5,5);
  \draw[violet,very thick,dashed] (7,4)--(7,2);
  \draw[violet,very thick,dashed] (11,1)--(11,3);

  \foreach \x/\y in {
      1/0,2/2,3/4,4/2,5/3,6/5,7/4,8/2,9/1,
      10/0,11/1,13/3,14/5,15/4,16/3,17/2,18/1}
    \draw[black,fill=white,line width=.6pt] (\x,\y) circle (1.65pt);
  \draw[red,fill=white,line width=.6pt] (12,3) circle (1.65pt);
  \fill[black] (0,0) circle (1.5pt);
  \fill[black] (9,0) circle (1.5pt);
  \fill[black] (18,0) circle (1.5pt);

  \node[blue,anchor=east] at (2.62,3) {$-d_1$};
  \node[blue,anchor=west] at (5.38,4.38) {$+d_1$};
  \node[violet,anchor=west] at (7.38,3) {$-d_2$};
  \node[violet,anchor=west] at (11.38,2) {$+d_2$};
  \node[blue,align=center] at (4,5.55) {same-excursion\\pair};
  \node[violet,align=center] at (9.35,5.55) {cross-excursion\\pair};
  \node[red,above=2pt] at (11.5,3) {height mark};
\end{tikzpicture}
\caption{A decorated lattice configuration for $m=2$.  On each interval between consecutive integer times the graph is horizontal, and the hollow circle at its right endpoint records the half-open convention $t\mapsto i_{\lfloor t\rfloor}$.  Ordinary jumps are represented only by the discontinuities between consecutive horizontal segments.  The blue dashed jumps form a same-excursion pair, whereas the violet dashed jumps form a cross-excursion pair; in each case the negative and positive jumps have the same size.  The red horizontal segment is a marked zero increment, and the gray dashed line marks the interface $Q_1$.}
\label{fig:decorated-configuration-m2}
\end{figure}
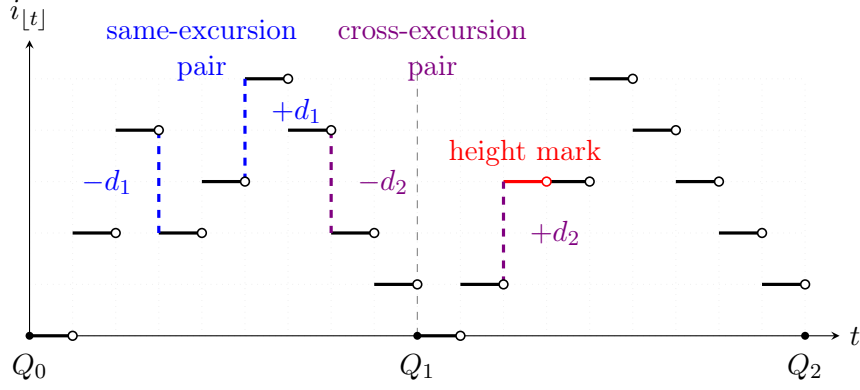

Let $\mathfrak C_N(r,\rho)$ be the set of configurations with $r$ jump pairs and $\rho$ marked zero increments. Denote the heights of the jump pairs by $d_1,\ldots,d_r\in\mathbb Z_{\geq1}$ and the heights at the marked positions by $H_1,\ldots,H_\rho$.
In the process calculation, a canonical step in block $p$ has
normalized weight
\begin{equation}
  \frac{a_{p,N}}{\mu_{p,N}}
  \quad\hbox{for a $-1$ step},
  \qquad
  \frac1{\mu_{p,N}}
  \quad\hbox{for a positive step}.
  \label{eq:canonical-normalized-weights}
\end{equation}
If the two jumps of arc $j$ lie in blocks
$\ell_-^j$ and $\ell_+^j$, and if $p(q)$ is the block containing height
mark $q$, the decoration weights are
\begin{equation}
  \frac{d_j}
  {\theta\mu_{\ell_-^j,N}\mu_{\ell_+^j,N}N^2},
  \qquad
  \frac{bH_q}{\mu_{p(q),N}N}.
  \label{eq:decoration-normalized-weights}
\end{equation}

Finally one multiplies by $N^m$. 
If $n_{p,-}(\Gamma)$ and $n_{p,+}(\Gamma)$ are the two
canonical-step counts in block $p$, the process weight, with its dependence
on $\boldsymbol\tau$ suppressed from the notation, is
\begin{align}
  w_N(\Gamma)
  &:=\frac{N^m}{r!}
  \prod_{p=1}^m
  \left(\frac{a_{p,N}}{\mu_{p,N}}\right)^{n_{p,-}(\Gamma)}
  \left(\frac1{\mu_{p,N}}\right)^{n_{p,+}(\Gamma)}
  \notag\\
  &\qquad\times
  \prod_{j=1}^r
  \frac{d_j}
  {\theta\mu_{\ell_-^j,N}\mu_{\ell_+^j,N}N^2}
  \prod_{q=1}^{\rho}
  \frac{bH_q}{\mu_{p(q),N}N}.
  \label{eq:weight-of-configuration}
\end{align}

The factor $1/r!$ compensates for the artificial labels of the arcs, which create $r!$ copies of the same Leibniz term.  The $\rho$ height marks are an unordered subset of the horizontal positions; accordingly, no additional factorial is included in $w_N(\Gamma)$.

We now introduce a mean 0 random-walk interpretation of the canonical steps of the operator expansion, whose increment law depends on $p$ and $N$. For block $p$, put
\begin{equation*}
  V_{p,N}(z):=\frac{a_{p,N}}z+\frac{z}{1-z},
  \qquad 0<z<1.
\end{equation*}
Then $V'_{p,N}(z_{p,N})=0$ and
$V_{p,N}(z_{p,N})=\mu_{p,N}$.  Denote the increment law in block $p$ by $\nu_{p,N}$, which satisfies 
\begin{equation}
  \mathbb P(X_{p,N}=-1)=\frac{a_{p,N}z_{p,N}^{-1}}{\mu_{p,N}},
  \qquad
  \mathbb P(X_{p,N}=d)=\frac{z_{p,N}^d}{\mu_{p,N}},
  \quad d\geq1.
  \label{eq:process-canonical-step-law}
\end{equation}
Taking a sequence $\boldsymbol{\tau}_j$ in a compact set and $N_j\to\infty$, one can easily check that for $p=1,2,\ldots,m$, the increment law sequence $\nu_{p,N_j}(\boldsymbol{\tau}_j)$ satisfies all the assumptions in Section \ref{subsec:conditional-walk-bridges}, so that  $\nu_{p,N_j}(\boldsymbol{\tau}_j)\Rightarrow \nu$, a centered random variable on $\mathbb Z$ with exponential tail.
Moreover, the values of $P_{-1}$ and $\sigma$ in \eqref{eq:intro-multitime-fixed-functional} are given by
\begin{align*}
  P_{p,N,-1}
  :=\frac{a_{p,N}z_{p,N}^{-1}}{\mu_{p,N}}
    =\frac{\sqrt{a_{p,N}}+1}{\sqrt{a_{p,N}}+2}\rightarrow\frac{1}{\sqrt{2}},&
  \qquad
  \sigma_{p,N}^2
  :=\frac{z_{p,N}^2V''_{p,N}(z_{p,N})}{\mu_{p,N}}
    =\frac{2(1+\sqrt{a_{p,N}})^2}{\sqrt{a_{p,N}}+2}\rightarrow 2+\sqrt{2}.
\end{align*}

The stochastic viewpoint applies only to the canonical steps. If one erases the arcs and the marked horizontal steps, the remaining path is a concatenation of nonnegative walk pieces governed by the centered law~\eqref{eq:process-canonical-step-law} in the corresponding block; the arcs and the height marks should then be viewed as deterministic insertions on top of this canonical walk. 
For a canonical segment in block $p$, the exact identity is
\begin{equation}
  \prod_{s=1}^{\ell}\frac{w_{p,N}(x_s)}{\mu_{p,N}}
  =z_{p,N}^{-\sum_sx_s}
    \prod_{s=1}^{\ell}\mathbb P(X_{p,N}=x_s),
  \qquad
  w_{p,N}(-1)=a_{p,N},\quad w_{p,N}(d)=1.
  \label{eq:exponential-tilt-identity}
\end{equation}
Let the signed paired-jump displacement in block $p$ be
\begin{equation*}
  J_p(\Gamma):=
  \sum_{j=1}^r d_j
  \left(\mathbf1_{\{\ell_+^j=p\}}
       -\mathbf1_{\{\ell_-^j=p\}}\right).
\end{equation*}
Because the full path in block $p$ starts and ends at zero, its canonical
increments sum to $-J_p(\Gamma)$.  Thus the tilt factors in
\eqref{eq:exponential-tilt-identity} leave the single extra multiplier
\begin{align}
  \Xi_N(\Gamma;\boldsymbol\tau)
  &:=\prod_{p=1}^m z_{p,N}^{J_p(\Gamma)}
    =\prod_{j=1}^r
      \left(\frac{z_{\ell_+^j,N}}
                 {z_{\ell_-^j,N}}\right)^{d_j}.
  \label{eq:finite-N-cross-time-factor}
\end{align}
A same-block pair contributes one.  For a cross-block pair
$\ell_-^j<\ell_+^j$, the ordering of the time vector and the monotonicity
of $a\mapsto\sqrt a/(1+\sqrt a)$ give
\begin{equation}
  0<\left(\frac{z_{\ell_+^j,N}}
                 {z_{\ell_-^j,N}}\right)^{d_j}\leq1,
  \qquad 0<\Xi_N(\Gamma;\boldsymbol\tau)\leq1.
  \label{eq:finite-N-cross-time-factor-bound}
\end{equation}
Apart from this multiplier, the undecorated pieces are bridges of the
centered row-$N$ law~\eqref{eq:process-canonical-step-law}, conditioned
to stay nonnegative.

\begin{lemma}[Exact finite-$N$ expansion]
\label{lem:exact-decorated-path-expansion}

For every
$\boldsymbol\varepsilon\in\{0,1\}^m$, with the block lengths in
$\mathfrak C_N(r,\rho)$ replaced by $\widetilde K_p$,
\begin{equation*}
  \mathcal A_N(\boldsymbol\varepsilon;
    \boldsymbol k,\boldsymbol\tau)
  =\sum_{r,\rho\geq0}
    \sum_{\Gamma\in\mathfrak C_N(r,\rho)}w_N(\Gamma).
\end{equation*}
\end{lemma}

\begin{proof}
Expand each $I^{(\widetilde K_{p})}$ in the nested
action by~\eqref{eq:Ik-as-positive-path}, and then apply the entries from right to left. The Leibniz rule gives exactly the four alternatives (i)--(iv). A nonzero term cannot contain an unpaired derivative other than a copy of $p_{-1}$ acting on the current block exponential. Pair every remaining derivative with the polynomial variable on which it acts. This produces a unique decorated configuration. Conversely, the configuration specifies, at every derivative, whether it acts on the exponential or on which polynomial variable it acts, and hence reconstructs the operator term. The weights in~\eqref{eq:canonical-normalized-weights} and~\eqref{eq:decoration-normalized-weights} are the corresponding matrix entries divided by $N\mu_{p,N}$ at a position in block $p$, while $1/r!$ removes the artificial labels of the arcs.  The height marks are already unordered in the definition of $\mathfrak C_N(r,\rho)$.  This proves both the bijection and the weight identity.
\end{proof}

\subsection{A uniform summability estimate}
\label{subsec:uniform-summability}

We now prove the uniform estimate needed to sum the decorated-path
expansion and to pass the infinite sums through the Brownian limit.  The
main point is a reconstruction which retains every configuration whose
canonical part may go below zero before the cross-excursion jumps are
inserted.

If a jump occupies the step from original time $t-1$ to original time
$t$, we call $t$ its \emph{position}.  Thus the negative position of a
jump pair is always smaller than its positive position.  Delete all paired
jump positions and marked horizontal positions, and put
\begin{equation*}
  D_p:=\#\{\text{deleted positions in the $p^{\mathrm{th}}$ excursion block}\},
  \qquad L_p:=\widetilde K_p-D_p.
\end{equation*}

Let $\mathcal X_p$ be the set of cross-excursion jump occurrences in the
$p^{\mathrm{th}}$ excursion block.  If $\mathcal X_p\neq\varnothing$, let
$\chi_p$ and
$\chi'_p$ be its leftmost and rightmost positions, respectively; thus
$\chi_p\leq\chi'_p$.  If
$\mathcal X_p=\varnothing$, set
$\chi_p=\chi'_p=\varnothing$.  We use no cut in the
latter case, one cut when
$\chi_p=\chi'_p\neq\varnothing$, and two cuts when
$\chi_p<\chi'_p$.  Write
\begin{equation*}
  c_p:=
  \begin{cases}
    0,&\chi_p=\chi'_p=\varnothing,\\
    1,&\chi_p=\chi'_p\neq\varnothing,\\
    2,&\chi_p<\chi'_p,
  \end{cases}
  \qquad s:=\sum_{p=1}^m c_p\leq2m.
\end{equation*}
The $s$ retained extreme cross-excursion occurrences are called \emph{boundary jumps},
and all remaining jump occurrences are called \emph{ordinary jumps}.  This
terminology refers to the boundaries of the pieces below; two boundary
jumps may belong to the same pair.

We process the excursion blocks from right to left.  Consequently, the height
of a negative cross-excursion jump has already been assigned at its positive
partner in a later excursion block, whereas the height of a positive
cross-excursion jump is assigned when that positive occurrence is processed.
Every cross-excursion pair height is therefore assigned exactly once.

If $c_p=0$,  retain one
nonnegative canonical bridge from $0$ to $0$.  If $c_p=1$, let
$\mathfrak{s}_p\in\{-,+\}$ and $d_p\in\mathbb Z_{\geq1}$ be the sign and height of the unique
cross-excursion jump.  If $\ell_p$ canonical steps survive to its left,
the two nonnegative pieces are
\begin{align*}
  0&\longrightarrow H_p+d_p,
  &H_p&\longrightarrow0,
  &&\mathfrak{s}_p=-,
  \\
  0&\longrightarrow H_p,
  &H_p+d_p&\longrightarrow0,
  &&\mathfrak{s}_p=+,
\end{align*}
for some $H_p\geq0$, with respective lengths $\ell_p$ and
$L_p-\ell_p$.

Suppose now that $\chi_p<\chi'_p$.  Let
$\ell_p^{\mathrm l},\ell_p^{\mathrm c},\ell_p^{\mathrm r}$ be the
numbers of surviving canonical steps to the left of $\chi_p$, strictly
between $\chi_p$ and $\chi'_p$, and to the right of $\chi'_p$.  Thus
\[
  \ell_p^{\mathrm l}+\ell_p^{\mathrm c}+\ell_p^{\mathrm r}=L_p.
\]
First expose the middle canonical walk $W_p^{\mathrm c}$ from its right
end toward its left end, starting with right ordinate $0$ and imposing
\emph{no} nonnegativity conditioning.  Here the word \emph{canonical}
refers to the original left-to-right increments: when the path is exposed
from right to left, the revealed increments are $-X$, with $X$ distributed
according to~the $N$-dependent law
\eqref{eq:process-canonical-step-law} for block $p$.  Insert the cross-excursion
jumps strictly between $\chi_p$ and $\chi'_p$, again from right to left.
At a positive jump of height $d$, choose $d$ and translate the part lying
to its left downward by $d$; at a negative jump use the height already
fixed in the later excursion block and translate that left part upward by $d$.
Denote the resulting middle profile by $Y_p$, normalized by
$Y_p(\chi'_p-1)=0$, and set
\begin{equation*}
  \mathfrak m_p:=\max(-\min Y_p,0),
  \qquad C_p:=Y_p(\chi_p).
\end{equation*}
Thus $\mathfrak m_p$ is the depth of the middle profile below zero.

For each $t\in\mathcal X_p$, write $\eta(t)\in\{-1,+1\}$ for the sign
of the jump, with $+1$ denoting a positive jump, and write $d(t)\in\mathbb Z_{\geq1}$ for its
height.  We shall use the signed total jump displacement
\begin{equation}
  \delta_p:=\sum_{t\in\mathcal X_p}\eta(t)d(t).
  \label{eq:signed-cross-jump-displacement}
\end{equation}
Let $\eta_p:=\eta(\chi_p)$,
$\eta'_p:=\eta(\chi'_p)$ and
$d_p:=d(\chi_p)$,
$d'_p:=d(\chi'_p)$.  A positive boundary height is
assigned now; a negative one is already fixed.  Choose an integer $H_p$ satisfying
\begin{equation}
  H_p\geq\mathfrak m_p,
  \qquad
  A_p:=H_p+C_p-\eta_p d_p\geq0,
  \qquad
  B_p:=H_p+\eta'_p d'_p\geq0.
  \label{eq:three-piece-feasibility}
\end{equation}
In particular, if the right boundary jump is negative, then
$H_p\geq d'_p$.  Translate the entire middle profile upward by $H_p$,
draw a nonnegative left bridge from $0$ to $A_p$ of length
$\ell_p^{\mathrm l}$, and draw a nonnegative right bridge from $B_p$ to
$0$ of length $\ell_p^{\mathrm r}$.  The jumps at
$\chi_p,\chi'_p$ join these two bridges to the translated middle profile.
The middle part is now nonnegative by the first condition in
\eqref{eq:three-piece-feasibility}.

Finally reinsert every pair whose two occurrences lie in the same
excursion block by translating the profile between its negative and positive
occurrences downward by the pair height, and then reinsert the marked
horizontal steps.  An interval belonging to such a pair either contains
both $t$ and $t-1$ of a distinct cross-excursion jump position $t$ or contains
neither, so this local operation does not change any cross-excursion jump
height.  We retain precisely the choices for which the resulting path is
nonnegative.

\begin{lemma}[Extreme-cut reconstruction]
\label{lem:extreme-cut-reconstruction}
For fixed jump positions, signs, pairing, and marked horizontal positions,
the preceding right-to-left construction is injective and has the
following properties.
\begin{enumerate}
  \item Every admissible decorated configuration occurs exactly once.
  In particular, a middle canonical profile which is negative before the
  cross-excursion translations and the shift by $H_p$ is not discarded.
  \item Every pair height is assigned at its positive occurrence.  The
  negative occurrence introduces no further height sum.
  \item The product of the canonical-step weights is the product of the
  weights of the auxiliary canonical pieces.
\end{enumerate}
\end{lemma}

\begin{proof}
Starting from an admissible configuration, first undo every
same-excursion interval lowering.  This operation is deterministic because
the pair heights and positions are fixed.  In an excursion block with two cuts,
$H_p$ is then the ordinate immediately to the left of the jump at
$\chi'_p$ after that jump has been removed.  Subtracting $H_p$ from the
middle profile and undoing the interior cross-excursion translations in
left-to-right order recovers the unique unrestricted middle walk with
right ordinate zero.  The two outer bridges and the one-cut or uncut
pieces are then read off uniquely.  This proves injectivity and shows that
no admissible configuration is omitted.  Conversely, the construction
and the final nonnegativity test plainly produce an admissible decorated
configuration.

The right-to-left order meets the positive occurrence of every
cross-excursion pair before its negative occurrence.  Hence its height is
assigned exactly once.  All reconstruction operations are vertical
translations on intervals and do not alter canonical increments.  As for the weight statement, simply note that the vertical shift does not change the increments on the canonical pieces.
\end{proof}
For a finite path $W$, denote its \emph{oscillation} as
\[
  \operatorname{osc}(W):=\max_jW(j)-\min_jW(j).
\]

Let $\widehat M(\Gamma)$ be the larger of the following two quantities:
the maximum height of the auxiliary skeleton after all cross-excursion
jumps and the shifts $H_p$ have been inserted but before the
same-excursion interval lowerings are applied, and the largest oscillation
of an unrestricted middle walk $W_p^{\mathrm c}$.  Let $M(\Gamma)$ be
the maximum of the final decorated path.  Then
\begin{equation}
  M(\Gamma)\leq\widehat M(\Gamma),
  \qquad H_p\leq\widehat M(\Gamma),
  \label{eq:new-max-comparison}
\end{equation}
and every jump height and every marked horizontal height is at most $M(\Gamma)$, hence at most
$\widehat M(\Gamma)$.

For a cross-excursion pair $e$, let $p_-(e)<p_+(e)$ be the indices of
the excursion blocks containing its negative and positive occurrences.
For $1\leq h<m$, define the total cross-excursion height passing through
the interface $h\mid h+1$ by
\begin{equation*}
  F_h:=\sum_{\substack{e\ \mathrm{cross\text{-}excursion}\\
                   p_-(e)\leq h<p_+(e)}}d_e,
  \qquad F_0=F_m=0.
\end{equation*}

For later use, define $\mathcal R_p$ to be the largest of the endpoint
heights of the \begin{itemize}
    \item nonnegative auxiliary pieces in the $p^{\mathrm{th}}$ excursion block,
    \item their
relative overshoots above the larger endpoint, and
     \item when $c_p=2$, the
oscillation of $W_p^{\mathrm c}$.

\end{itemize}  
\begin{lemma}[Height propagation]
\label{lem:extreme-cut-height-propagation}
Put $R:=\max_{1\leq p\leq m}\mathcal R_p$. Then, for
$1\leq h<m$,
\begin{equation}
  F_h=-\sum_{a=1}^h\delta_a,
  \qquad 0\leq F_h\leq3hR.
  \label{eq:cross-excursion-flow-identity}
\end{equation}
Moreover,
\begin{equation}
  \widehat M(\Gamma)\leq(3m+2)R.
  \label{eq:skeleton-height-propagation}
\end{equation}
\end{lemma}

\begin{proof}
If $c_p=0$, then $\delta_p=0$. If $c_p=1$, the unique
cross-excursion jump has height at most $R$. If $c_p=2$ and $x_p$ is
the left-to-right displacement of the raw middle walk, then
\[
  A_p+x_p+\delta_p=B_p,
\]
and therefore $|\delta_p|\leq3R$. A pair passing through the
interface $h\mid h+1$ contributes $-d_e$ to
$\sum_{a=1}^h\delta_a$, while every other pair contributes zero.
This proves \eqref{eq:cross-excursion-flow-identity}.

By the definition of $R$, the maximum of an outer auxiliary piece is at
most its larger endpoint plus its relative overshoot, hence at most $2R$.
Before the cross-excursion jumps are inserted, the same bound holds for a
middle piece because its right endpoint and its raw oscillation are each at
most $R$. Read this piece from right to left. Positive cross-excursion jumps
translate its left part downward and cannot increase its upper envelope,
whereas the negative cross-excursion jumps in block $p$ can raise it by a
total height at most $F_p$. Consequently its maximum is at most
\[
  2R+F_p\leq(3m+2)R.
\]
This proves \eqref{eq:skeleton-height-propagation}.
\end{proof}

Recall from \eqref{eq:walk-partition-function} that
$Z_L(A,B)$ is the partition function of a nonnegative walk bridge from $A$ to $B$ of length $L$, and use the notation $Z_n^{\rightarrow}$ and
$Z_n^{\leftarrow}$ introduced in
Proposition~\ref{prop:KX-one-sided-walk-estimates}.

We estimate one-cut excursion blocks as follows.
For $L\in\mathbb Z_{\geq0}$ and $d\in\mathbb Z_{\geq1}$, define two cut kernels
\begin{align*}
  \mathcal K_L^-(d)
  &:=\sum_{\ell=0}^{L}\sum_{H\geq0}
       Z_\ell^{\rightarrow}(H+d)
       Z_{L-\ell}^{\leftarrow}(H),
       \\
  \mathcal K_L^+(d)
  &:=\sum_{\ell=0}^{L}\sum_{H\geq0}
       Z_\ell^{\rightarrow}(H)
       Z_{L-\ell}^{\leftarrow}(H+d).
\end{align*}
The superscript records the sign of the boundary jump.

\begin{lemma}[Summed cut-kernel estimate]
\label{lem:summed-cut-kernel}
There are constants $C,c>0$ such that, uniformly in
$L\in\mathbb Z_{\geq0}$ and $d\in\mathbb Z_{\geq1}$,
\begin{equation}
  \mathcal K_L^\pm(d)
  \leq \frac{C}{(L+1)^{1/2}}
       \exp\!\left\{-\frac{c d^2}{L+d+1}\right\}.
  \label{eq:summed-cut-kernel}
\end{equation}
Expand the two partition functions in each summand into paths, and let
$\mathcal K_{L,\geq R}^\pm(d)$ denote the restriction to pairs of raw
walks $(W^{\mathrm l},W^{\mathrm r})$ satisfying
\[
  \max\{\max W^{\mathrm l},\max W^{\mathrm r}\}\geq R.
\]
Then uniformly in $R\geq0$,
\begin{equation}
  \mathcal K_{L,\geq R}^\pm(d)
  \leq\frac{C}{(L+1)^{1/2}}
  \exp\!\left\{-c\left(
    \frac{d^2}{L+d+1}+\frac{R^2}{L+d+R+1}
  \right)\right\}.
  \label{eq:summed-cut-kernel-tail}
\end{equation}
\end{lemma}

\begin{proof}
Put $A=\ell+1$, $B=L-\ell+1$, and $T=A+B=L+2$.  We first prove
the estimate for a negative jump.  For fixed $A,B$, split the height sum
into $H+d\leq A$, $H\leq B$; $H+d\leq A$, $H>B$; and $H+d>A$.

In the first region, $H+d\leq A$, $H\le B$, so the two exponential terms both give Gaussian tail.  Together with
\eqref{eq:right-bridge-relative-maximum-tail} at $R=0$, this gives
\[
  Z_\ell^{\rightarrow}(H+d)Z_{L-\ell}^{\leftarrow}(H)
  \leq
  \frac{C}{AB}
  \exp\!\left\{-c\left(
    \frac{(H+d)^2}{A}+\frac{H^2}{B}
  \right)\right\}.
\]
Since
\[
  \frac{(H+d)^2}{A}+\frac{H^2}{B}
  =\frac{T}{AB}\left(H+\frac{Bd}{T}\right)^2
   +\frac{d^2}{T},
\]
summing the remaining Gaussian in $H$ gives
\begin{equation*}
 \sum_{\substack{H\geq0\\H+d\leq A,\ H\leq B}}
 Z_\ell^{\rightarrow}(H+d)Z_{L-\ell}^{\leftarrow}(H)
 \leq \frac{C}{\sqrt{ABT}}e^{-cd^2/T}.
\end{equation*}

In the second region, $H+d\le A$ and $H> B$, then \eqref{eq:right-bridge-relative-maximum-tail} at $R=0$ gives the bound
\[Z_{L-\ell}^{\leftarrow}(H)\le \frac{C}{B^{3/2}},\]
therefore \[ Z_\ell^{\rightarrow}(H+d)Z_{L-\ell}^{\leftarrow}(H)
  \leq
  \frac{C}{AB^{3/2}}
  \exp\!\left\{-c\left(
    \frac{(H+d)^2}{A}
  \right)\right\}\]
and we have \begin{equation*}
 \sum_{\substack{H\geq0\\H+d\leq A,\ H>B}}
 Z_\ell^{\rightarrow}(H+d)Z_{L-\ell}^{\leftarrow}(H)
 \leq \frac{C}{\sqrt{A}B^{3/2}}e^{-cd^2/T}.
\end{equation*}

In the third region, $x:=H+d>A$, and hence
$x^2/(A+x)\geq x/2$.  \eqref{eq:forward-relative-maximum-tail} and
\eqref{eq:right-bridge-relative-maximum-tail} at $R=0$ therefore give
\begin{align*}
 &\sum_{\substack{H\geq0\\H+d>A}}
 Z_\ell^{\rightarrow}(H+d)Z_{L-\ell}^{\leftarrow}(H)
 \leq
 \frac{C}{A^{3/2}B}
 \sum_{H\geq0}(H+d+1)e^{-c(H+d)}
 \leq \frac{C}{A^{3/2}B}e^{-c'd}.
\end{align*}

We now sum the cut time.  The elementary bounds
\[
  \sum_{\ell=0}^{L}\frac1{\sqrt{AB}}\leq C,
  \qquad
\sum_{\ell=0}^{L}\frac{1}{\sqrt{A}B^{3/2}}\le \frac{C}{\sqrt{T}},
  \qquad
  \sum_{\ell=0}^{L}\frac1{A^{3/2}B}\leq\frac{C}{T}
\]
show that the first two regions together and the third region contribute at most
\[
  CT^{-1/2}e^{-cd^2/T}
  \qquad\text{and}\qquad
  CT^{-1}e^{-cd},
\]
respectively.  Both are bounded by
$CT^{-1/2}\exp\{-c'd^2/(T+d)\}$.
Combining the three regions gives
\eqref{eq:summed-cut-kernel} for $\mathcal K_L^-(d)$.

We next retain the maximum cutoff.  For $x,y\geq0$, set
\begin{align*}
 F_A^{[R]}(x)
 &:=\sum_{W\in\mathcal W(0,x,A-1)}\operatorname{wt}(W)
       \mathbf 1_{\{\max W\geq R\}},\\
 G_B^{[R]}(y)
 &:=\sum_{W\in\mathcal W(y,0,B-1)}\operatorname{wt}(W)
       \mathbf 1_{\{\max W\geq R\}}.
\end{align*}
If $x\geq R/2$, reserve half of the endpoint exponential in
\eqref{eq:forward-relative-maximum-tail} with $R=0$; if $x<R/2$, then
$\max W-x\geq R/2$, so
\eqref{eq:forward-relative-maximum-tail} applies.  These two cases give
the pointwise bound
\begin{equation*}
 F_A^{[R]}(x)
 \leq C\frac{x+1}{A^{3/2}}
 \exp\!\left\{-c\left(
   \frac{x^2}{A+x}+\frac{R^2}{A+R}
 \right)\right\}.
\end{equation*}
For the right bridge, split instead according to $y\geq R/2$ and use
\eqref{eq:right-bridge-relative-maximum-tail} with $R=0$ in that case;
when $y<R/2$, use the same estimate with relative maximum $R/2$.  Thus
\begin{equation*}
 G_B^{[R]}(y)
 \leq C\frac{y+1}{B^{3/2}}
 \exp\!\left\{-c\left(
   \frac{y^2}{B+y}+\frac{R^2}{B+R}
 \right)\right\}.
\end{equation*}
By the union bound,
\begin{align*}
 \mathcal K_{L,\geq R}^-(d)
 \leq\sum_{\ell=0}^{L}\sum_{H\geq0}
 \bigl[
   F_A^{[R]}(H+d)Z_{B-1}^{\leftarrow}(H)
   +Z_{A-1}^{\rightarrow}(H+d)G_B^{[R]}(H)
 \bigr].
\end{align*}
Since $A,B\leq T$, each term contains the common factor
$\exp\{-cR^2/(T+R)\}$.  After removing this factor, what remains is
exactly the preceding no-cutoff calculation with smaller exponential
constants.  Consequently,
\[
 \mathcal K_{L,\geq R}^-(d)
 \leq CT^{-1/2}
 \exp\!\left\{-c\left(
   \frac{d^2}{T+d}+\frac{R^2}{T+R}
 \right)\right\}.
\]
Since $T=L+2$ and $T+R\leq T+d+R$, this is
\eqref{eq:summed-cut-kernel-tail}, after changing the constants.

The positive jump kernel can be estimated by a similar calculation.\end{proof}

The two-cut excursion block consists of a middle canonical walk and two nonnegative walk bridges ending at 0.
Before gluing, the middle path is a time-reversed ordinary walk with no positivity constraint.

To keep the notation in the right-to-left orientation used in the reconstruction, set
\[
  \check X_{N,j}:=-X_{N,j},\qquad
  \check S_0:=0,\qquad
  \check S_{L}:=\sum_{j=1}^{L}\check X_{N,j},
  \qquad
  \check p_{L}(x):=\mathbb P(\check S_{L}=x).
\]
Thus $\check p_{L}$ is the fixed-endpoint transition mass of the middle path; in the full three-piece convolution its displacement is summed through the two gluing heights.  
Since reflection preserves absolute maxima,
Lemma~\ref{lem:unconditioned-local-maximum} gives the reflected version of
\eqref{eq:unconditioned-maximum-bound}:
\begin{equation}
  \sup_{x\in\mathbb Z}
  \mathbb P\!\left(
    \check S_{L}=x,\ \max_{0\leq j\leq L}|\check S_j|\geq R
  \right)
  \leq C(L+1)^{-1/2}
  \exp\!\left\{-\frac{cR^2}{L+R+1}\right\}.
  \label{eq:reflected-unconditioned-maximum-bound}
\end{equation}

For $L\geq0$ and an integer $\delta$, define the two-cut convolution

\begin{equation*}
  \mathcal T_L(\delta)
  :=\sum_{\ell_0+\ell_1+\ell_2=L}
    \sum_{A,B\geq0}
      Z_{\ell_0}^{\rightarrow}(A)
      \check p_{\ell_1}(A-B+\delta)
      Z_{\ell_2}^{\leftarrow}(B).
\end{equation*}

Here $\delta$ is the signed total displacement of the cross-excursion
jumps, with a positive jump carrying positive sign.  We expose the middle
canonical walk from right to left, as in the reconstruction above.  Its
right-to-left displacement is $A-B+\delta$, so its transition mass is
$\check p_{\ell_1}(A-B+\delta)$.

Dropping the nonnegativity restrictions on the translated
middle profile gives an upper bound of the canonical walk weight by the convolution $\mathcal T_L(\delta)$.

\begin{lemma}[Two-cut kernel]
\label{lem:two-cut-kernel}
Uniformly in $L\ge 0$ and $\delta\in\mathbb Z$,
\begin{equation}
  \mathcal T_L(\delta)\leq C(L+1)^{1/2}.
  \label{eq:two-cut-kernel-bound}
\end{equation}
In the pathwise expansion of $\mathcal T_L(\delta)$, restrict the sum
to configurations satisfying at least one of the following conditions:
$A\geq R$, $B\geq R$, or the unrestricted middle
walk $W^{\mathrm c}$ satisfies
$\operatorname{osc}(W^{\mathrm c})\geq R$.  The restricted sum is
bounded by
\begin{equation}
  C(L+1)^{1/2}
  \exp\!\left\{-\frac{cR^2}{L+R+1}\right\}.
  \label{eq:two-cut-kernel-tail}
\end{equation}
\end{lemma}

\begin{proof}
Summing \eqref{eq:forward-relative-maximum-tail} with $R=0$ over the endpoint and using
\eqref{eq:KX-free-to-zero-tail} with $R=0$ give

\[
  \sum_{A\geq0}Z_{\ell_0}^{\rightarrow}(A)
  \leq C(\ell_0+1)^{-1/2},
  \qquad
  \sum_{B\geq0}Z_{\ell_2}^{\leftarrow}(B)
  \leq C(\ell_2+1)^{-1/2}.
\]

Together with \eqref{eq:reflected-unconditioned-maximum-bound} with $R=0$, this bounds the
summand after the $A,B$ sums by
\[
  C\prod_{i=0}^2(\ell_i+1)^{-1/2}.
\]
Summing over the three-part compositions of $L$ gives
\eqref{eq:two-cut-kernel-bound}.  For the tail, split according to which of the three pieces realizes the
relevant event and repeat the argument in
the proof of Lemma~\ref{lem:summed-cut-kernel}, using
\eqref{eq:forward-relative-maximum-tail} and
\eqref{eq:KX-free-to-zero-tail} for the two outer bridges and
\eqref{eq:reflected-unconditioned-maximum-bound} for the middle walk. This adds the exponential term to the bound and gives \eqref{eq:two-cut-kernel-tail}, since every piece
has length at most $L$.
\end{proof}

Put
\begin{equation*}
  \mathsf H_N(\Gamma):=\frac{M(\Gamma)}{ N^{1/3}},
  \qquad
  \widehat{\mathsf H}_N(\Gamma)
  :=\frac{\widehat M(\Gamma)}{ N^{1/3}}.
\end{equation*}
For $r\geq0$ define
\begin{equation*}
  T_r:=\max\left\{1,
    \left\lfloor\frac{r}{\sqrt{\log(r+2)}}\right\rfloor\right\}.
\end{equation*}

\begin{proposition}[Uniform decorated-path bound]
\label{prop:uniform-decorated-path-bound}
Fix $b>-1$ and let $\boldsymbol k$ range over a compact subset of
$\mathbb R_{>0}^m$.  Let also $\boldsymbol\tau$ range
over a compact subset of
$\{\tau_1\geq\cdots\geq\tau_m\}$.  There exist constants $C,c>0$ and $N_0\geq1$ such that for all
$N\geq N_0$, $r,\rho\geq0$, and $h\geq0$,
\begin{align}
  &\sum_{\substack{\Gamma\in\mathfrak C_N(r,\rho)\\
                    \mathsf H_N(\Gamma)\geq h}}
     |w_N(\Gamma)|
  \notag\\
  &\qquad\leq
  \frac{C^{r+\rho+1}(r+\rho+1)^{4m}}{\sqrt{\rho!}}
  \left
    \{\frac1{T_r!}
      +\left(\frac{C}{\sqrt{\log(r+2)}}\right)^r
  \right\}[e^{-ch^2}+\exp(-chN^{1/3})].
  \label{eq:new-uniform-decoration-tail}
\end{align}
Consequently,
\begin{equation}
  \sup_{N\geq N_0}\sum_{r,\rho\geq0}
  \sum_{\Gamma\in\mathfrak C_N(r,\rho)}|w_N(\Gamma)|<\infty,
  \label{eq:new-absolute-uniform-summability}
\end{equation}
and the tails in $r+\rho$ vanish uniformly in $N\geq N_0$.
The constants and both conclusions are uniform in
$\boldsymbol\tau$ on the prescribed compact set.
\end{proposition}

\begin{proof}
We start with the expression \eqref{eq:weight-of-configuration}, and recall that we may forget the artificial arc labels: the factor $1/r!$
 exactly cancels the $r!$ labelings of
each unlabeled jump pair. Since $\theta$ is fixed and, for $N$ sufficiently large,
the parameters $\mu_{p,N}$ range over a fixed compact subset of
$(0,\infty)$, all factors involving only $\theta$, $b$, and the
$\mu_{p,N}$ may be absorbed into $C^{r+\rho+1}$, uniformly in $p,N$
and locally uniformly in $\boldsymbol\tau$.  More precisely, after
the exponential tilt \eqref{eq:exponential-tilt-identity} is applied
to the canonical pieces, the contribution of the decoration factors
in the second line of \eqref{eq:weight-of-configuration} is bounded in absolute value by
\begin{equation}
  C^{r+\rho+1} N^{-2r-\rho}
  \prod_{j=1}^r d_j
  \prod_{q=1}^{\rho} H_q.
  \label{eq:decoration-weight-crude-bound}
\end{equation}
The remaining product of the canonical-step factors is exactly the product of the corresponding centered random-walk probabilities times the residual factor \(\Xi_N(\Gamma;\boldsymbol\tau)\), which is bounded by one by \eqref{eq:finite-N-cross-time-factor-bound}. Thus the random-walk
part of the weight is controlled by the uniform walk estimates, and
\eqref{eq:decoration-weight-crude-bound} contains all additional
factors that must be accounted for in the combinatorial summation
below. Constants depending on $b$, $m$, and the compact
set of $\boldsymbol k$ and $\boldsymbol\tau$ are denoted by $C$ and may change from line to
line.  Put
\[
  K:=\widetilde K_1+\cdots+\widetilde K_m\asymp N^{2/3}.
\]

Here and below $N\geq N_0$, where $N_0$ is chosen uniformly for
$\boldsymbol k$ and $\boldsymbol\tau$ in the prescribed compact sets. 

\smallskip
\noindent\emph{Step 1: positions, signs, and pairings.}
For each excursion block there are seven boundary patterns: none; one
boundary occurrence of either sign; or two boundary occurrences with any
of the four ordered sign pairs.  Thus there are at most $7^m$ patterns.
Fix one such pattern $\mathfrak t$, with $s$ boundary occurrences, of
which $s_+$ are positive and $s_-$ are negative.  At this stage,
$\mathfrak t$ records only the number and signs of the boundary occurrences in each excursion block, and their positions will be
summed in Step~2.  The final sum
over $\mathfrak t$ costs only the constant factor $7^m$.  In the combinatorial counts below we allow every pairing compatible with
the prescribed signs, without requiring the boundary occurrences to be
the leftmost and rightmost cross-excursion jumps.  The walk sums in
Step~2 still run only over configurations in which they are the actual
extreme cross-excursion occurrences; forgetting that condition only in
the combinatorial count can only enlarge the bound.  A pattern with $s>2m$, $s_+>m$,
or $s_->m$ contributes nothing, so henceforth we restrict to the
complementary case.  After removing the boundary jumps, let $N_p$ be the number of ordinary
jumps in the $p^{\mathrm{th}}$
excursion block.  Since $\sum_pN_p=2r-s$, Vandermonde's identity gives
\begin{equation*}
  \sum_{N_1+\cdots+N_m=2r-s}\prod_{p=1}^m\binom{\widetilde K_p}{N_p}
  =\binom K{2r-s}
  \leq\frac{K^{2r-s}}{(2r-s)!}.
\end{equation*}
The unordered horizontal positions contribute
\begin{equation*}
  \binom K\rho\leq\frac{K^\rho}{\rho!}.
\end{equation*}
For each boundary occurrence, we also record its position in the ordering
with the ordinary deleted positions lying in the same gap between two
surviving canonical steps.  There are at most $2r+\rho+1$ possibilities.
Once this relative ordering and the surviving cut-time coordinate or
coordinates are fixed, the original position of every boundary occurrence
is fixed.  We ignore collisions, which only enlarges the sum.
Thus, before the cut times are summed in Step~2, the position factor is at
most
\begin{equation}
  \frac{C(r+\rho+1)^{2m}K^{2r-s+\rho}}
       {(2r-s)!\rho!}.
  \label{eq:new-total-position-count}
\end{equation}

Fix the ordinary positions and the relative-ordering data, and read the ordinary jumps from right to left.
Split at every excursion-block boundary and at every boundary time, and
let $J\leq3m$ be the resulting number of nonempty sign-word segments.
Enumerate the segments from right to left by $j=1,\ldots,J$.
Inside each segment, read the ordinary-jump sign word from right to left
and group its alternating maximal positive and negative runs into slots
\[
 a_{\nu_{j-1}+1},b_{\nu_{j-1}+1},
 a_{\nu_{j-1}+2},b_{\nu_{j-1}+2},\ldots,
 a_{\nu_j},b_{\nu_j},
 \qquad \nu_0=0.
\]
Here $\nu_j$ is the cumulative number of block slots through the $j^{\mathrm{th}}$
segment; in particular, it is unrelated to the cut time $\ell_p$ of an
excursion block.  The variable $a_i$ is the length of a maximal positive run and $b_i$ is
the length of the maximal negative run immediately to its left.  Both
blocks consist only of ordinary jumps; boundary jumps are not
counted in either variable.  All interior displayed blocks are
nonempty.  If a segment begins on its right boundary with a negative run,
set the first $a_i$ in that segment equal to zero; if it ends on its left
boundary with a positive run, set the last $b_i$ equal to zero.  These are
bookkeeping zeros, not additional blocks.  With this convention every local
signed balance below is an unambiguous sum over the common index interval
$\{\nu_{j-1}+1,\ldots,\nu_j\}$.

For a negative run of length $b_k$, let $U_k$ be the number of unused
ordinary positive jumps to its right.  At most $2m$ boundary positive jumps are also available.  Therefore the number $n_k$ of
ways to pair that block is bounded by
\begin{equation*}
  n_k\leq\widetilde n_k:=(U_k+2m)_{b_k}
  :=(U_k+2m)(U_k+2m-1)\cdots(U_k+2m-b_k+1).
\end{equation*}
If $C_j^{\mathrm{in}}$ denotes the number of unused ordinary positive
endpoints carried into the $j^{\mathrm{th}}$ segment from segments to its right, let
$E_{j,k}$ be the number of negative jumps in the earlier negative runs of the
$j^{\mathrm{th}}$ segment that are paired with boundary positive occurrences.
Then $0\leq E_{j,k}\leq2m$, and, for
$\nu_{j-1}<k\leq\nu_j$,
\begin{equation}
 U_k=C_j^{\mathrm{in}}
 +\sum_{i=\nu_{j-1}+1}^{k}a_i
 -\sum_{i=\nu_{j-1}+1}^{k-1}b_i
 +E_{j,k}.
 \label{eq:open-count-with-carry}
\end{equation}
The correction $E_{j,k}$ is necessary because a negative endpoint paired
with a boundary positive occurrence does not consume an ordinary
positive endpoint.  Across all segments, the total number of such
corrections is at most $2m$.
The at most $2m$ boundary negative occurrences have at most
$(2r+1)^{2m}$ possible partners.  For later use, we also record the crude
bound
\begin{equation}
  \#\{\text{all compatible pairings}\}
  \leq(2r-1)!!=\frac{(2r)!}{2^r r!}.
  \label{eq:new-crude-pairing-bound}
\end{equation}
Here and below we use the convention $(-1)!!=1$.  The prescribed
boundary signs can only reduce this number.

\smallskip
\noindent\emph{Step 2: the walk partition functions.}
Fix the positions of the ordinary jumps, the marked horizontal
positions,
and the relative-ordering data from Step~1, together with the jump signs, a
compatible pairing, and all pair heights; call this \emph{fixed data}
$\mathscr D$.  The values of the pair-height factors $d_e$ and of the
marked-height factors are not included in the walk partition function
below; they are summed in Step~3.  Here $\operatorname{wt}_{\mathrm{can}}(\Gamma)$ denotes the product of the probability weights of all auxiliary canonical pieces.  For $q\in\mathbb Z_{\geq0}$, define
\begin{equation}
  \mathcal Z_N^{\mathrm{walk}}(\mathscr D;q)
  :=\sum_{\substack{\Gamma:\,\text{fixed data }\mathscr D\\
          q\leq\widehat{\mathsf H}_N(\Gamma)<q+1}}
       \operatorname{wt}_{\mathrm{can}}(\Gamma),
  \label{eq:fixed-data-walk-partition}
\end{equation}
where the sum is over the auxiliary canonical walks and over the one or
two surviving cut times in every cut excursion block.  Thus
\eqref{eq:fixed-data-walk-partition} is precisely the total
canonical-step weight with all discrete decoration data except the cut
times held fixed.

\begin{claim}
Uniformly in the fixed data for which the slab is nonempty,
\begin{equation}
  \mathcal Z_N^{\mathrm{walk}}(\mathscr D;q)
  \leq C(r+\rho+1)^{3m/2}K^{-3m/2+s}\exp\left(-c\frac{q^{2}K}{K+q\sqrt{K}}\right).
  \label{eq:new-walk-partition-bound}
\end{equation}
\end{claim}

\begin{proof}[Proof of the claim]
Consider first the unrestricted sum, without the height slab.  If
$c_p=0$, the canonical part of the $p^{\mathrm{th}}$ excursion block is one nonnegative
$0$-to-$0$ walk of length $L_p$, and its total weight is
\[
  Z_{L_p}(0,0)\leq C(L_p+1)^{-3/2}.
\]
If $c_p=1$, summing the unique surviving cut time and the lower adjacent
ordinate gives exactly $\mathcal K_{L_p}^{\pm}(d_p)$.  Therefore
Lemma~\ref{lem:summed-cut-kernel}, with its exponential factor discarded,
gives
\[
  \sum_{\ell_p=0}^{L_p}\sum_{H_p\geq0}
  Z_{\ell_p}^{\rightarrow}
    \!\left(H_p+\mathbf 1_{\{\mathfrak{s}_p=-\}}d_p\right)
  Z_{L_p-\ell_p}^{\leftarrow}
    \!\left(H_p+\mathbf 1_{\{\mathfrak{s}_p=+\}}d_p\right)
  \leq C(L_p+1)^{-1/2}.
\]
The displayed quantity is the summed canonical-walk weight of this
one-cut excursion block with the jump height fixed.

Suppose $c_p=2$.  Once all cross-excursion jump heights and signs are
fixed, let $\delta_p$ be their signed total displacement as in
\eqref{eq:signed-cross-jump-displacement}.  
For fixed outer endpoint heights $A,B$ and fixed three-part lengths,
the raw middle walk has right-to-left displacement $A-B+\delta_p$ and
therefore contributes
\[
  \check p_{\ell_p^{\mathrm c}}(A-B+\delta_p).
\]  After summing $A,B$ and
the two cut times, the total canonical-walk weight is bounded by
$\mathcal T_{L_p}(\delta_p)$ and hence by
\[
  C(L_p+1)^{1/2}
\]
by Lemma~\ref{lem:two-cut-kernel}.  The conditions
$H_p\geq\mathfrak m_p$ and \eqref{eq:three-piece-feasibility}, as well
as the final nonnegativity conditions after the same-excursion pairs are
lowered, only restrict this sum.

Thus an excursion block with $c_p$ cuts contributes
$C(L_p+1)^{-3/2+c_p}$.  Since $\widetilde K_p=L_p+D_p$ and
$D_p\leq2r+\rho$,
\[
  (L_p+1)^{-\alpha}
  \leq(D_p+1)^\alpha(\widetilde K_p+1)^{-\alpha},
  \qquad \alpha\in\{1/2,3/2\}.
\]
The positive exponent occurring when $c_p=2$ is bounded directly by
$L_p+1\leq \widetilde K_p+1$.  The sum of the negative exponents $\alpha$ over all excursion blocks is at
most $3m/2$.  Multiplication over the excursion blocks therefore gives
\begin{equation}
  C(r+\rho+1)^{3m/2}K^{-3m/2+s}.
  \label{eq:new-unrestricted-walk-scale}
\end{equation}
This explains both the object controlled by the bound and every power of
$K$: each of the $s$ surviving cut-time sums raises the uncut excursion block
scale by one power of $K$.

It remains to impose the slab. Because $p$ attains
$\max_j\mathcal R_j$, Lemma~\ref{lem:extreme-cut-height-propagation} gives
$\widehat M(\Gamma)\leq C\mathcal R_p$.  Then a configuration in the slab
with $q\geq1$ has
\[
  \mathcal R_p\geq c qK^{1/2}
\]
for an excursion block $p$ attaining $\max_j\mathcal R_j$.  If $c_p=0$, apply
\eqref{eq:KX-zero-to-zero-tail}; if $c_p=1$, apply
\eqref{eq:summed-cut-kernel-tail}; and if $c_p=2$, apply
\eqref{eq:two-cut-kernel-tail}.  Use the corresponding unrestricted
estimate in every other excursion block.
A union bound over the $m$ excursion blocks therefore multiplies
\eqref{eq:new-unrestricted-walk-scale} by $C\exp\left(-c\frac{q^{2}K}{K+q\sqrt{K}}\right)$.  The case $q=0$
follows from the unrestricted estimate.  This proves
\eqref{eq:new-walk-partition-bound}.
\end{proof}

\smallskip
\noindent\emph{Step 3: deterministic height bounds.}
Choose the fixed constant $C$ large enough and set
\[
  R_q:=\left\lceil C(q+1)K^{1/2}\right\rceil.
\]
Since $K^{1/2}\asymp N^{1/3}$, the slab
$q\leq\widehat{\mathsf H}_N<q+1$ implies
\[
  \widehat M(\Gamma)
  <(q+1)\sigma N^{1/3}\leq R_q.
\]
By \eqref{eq:new-max-comparison}, every pair height and every marked
horizontal height is therefore at most $R_q$ on this slab.

For each pair, summing its height over $d\in\{1,\ldots,R_q\}$ and
retaining its numerical height factor costs at most $CR_q^2$; each marked horizontal factor costs at
most $CR_q$.  Thus the crude total height factor is
\begin{equation}
  C^{r+\rho+1}R_q^{2r+\rho}
  \leq C^{r+\rho+1}
       K^{r+\rho/2}(q+1)^{2r+\rho}.
  \label{eq:new-crude-height-factor}
\end{equation}

We now extract the two factorial improvements used below.  Both are
consequences of the same deterministic observation.  Suppose that
$u$ ordinary positive occurrences lie to the right of a reference
position and that their negative partners all lie to its left.  Here a
reference position is only a bookkeeping location between two surviving
canonical steps; it is not one of the reconstruction cut times
$\chi_p,\chi_p'$.  Order the corresponding pair heights
$d_1,\ldots,d_u$ by reading the ordinary positive occurrences from
right to left.  Crossing the positive occurrence of the $i^{\mathrm{th}}$ ordinary pair makes its lowering by $d_i$ effective.  Since the corresponding
negative occurrence lies to the left of the reference position, that
lowering remains in force until the reference position is reached.
Hence, after the first $i$ ordinary positive occurrences have been
crossed, their cumulative lowering is $d_1+\cdots+d_i$.

\begin{lemma}[Cumulative height of active pairs]
\label{lem:active-pair-cumulative-height}
Fix a sign-word segment $\mathcal S_j$ contained in the $p^{\mathrm{th}}$ excursion
block, and let $t$ be a reference position in $\mathcal S_j$.
Let $\mathcal O_j(t)$ be any collection of pairs $e$ such that
\begin{enumerate}[label=\textnormal{(\roman*)}]
  \item the positive occurrence $t_+(e)$ is an ordinary jump belonging to
  $\mathcal S_j$;
  \item
  \[
    t_-(e)<t<t_+(e),
  \]
  where $t_-(e)$ is the negative occurrence of $e$.
\end{enumerate}
Then
\begin{equation*}
  \sum_{e\in\mathcal O_j(t)}d_e
  \leq C_m\widehat M(\Gamma),
\end{equation*}
where $C_m$ depends only on $m$.
\end{lemma}

\begin{proof}
Split $\mathcal O_j(t)$ according to whether the two occurrences lie in
the same excursion block or in different excursion blocks, and denote the
corresponding total heights by $D_{\mathrm{same}}$ and
$D_{\mathrm{cross}}$.

Let $\widehat Y(t)$ be the auxiliary-skeleton ordinate at $t$, after all
cross-excursion jumps have been inserted but before the same-excursion
interval lowerings are applied. For every same-excursion pair under
consideration, the lowering interval contains $t$. Nonnegativity of the
final decorated path therefore gives
\begin{equation*}
  D_{\mathrm{same}}
  \leq\widehat Y(t)
  \leq\widehat M(\Gamma).
\end{equation*}

For a cross-excursion pair in $\mathcal O_j(t)$, its positive occurrence
lies in excursion block $p$ and its negative occurrence lies in an earlier
block. It therefore passes through the interface $(p-1)\mid p$. By
\eqref{eq:cross-excursion-flow-identity},
\begin{equation*}
  D_{\mathrm{cross}}
  \leq F_{p-1}
  \leq3(p-1)\max_a\mathcal R_a
  \leq3m\widehat M(\Gamma),
\end{equation*}
where the last inequality follows directly from the definitions of
$\mathcal R_a$ and $\widehat M$. Combining the two bounds proves the
claim, with $C_m=3m+1$.
\end{proof}
For each $i\leq u$, place the reference position immediately after the
first $i$ positive occurrences have been crossed in the right-to-left
order.  Those $i$ pairs are active there, so the lemma and positivity of
the $d_i$ give
\begin{equation*}
  0<d_1<d_1+d_2<\cdots<d_1+\cdots+d_u
  \leq C_m\widehat M(\Gamma).
\end{equation*}
On the $q^{\mathrm{th}}$ slab these partial sums are strictly increasing positive
integers bounded by $C_mR_q$.  The map from
$(d_1,\ldots,d_u)$ to its sequence of partial sums is injective, and
therefore
\begin{equation}
  \#\{\text{admissible assignments of these $u$ heights}\}
  \leq\binom{C_mR_q}{u}
  \leq\frac{(C_mR_q)^u}{u!}.
  \label{eq:active-height-composition-count}
\end{equation}

\begin{claim}[A single positive run]
Fix a positive run of length $a_k$.  Conditional on all other discrete data and heights, the
number of assignments of the $a_k$ pair heights that can occur in the
$q^{\mathrm{th}}$ slab is at most
\begin{equation*}
  \frac{(C_mR_q)^{a_k}}{a_k!}.
\end{equation*}
\end{claim}

\begin{proof}
Choose the reference position between the surviving canonical steps
immediately to the left of the block.  All $a_k$ ordinary positive
occurrences lie to the right of this reference position, whereas their
negative partners lie to its left.
Applying \eqref{eq:active-height-composition-count} with $u=a_k$ proves
the claim.
\end{proof}

Thus, relative to the crude factor $R_q^{a_k}$, the
bound contains $C_m^{a_k}/a_k!$.  Since $a_k\leq r$, the factor
$C_m^{a_k}$ is absorbed into the ambient $C^r$, while the factorial gain
$1/a_k!$ is unchanged.

For the second situation, fix a genuine positive run of length $a_{k'}>0$ in the
right-to-left processing of the $j^{\mathrm{th}}$ sign-word segment, and set
\[
  R_{j,k'}:=
  \sum_{i=\nu_{j-1}+1}^{k'-1}a_i
  -\sum_{i=\nu_{j-1}+1}^{k'-1}b_i,
\]
\begin{equation*}
  S_{j,k'}:=a_{k'}+\max\{R_{j,k'},0\}.
\end{equation*}
The quantity $R_{j,k'}$ records only the ordinary positive occurrences
created earlier in the current segment and not yet matched there; it does
not include heights carried into the segment from the right.

\begin{claim}[A positive run with inherited open heights]
Assume $R_{j,k'}>0$.  Select the $a_{k'}$ heights created by the current
positive run and, in increasing order of their positive jump times, the
first $R_{j,k'}$ still-open heights created earlier in the same segment.
This is a canonical set of $S_{j,k'}$ heights, and,
conditional on all other discrete data and heights, its number of
assignments in the $q^{\mathrm{th}}$ slab is at most
\begin{equation*}
  \frac{(C_mR_q)^{S_{j,k'}}}{S_{j,k'}!}.
\end{equation*}
\end{claim}

\begin{proof}
Among the earlier local positive occurrences, each intervening negative
occurrence can match at most one height.  At least $R_{j,k'}$ of the
earlier heights therefore remain open, so the selection in the claim is
well defined and canonical.  At the reference position immediately to the left of the current
positive run, these inherited occurrences and the $a_{k'}$ new
occurrences all lie to the right, whereas every one of their negative
partners lies to the left.  Thus
\eqref{eq:active-height-composition-count} applies with
$u=S_{j,k'}$ and proves the claim.
\end{proof}

If $R_{j,k'}\leq0$, only the $a_{k'}$ newly created heights are used,
and the first claim applies.  No independence between auxiliary pieces
is used in either argument.  Consequently, if some positive run satisfies
$a_k\geq T_r$, or if some segment-level count satisfies
$S_{j,k'}\geq2T_r$, then, after the additional factor of at most $C_m^r$ is absorbed into the ambient $C^r$, the crude height factor
\eqref{eq:new-crude-height-factor} retains the factorial gain
\begin{equation}
  \frac1{T_r!}.
  \label{eq:case-factorial-saving}
\end{equation}

\smallskip

\noindent\emph{Step 4: the three combinatorial regimes.}
We partition all sign words and pairings into the following three cases.

\begin{enumerate}[label=\textnormal{(\Roman*)}]
  \item Some positive run has length $a_k\geq T_r$.
  \item Every positive run has length $<T_r$, but
  $S_{j,k'}\geq2T_r$ for some $j,k'$.
  \item Neither of the preceding alternatives occurs.
\end{enumerate}

In Cases (I) and (II) we use the crude pairing bound
\eqref{eq:new-crude-pairing-bound}.  By the two claims in Step~3, the
height sum gains the factor $1/T_r!$.

In Case (III) we use the following deterministic pairing estimate.

\begin{claim}[Pairing bound in the small-run regime]
The sign and pairing multiplicity in Case \textnormal{(III)} satisfies
\begin{equation}
  \#\bigl\{(\text{signs of ordinary jumps},\text{ compatible pairing})
      :\text{Case \textnormal{(III)}}\bigr\}
  \leq C^r(r+1)^{2m}T_r^r.
  \label{eq:new-refined-pairing-bound}
\end{equation}
\end{claim}

\begin{proof}[Proof of the claim]

Let a negative run of length $b_k$ be processed in one of the sign-word
segments.  Its unused positive partners consist of the ordinary positives
carried into that segment, the ordinary positives created locally, and at
most $2m$ boundary positive occurrences.

Immediately before a positive run of length $a_{k'}$ in the $j^{\mathrm{th}}$ segment, its local
signed balance is $R_{j,k'}$.  Immediately after that block, the positive
part of the local contribution is bounded by
\[
 (R_{j,k'}+a_{k'})_+
 \leq a_{k'}+(R_{j,k'})_+
 =S_{j,k'}.
\]
The maximal positive local signed surplus is attained immediately after a
positive run.  In Case~\textnormal{(III)}, this surplus is less than
$2T_r$ in every segment.  Relative to this signed balance, pairing a
negative endpoint with a boundary positive occurrence can add back one
ordinary open endpoint; by \eqref{eq:open-count-with-carry}, there are at
most $2m$ such corrections over the entire sign word.  Since
$J\leq3m$, induction from the rightmost segment to the leftmost therefore
gives at most $6mT_r+2m$ available ordinary positive jumps at every stage.
Adding the at most $2m$ still-available boundary positive occurrences
and using $T_r\geq1$ yields
\begin{equation*}
  U_k+2m\leq10mT_r.
\end{equation*}
Consequently,
\[
  \prod_k\widetilde n_k
  \leq(10mT_r)^{\sum_kb_k}
  \leq(10mT_r)^r.
\]

Once the positions of the ordinary jumps are fixed, assigning a
sign to each of the $2r-s$ ordinary jumps determines uniquely, after cutting at
excursion block boundaries and boundary occurrences, the number of blocks and
all nonzero block lengths
\[
  b_{\nu_J},a_{\nu_J},\ldots,b_1,a_1.
\]
Indeed, these are precisely the maximal negative and positive runs in the
resulting sign word.  The number of admissible sign assignments is at most
\begin{equation*}
  \binom{2r-s}{r-s_+}\leq 2^{2r}.
\end{equation*}
Thus summing over all choices of $\nu_J$ and
$b_{\nu_J},a_{\nu_J},\ldots,b_1,a_1$ contributes only this additional
exponential factor.  The at most $2m$ boundary negative occurrences
have at most $(2r+1)^{2m}$ choices of positive partners.  Combining these
bounds with the preceding estimate for $\prod_k\widetilde n_k$, and
absorbing the fixed powers of $2$ and $m$ into $C$, proves
\eqref{eq:new-refined-pairing-bound}.
\end{proof}

This is the only case in which the refined pairing estimate is needed.

\smallskip
\noindent\emph{Step 5: combine the scales and sum $q$.}
Let $\mathcal P_r$ denote the sign-and-pairing upper bound used in the
case under consideration, and let $\mathcal G_r$ denote the multiplicative
factorial-saving factor applied to the crude height count from Step~3.
Thus $\mathcal G_r=1$ in Case~\textnormal{(III)}, whereas in Cases
\textnormal{(I)} and \textnormal{(II)}, by
\eqref{eq:case-factorial-saving},
\[
  \mathcal G_r\leq\frac1{T_r!}.
\]
Steps 1--3 give, for each of the three cases and each $q\geq0$,
\begin{align}
  &\sum_{\substack{\Gamma\in\mathfrak C_N(r,\rho),\ 
                    \Gamma\text{ has boundary pattern }\mathfrak t,\ 
                    \Gamma\text{ in the case}\\
          q\leq\widehat{\mathsf H}_N(\Gamma)<q+1}}
     |w_N(\Gamma)|
  \notag\\
  &\quad\leq
  \frac{C^{r+\rho+1}(r+\rho+1)^{7m/2}}
       {(2r-s)!\rho!}
  \mathcal P_r\mathcal G_r
  (q+1)^{2r+\rho}\exp\left(-c\frac{q^{2}K}{K+q\sqrt{K}}\right)
  \notag\\
  &\qquad\times
  N^mN^{-2r-\rho}
  K^{2r-s+\rho}K^{r+\rho/2}K^{-3m/2+s}.
  \label{eq:new-master-slab-bound}
\end{align}
The three powers of $K$ come, respectively, from the ordinary decorated
positions, the height sums and numerical height factors, and the walk
partition functions.  The missing position factor $K^s$ is supplied by
the cut-time sums inside the walk kernels, which replace
$K^{-3m/2}$ by $K^{-3m/2+s}$; equivalently, the two displayed powers of
$s$ cancel.  Since $K\asymp N^{2/3}$, the last line of
\eqref{eq:new-master-slab-bound} is bounded uniformly in $N$.

By \eqref{eq:new-max-comparison},
$\{\mathsf H_N\geq h\}\subseteq
 \{\widehat{\mathsf H}_N\geq h\}$.  Hence the relevant slabs satisfy
$q\geq\lfloor h\rfloor$. For $\lfloor h\rfloor\le \sqrt{K}$, we have
\begin{equation}
\begin{split}
  &\sum_{q\geq\lfloor h\rfloor}
       (q+1)^{2r+\rho}\exp\left(-c\frac{q^{2}K}{K+q\sqrt{K}}\right)\\
  \leq&
  \sum_{\lfloor h\rfloor\le q\le \sqrt{K}}
       (q+1)^{2r+\rho}\exp\left(-c'q^{2}\right)+\sum_{q> \sqrt{K}}
       (q+1)^{2r+\rho}\exp\left(-cq\sqrt{K}\right)\\
       \le &C^{r+\rho+1}\Gamma\!\left(r+\frac\rho2+1\right)e^{-c'h^2}+\frac{C^{2r+\rho}\Gamma(2r+\rho+1)}{K^{r+\frac{\rho}{2}}}\exp(-c'K)\\
       \le &C^{r+\rho+1}\Gamma\!\left(r+\frac\rho2+1\right)[e^{-c'h^2}+\exp(-c'h\sqrt{K})],
       \label{eq:new-Gaussian-moment-bound}\end{split}\end{equation}
where the last inequality holds since $2r+\rho\le K$. For $\lfloor h\rfloor>\sqrt{K}$, only the second sum above survives so the same bound holds.

We shall also use log-convexity of the Gamma function in the form
\begin{equation}
  \Gamma\!\left(r+\frac\rho2+1\right)
  \leq\sqrt{(2r)!\,\rho!},
  \label{eq:new-Gamma-convexity}
\end{equation}
and the elementary estimate
\begin{equation*}
  \frac{\sqrt{(2r)!}}{r!}\leq2^r.
\end{equation*}

In Cases \textnormal{(I)} and \textnormal{(II)}, take
$\mathcal P_r=(2r-1)!!$.  Since $s\leq2m$,
\[
  \frac{(2r-1)!!}{(2r-s)!}
  \frac{\Gamma(r+\rho/2+1)}{\rho!}
  \leq
  \frac{C^r(r+1)^{2m}}{\sqrt{\rho!}}.
\]
Using \eqref{eq:case-factorial-saving} and absorbing its polynomial loss,
we obtain the explicit bound
\begin{equation}
  \sum_{\substack{\Gamma\in\mathfrak C_N(r,\rho)\text{ in Case
                   \textnormal{(I)} or \textnormal{(II)}}\\
                    \mathsf H_N(\Gamma)\geq h}}
       |w_N(\Gamma)|
  \leq
  \frac{C^{r+\rho+1}(r+\rho+1)^{4m}}
       {\sqrt{\rho!}\,T_r!}[e^{-c'h^2}+\exp(-c'h\sqrt{K})].
  \label{eq:new-large-block-bound}
\end{equation}

In Case \textnormal{(III)}, use
\eqref{eq:new-refined-pairing-bound} to take
$\mathcal P_r=C^r(r+1)^{2m}T_r^r$, together with $\mathcal G_r=1$ and
\eqref{eq:new-Gamma-convexity}.  Since $s\leq2m$,
\begin{align*}
  &\frac{C^r(r+1)^{2m}T_r^r}{(2r-s)!}
   \frac{\Gamma(r+\rho/2+1)}{\rho!}\\
  &\qquad\leq
  \frac{C^r(r+1)^{4m}}{\sqrt{\rho!}}
  \frac{T_r^r}{\sqrt{(2r)!}}
  \leq
  \frac{C^r(r+1)^{4m}}{\sqrt{\rho!}}
  \left(\frac{C}{\sqrt{\log(r+2)}}\right)^r,
\end{align*}
where the last inequality follows from Stirling's formula and the
definition of $T_r$.  Substitution into
\eqref{eq:new-master-slab-bound}, followed by
\eqref{eq:new-Gaussian-moment-bound}, yields
\begin{equation*}
  \sum_{\substack{\Gamma\in\mathfrak C_N(r,\rho)\text{ in Case
                   \textnormal{(III)}}\\
                    \mathsf H_N(\Gamma)\geq h}}
       |w_N(\Gamma)|
  \leq
  \frac{C^{r+\rho+1}(r+\rho+1)^{4m}}{\sqrt{\rho!}}
  \left(\frac{C}{\sqrt{\log(r+2)}}\right)^r [e^{-c'h^2}+\exp(-c'h\sqrt{K})].
\end{equation*}
The only polynomial powers retained explicitly are the $2m$ powers from
(\ref{eq:new-total-position-count}) in Step~1 and the $3m/2$ powers from the
comparison of $L_p$ with $\widetilde K_p$ in Step~2.  Their sum is $7m/2$, which we
bound by $4m$.  The subsequent factors $(r+1)^{2m}$ and $(r+1)^{4m}$ have fixed degree
and are absorbed into $C^{r+\rho+1}$ (with $C$ depending on $m$).

Together with \eqref{eq:new-large-block-bound}, this proves
\eqref{eq:new-uniform-decoration-tail}.

Finally,
\[
  \sum_{r,\rho\geq0}
  \frac{C^{r+\rho+1}(r+\rho+1)^{4m}}{\sqrt{\rho!}}
  \left
    \{\frac1{T_r!}
      +\left(\frac{C}{\sqrt{\log(r+2)}}\right)^r
  \right\}<\infty.
\]
Indeed, $T_r\log T_r\asymp r\sqrt{\log r}$, while
$\sqrt{\rho!}$ dominates every exponential times a fixed polynomial.
This proves \eqref{eq:new-absolute-uniform-summability} and the uniform
vanishing of the $r+\rho$ tails.
\end{proof}

\subsection{Passage to the Brownian functional}
\label{subsec:passage-to-Brownian-functional}

We now identify the limit of the decorated-path expansion.  We first fix
\(r,\rho\) and a combinatorial type and prove that its discrete contribution
converges to the corresponding Brownian Riemann integral.  The uniform
absolute estimate of Proposition~\ref{prop:uniform-decorated-path-bound}
then permits the sum over \(r,\rho\) and all types to pass through the
limit.

For fixed \(r\) and \(\rho\), a \emph{combinatorial type}
records the excursion block containing each of the two jumps of every
labeled jump pair, the order of the jump occurrences within each excursion block, and
the numbers \(\rho_1,\ldots,\rho_m\) of height marks assigned to the
excursion blocks, where \(\sum_p\rho_p=\rho\).  The positions of those
unordered marks are still summed over; in particular, their possible
interlacings with the fixed jumps are not separated into different
types.  There are only finitely many such types for fixed \(m,r,\rho\).
We keep the jump pair
labels at this stage and the factor \(1/r!\) from
\eqref{eq:weight-of-configuration}.

Fix for the moment
\(\boldsymbol\varepsilon\in\{0,1\}^m\), and use the actual excursion
lengths \(\widetilde K_p=K_p+\varepsilon_p\).  Let \(s_{\pm,N}^j\) be
the two integer positions of the \(j^{\mathrm{th}}\) jump pair, measured from the beginning
of their respective excursion blocks, and let \(d_j\geq1\) be its
height.  Put
\[
  Q_{p,N}^{\boldsymbol\varepsilon}
  :=N^{-2/3}\sum_{q=1}^p\widetilde K_q,
  \qquad Q_{0,N}^{\boldsymbol\varepsilon}:=0.
\]
If the two occurrences lie on excursion blocks \(\ell_-^j\) and
\(\ell_+^j\), respectively, set
\begin{equation}
  t_{\pm,N}^j
  =Q_{\ell_\pm^j-1,N}^{\boldsymbol\varepsilon}
   +N^{-2/3}s_{\pm,N}^j,
  \qquad
  h_N^j=\frac{d_j}{\sigma N^{1/3}}.
  \label{eq:jump-rescaling}
\end{equation}
Since \(\widetilde K_pN^{-2/3}\to k_p\), we have
\(Q_{p,N}^{\boldsymbol\varepsilon}\to Q_p\); hence these are precisely
the global jump-time coordinates used in
\(\mathcal D_{\boldsymbol k}(\boldsymbol\ell)\) in the introduction.
At the negative (respectively positive) occurrence, write
\(H_{-,N}^j\) (respectively \(H_{+,N}^j\)) for the baseline boundary
height appearing in the definition of \(H_-^j,H_+^j\) in the
introduction, divided by \(\sigma N^{1/3}\).  Thus the time variables have
mesh \(N^{-2/3}\), while \(h_N^j\) and every \(H_{\pm,N}^j\) have mesh
\((\sigma N^{1/3})^{-1}\).

We next make the path rescaling explicit.  Recall that the \(p^{\mathrm{th}}\)
excursion block of a decorated configuration is an index path
\[
  \Gamma_p=(i_{p,0},i_{p,1},\ldots,i_{p,\widetilde K_p}),
  \qquad i_{p,0}=i_{p,\widetilde K_p}=0.
\]
Define the right-continuous step process
\begin{equation*}
  \widehat\Gamma_{p,N}(u)
  :=\frac{i_{p,\lfloor uN^{2/3}\rfloor}}{\sigma N^{1/3}},
  \qquad
  0\leq u\leq N^{-2/3}\widetilde K_p.
\end{equation*}
Shifting its time coordinate by
\(Q_{p-1,N}^{\boldsymbol\varepsilon}\) places it on the global time
interval of the \(p^{\mathrm{th}}\) excursion block.  At a paired-jump position this
c\`adl\`ag process has unequal left and right limits.  We therefore use
convergence separately on the canonical pieces between consecutive jump
occurrences.

To define these pieces, order the \(n_p\) jump occurrences in the \(p^{\mathrm{th}}\)
excursion block by their raw integer positions
\[
  0=s_{p,0,N}<s_{p,1,N}<\cdots<s_{p,n_p,N}
  <s_{p,n_p+1,N}=\widetilde K_p.
\]
Delete the jump steps and the marked horizontal steps.  For
\(a=1,\ldots,n_p+1\), let \(W_{p,a,N}\) be the remaining canonical
walk segment between the \((a-1)\)st and \(a^{\mathrm{th}}\) jump occurrences, and
let \(L_{p,a,N}\) be its number of canonical steps.  Put
\[
  \bar s_{p,0,N}=0,\qquad
  \bar s_{p,a,N}
  :=N^{-2/3}\sum_{c=1}^aL_{p,c,N}.
\]
Since only \(2r+\rho\) positions are deleted in the whole configuration,
\begin{equation}
  \max_{p,a}
  \left|\bar s_{p,a,N}-N^{-2/3}s_{p,a,N}\right|
  \leq (2r+\rho)N^{-2/3}.
  \label{eq:canonical-and-raw-clocks}
\end{equation}
Define the right-continuous canonical piece
\[
  \widehat W_{p,a,N}(u)
  :=\frac{W_{p,a,N}(\lfloor uN^{2/3}\rfloor)}
       {\sigma N^{1/3}},
  \qquad 0\leq u\leq N^{-2/3}L_{p,a,N},
\]
and regard it, after translation, as a c\`adl\`ag path on
\([\bar s_{p,a-1,N},\bar s_{p,a,N}]\).  Equation
\eqref{eq:canonical-and-raw-clocks} shows that this canonical clock and
the raw jump-time clock have the same limit.

We now identify the limit of the extra factor
\eqref{eq:finite-N-cross-time-factor}.  For
\begin{equation*}
  z(a):=\frac{\sqrt a}{1+\sqrt a},
\end{equation*}
one has
\begin{equation*}
  \left.\frac{\mathrm d}{\mathrm da}\log z(a)\right|_{a=2}
  =\frac1{4(1+\sqrt2)}.
\end{equation*}
Hence, locally uniformly in the ordered time vector,
\begin{equation*}
  \log\frac{z_{q,N}}{z_{p,N}}
  =\frac{\tau_q-\tau_p}{4(1+\sqrt2)}N^{-1/3}
   +O(N^{-2/3}),
  \qquad p\leq q.
\end{equation*}
If $d_j/(\sigma N^{1/3})\to h^j$, then
\begin{equation}
  \left(\frac{z_{\ell_+^j,N}}
             {z_{\ell_-^j,N}}\right)^{d_j}
  \longrightarrow
  \exp\left\{\tilde{c}
    (\tau_{\ell_+^j}-\tau_{\ell_-^j})h^j\right\}.
  \label{eq:one-cross-pair-time-limit}
\end{equation}
Multiplication over $j$ gives precisely
$\mathbf T_{\boldsymbol\tau}$ in
\eqref{eq:intro-multitime-factor}.

We repeatedly use the following two elementary calculations.  The weight of a
jump pair, considering the two sums over jump time and the sum over \(d_j\), under  proper scaling becomes

\begin{equation}
  N^{4/3}\cdot \sigma N^{1/3}\cdot
  \frac{\sigma h^jN^{1/3}}
       {\theta\mu_{\ell_-^j,N}\mu_{\ell_+^j,N}N^2}
  \longrightarrow\frac{\sigma^2}{\theta\mu_+^2}h^j.
  \label{eq:jump-pair-mesh-calculation}
\end{equation}
For a height mark at the \(q^{\mathrm{th}}\) horizontal position of
\(\Gamma_p\), the recorded height is \(i_{p,q}\).  A proper rescaling of the sum over its position then gives the weight factor

\begin{equation*}
  N^{2/3}\cdot\frac{b\,i_{p,q}}{\mu_{p,N}N}
  =\frac{b\sigma}{\mu_{p,N}}
    \widehat\Gamma_{p,N}(qN^{-2/3})
  =\left(\frac{b\sigma}{\mu_+}+o(1)\right)
    \widehat\Gamma_{p,N}(qN^{-2/3}).
\end{equation*}
These identities account for all powers of \(N\) contributed by the
decorations.  The meshes of the baseline boundary heights must still be
combined with the local limits of the canonical bridge partition
functions.

Fix the discrete jump times and boundary heights.  By \eqref{eq:exponential-tilt-identity}, the canonical-step
weight on \(W_{p,a,N}\), after normalization, is exactly the law of the
centered walk \eqref{eq:process-canonical-step-law} for block $p$ conditioned to stay
nonnegative and to have its prescribed endpoints.  Suppose that its
rescaled length and endpoints converge to
\[
  s_{p,a}-s_{p,a-1},\qquad x_{p,a-1},\qquad y_{p,a},
\]
respectively.  
Let \(\phi_{p,a,N}\) be the affine bijection from
\([s_{p,a-1},s_{p,a}]\) onto
\([\bar s_{p,a-1,N},\bar s_{p,a,N}]\).
Proposition~\ref{prop:conditional-bridge-functional-limit}
then gives
\(J_1\)-convergence of
\(\widehat W_{p,a,N}\circ\phi_{p,a,N}\) in
\(D([s_{p,a-1},s_{p,a}])\) to the nonnegative Brownian bridge
\(B_{p,a}\) from \(x_{p,a-1}\) to \(y_{p,a}\).      Since the canonical-step weights factor over the finitely many
pieces, their joint conditional limit is the product of these bridge
laws.  Proposition~\ref{prop:conditional-bridge-local-limits} gives at
the same time the normalizing transition density of each piece.
According to its boundary data, this density is \(\mathbf F\),
\(\mathbf F_0\), or \(\mathbf F_{0,0}\), exactly as in
\eqref{eq:intro-transition-factor}.

\paragraph{Count of the excursion block constants.}
We spell out this count because it explains both the number of factors
\(P_{-1}\) and the power of \(\sigma\) in the answer.
Suppose first that
the \(p^{\mathrm{th}}\) excursion block contains \(n_p\geq1\) jump occurrences.  The factor
\(N^m\) in \eqref{eq:deformed-moment-action} supplies one factor \(N\)
for this excursion block.  Reading from left to right, the zero-to-positive
piece, the \(n_p-1\) positive-to-positive pieces, and the
positive-to-zero piece contribute, by
\eqref{eq:zero-positive-local-limit},
\eqref{eq:positive-positive-local-limit}, and
\eqref{eq:positive-zero-local-limit}, respectively,
\[
  N^{-2/3}\sigma^{-2},\qquad
  \bigl(N^{-1/3}\sigma^{-1}\bigr)^{n_p-1},
  \qquad
  N^{-2/3}(2P_{-1})^{-1}.
\]
There is one freely summed baseline height at every jump occurrence.
Extracting its Riemann mesh contributes
\((\sigma N^{1/3})^{n_p}\), while leaving the corresponding Lebesgue
differentials in the limiting integral.  The complete scalar factor is
therefore
\begin{align}
 &N\,
  \bigl(N^{-2/3}\sigma^{-2}\bigr)
  \bigl(N^{-1/3}\sigma^{-1}\bigr)^{n_p-1}
  \bigl(N^{-2/3}(2P_{-1})^{-1}\bigr)
  \bigl(\sigma N^{1/3}\bigr)^{n_p}                 \notag\\
 &\hspace{4cm}=(2P_{-1}\sigma)^{-1}.
 \label{eq:one-excursion-constant-count}
\end{align}
Indeed, the exponent of \(N\) is
\(1-2/3-(n_p-1)/3-2/3+n_p/3=0\), and all but one inverse power of
\(\sigma\) cancel.  If \(n_p=0\), then
\(\widetilde K_pN^{-2/3}\to k_p\) and
\eqref{eq:zero-zero-local-limit} gives directly
\[
  N Z_{\widetilde K_p}(0,0)
  \longrightarrow
  (2P_{-1}\sigma)^{-1}\mathbf F_{0,0}(k_p).
\]
Thus every excursion block supplies exactly one factor
\((2P_{-1}\sigma)^{-1}\), independently of its number of jumps, and the
\(m\)-fold moment supplies \((2P_{-1}\sigma)^{-m}\).
The apparent asymmetry between the two ends is genuine: because the
canonical walk has a unique negative increment $-1$, \(P_{-1}\) enters the
positive-to-zero local limit but not the zero-to-positive one.  Hence it
appears exactly once per excursion block (and is already included in the
zero-to-zero limit when \(n_p=0\)).

After the baseline-height mesh factors have been absorbed into
\eqref{eq:one-excursion-constant-count}, the remaining jump-size and
time meshes, together with the paired-jump weight, give the factor in
\eqref{eq:jump-pair-mesh-calculation}.  The differentials
\(\mathrm dH_-^j\,\mathrm dH_+^j\) remain from the baseline-height
Riemann sums.  Thus the full measure associated with the \(j^{\mathrm{th}}\) pair is
\begin{equation}
  \frac1\theta
  \frac{\sigma^2}{\mu_+^2}
  h^j\,\mathrm dh^j\,
  \mathrm dt_-^j\,\mathrm dt_+^j\,
  \mathrm dH_-^j\,\mathrm dH_+^j.
  \label{eq:limiting-jump-pair-measure}
\end{equation}
This is precisely the \(j^{\mathrm{th}}\) factor of \eqref{eq:intro-jump-measure},
since
\(\theta^{-1}=1+b\).  In particular, \(r\) jump pairs contribute the
factor \(\theta^{-r}=(1+b)^r\).

We next identify the height-mark contribution.  Write
\[
  0=s_{p,0}<s_{p,1}<\cdots<s_{p,n_p}<s_{p,n_p+1}=k_p
\]
for the limiting ordered jump times in the \(p^{\mathrm{th}}\) excursion block.
Condition on these times and on the limiting boundary heights.  As just
proved, the \(a^{\mathrm{th}}\) canonical piece converges to the nonnegative bridge
\(B_{p,a}\) on \([s_{p,a-1},s_{p,a}]\), from
\(x_{p,a-1}\) to \(y_{p,a}\), and the different pieces are conditionally
independent.  Since \(B_{p,a}\) is continuous, the area functional is
continuous at the limiting path in this topology. Let \(\mathbb E_{\boldsymbol B_p}\) denote expectation
under this product law.  The cut-kernel bounds of Section~\ref{subsec:uniform-summability}
give the required uniform integrability, and hence every fixed area moment
converges. Since the height
marks form an unordered subset of the available horizontal positions,
the contribution of \(\rho_p\) marks converges to
\begin{equation}
  \frac1{\rho_p!}\,
  \mathbb E_{\boldsymbol B_p}\left[
  \left(
    \frac{b\sigma}{\mu_+}
    \sum_{a=1}^{n_p+1}
      \int_{s_{p,a-1}}^{s_{p,a}}B_{p,a}(u)\,\mathrm du
  \right)^{\rho_p}\right].
  \label{eq:fixed-rho-area-limit}
\end{equation}
Since \(\rho_p\) is fixed, the sum over distinct unordered marked
positions is the standard elementary-symmetric Riemann sum for the
\(\rho_p^{\mathrm{th}}\) power of the area, with the factor \(1/\rho_p!\).
The diagonals and the finitely many excluded jump positions have
vanishing contribution.  Consequently,
\begin{align*}
 &\sum_{\rho_p=0}^{\infty}\frac1{\rho_p!}\,
  \mathbb E_{\boldsymbol B_p}\left[
  \left(
    \frac{b\sigma}{\mu_+}
    \sum_{a=1}^{n_p+1}
      \int_{s_{p,a-1}}^{s_{p,a}}B_{p,a}(u)\,\mathrm du
  \right)^{\rho_p}\right] \\
 &\qquad=
  \prod_{a=1}^{n_p+1}
  \mathbb E\left[
    \exp\left(
      \frac{b\sigma}{\mu_+}
      \int_{s_{p,a-1}}^{s_{p,a}}B_{p,a}(u)\,\mathrm du
    \right)\right]
  =\mathbf A_{b,p}.
\end{align*}
Thus summing the marks produces exactly the area factor in
\eqref{eq:intro-area-factor}.
It remains to prove the fixed-type limit and to remove the cutoffs at small
jump heights, colliding jump times, and infinity.

\begin{lemma}[Limit of a fixed decorated type]
\label{lem:fixed-decoration-limit}
Fix \(\boldsymbol\varepsilon\in\{0,1\}^m\), \(r,\rho\geq0\), and a
combinatorial type.  Its contribution to
\(\mathcal A_N(\boldsymbol\varepsilon;\boldsymbol k
,\boldsymbol\tau)\) converges to
the matching Riemann integral in the degree-\(\rho\) expansion of
\(\mathbf L_{b,r}[\boldsymbol k;\boldsymbol\tau]\).
\end{lemma}

\begin{proof}
Fix \(0<\eta<1<\mathscr R\).  Restrict to configurations for which every rescaled jump height is at most
\(\mathscr R\), every rescaled baseline height lies in
\([\eta,\mathscr R]\), and every gap
between consecutive jump times, including the first and last gap in each
excursion block, is at least \(\eta\).

We first show that the discarded cutoff regions are uniformly negligible.
For fixed \(r,\rho\), the contribution from configurations with a
rescaled height exceeding \(\mathscr R\) is
\(O_{r,\rho}(e^{-c\mathscr R^2}+e^{-cN^{2/3}})\), uniformly in \(N\), by
\eqref{eq:new-uniform-decoration-tail}.

For the lower cutoff it is enough to revisit the degrees of freedom
counted in the proof of Proposition~\ref{prop:uniform-decorated-path-bound}.
For this estimate, put \(K=\sum_p\widetilde K_p\asymp N^{2/3}\).
Before the reconstruction of Section~\ref{subsec:uniform-summability},
each ordinary jump time is chosen from \(K\) possible positions.  If a
rescaled gap involving an ordinary jump is smaller than \(\eta\), then,
after all other positions are fixed, that jump time has only
\(O(\eta K+1)\) choices instead of \(O(K)\).  Repeating the position
count \eqref{eq:new-total-position-count} therefore gains a factor
\(O_{r,\rho}(\eta+K^{-1})\).  The remaining case, in which a boundary
jump approaches an excursion-block boundary or the two boundary jumps
coalesce, has the same \(o_{\eta\downarrow0}(1)\) bound by the one- and
two-cut kernel sums of Section~\ref{subsec:uniform-summability},
uniformly in \(N\).

A boundary baseline below \(\eta\) is handled by restricting the
height sum in the corresponding one- or two-cut kernel of
Section~\ref{subsec:uniform-summability}.  For an ordinary occurrence of
pair \(j\), once the raw walk and all other pair heights are fixed, the
unscaled baseline \(\sigma N^{1/3}H_{\pm,N}^j\) is an affine
function of \(d_j\) with coefficient \(-1\).  Thus
\(H_{\pm,N}^j<\eta\) confines \(d_j\) to
\(O(\eta R_q+1)\) consecutive integer values and, including its
numerical pair-height factor, gains \(O(\eta+R_q^{-1})\).  A union bound over the finitely many gaps and baseline heights shows that
the entire discarded lower-cutoff region has this bound.

We now work on the truncated domain.  Every intermediate bridge has two
positive endpoints, uniformly bounded away from zero, while only the first
and last bridges in an excursion block have a zero endpoint.  Thus the
positive-to-positive local and functional limits apply uniformly to all
intermediate bridges, and the appropriate fixed zero-endpoint limits apply
to the boundary bridge or bridges.  The cut-kernel bounds of
Section~\ref{subsec:uniform-summability} give the required uniform
summability.  Together with the elementary-symmetric Riemann sum for the unordered
marks and the convergence in
\eqref{eq:one-cross-pair-time-limit}, this shows that the contribution
of the fixed type converges to its truncated Brownian integral.

The scalar normalization is \((2P_{-1}\sigma)^{-m}\) by
\eqref{eq:one-excursion-constant-count}; the jump and time meshes give
\eqref{eq:limiting-jump-pair-measure}; and the marks give
\eqref{eq:fixed-rho-area-limit}.  The factor \(1/r!\) is retained
from \eqref{eq:weight-of-configuration}, while \(1/\rho_p!\) comes from the
unordered marks in excursion block \(p\).

Passing first \(N\to\infty\), then \(\eta\downarrow0\), and finally
\(\mathscr R\to\infty\) proves the fixed-type convergence.  On the
Brownian side, the lower time cutoff may be removed by the integrability
of \(\mathbf F_0(t;u)\) near \(t=0\), while the Gaussian kernel factors
control the upper height cutoff.  Hence the full fixed-type Brownian
integral is absolutely convergent.
\end{proof}

\begin{proof}[Proof of Theorem~\ref{thm:deformed-moment-convergence}]
Fix \(\boldsymbol\varepsilon\in\{0,1\}^m\).  The actual lengths
\(\widetilde K_p=K_p+\varepsilon_p\) and their time coordinates have
already been used in \eqref{eq:jump-rescaling}.  Lemma
\ref{lem:fixed-decoration-limit} gives the desired limit for every
fixed \(r,\rho\) and every fixed combinatorial type.

For fixed \(r,\rho\), there are only finitely many combinatorial
types, so Lemma~\ref{lem:fixed-decoration-limit} may be summed over
all of them. This gives the full degree-\(\rho\), \(r\)-jump-pair contribution to
\(\mathbf L_b[\boldsymbol k;\boldsymbol\tau]\). 
By \eqref{eq:new-absolute-uniform-summability}, the total contribution
of \(r+\rho>M\) tends to zero uniformly in \(N\) as \(M\to\infty\).
Hence the sums over \(r,\rho\) may be passed through the limit.  Applying the fixed-type argument to
absolute weights, equivalently replacing each height-mark factor \(b\)
by \(|b|\), and then using Fatou's lemma shows at the same time that the
limiting Brownian series is absolutely convergent. Since both the fixed-type convergence and the
uniform summability bound are locally uniform in
\(\boldsymbol\tau\), the resulting convergence is locally uniform in
\(\boldsymbol\tau\) as well. A truncation at
\(r+\rho\leq M\), followed first by \(N\to\infty\) and then by
\(M\to\infty\), therefore permits all decoration sums to pass through
the limit.

Summing over \(\rho_1,\ldots,\rho_m\) produces
\(\prod_p\mathbf A_{b,p}\) by
\eqref{eq:fixed-rho-area-limit}, while summing over the finitely many
excursion block assignments for each \(r\), and then over \(r\), gives exactly
\(\mathbf L_b[\boldsymbol k;\boldsymbol\tau]\).  This proves
\eqref{eq:single-product-action-limit} and its asserted absolute
convergence.  Finally, Lemma~\ref{lem:exponential-representative} gives

\begin{equation*}
  \mathcal M_N(\boldsymbol k;\boldsymbol\tau)
  =\sum_{\boldsymbol\varepsilon\in\{0,1\}^m}
    \left(\prod_{p=1}^m\mu_{p,N}^{\varepsilon_p}\right)
    \mathcal A_N(\boldsymbol\varepsilon;
      \boldsymbol k,\boldsymbol\tau).
\end{equation*}
The sum has only \(2^m\) terms, and

\begin{equation*}
  \sum_{\boldsymbol\varepsilon\in\{0,1\}^m}
  \prod_{p=1}^m\mu_{p,N}^{\varepsilon_p}
  =\prod_{p=1}^m(1+\mu_{p,N})
  \longrightarrow(1+\mu_+)^m.
\end{equation*}
Taking the limit term by term proves
\eqref{eq:deformed-moment-limit}. This also proves Theorem~\ref{thm:intro-Jack-Plancherel-process-moments}.
\end{proof}

\section{Edge limit of general Jack processes}
\label{sec:general-Jack-processes}

In this section, we briefly discuss the extension of Theorem \ref{thm:intro-Jack-Plancherel-process-moments} to a class of general Jack processes, using the same arguments as in Section \ref{sec:combinatorial-asymptotics}.

For a specialization 
$\rho$ of symmetric polynomials, recall that $\rho$ is Jack-positive if $J_\lambda(\rho;\theta)\geq0$ for every partition
$\lambda$. The classification of all Jack-positive specializations was given by Kerov--Okounkov--Olshanski \cite[Theorem~A]{KOO}. These specializations are
parametrized by nonnegative non-increasing sequences
$\boldsymbol\alpha,\boldsymbol\beta$ with
$\sum_i(\alpha_i+\beta_i)<\infty$, and $\gamma\geq0$, through
\begin{equation*}
  \kappa_1:=p_1(\rho)=\gamma+\sum_i(\alpha_i+\beta_i),\qquad
  \kappa_d:=p_d(\rho)=\sum_i\alpha_i^d
       +(-\theta)^{d-1}\sum_i\beta_i^d,\quad d\geq2.
\end{equation*}

For discrete time $n=0,1,2,\ldots$ the Jack process indexed by a Jack-positive $\rho$ is a Markov process $\lambda^{(N)}(n)$ with transition kernel 
\begin{equation}
  \mathsf P_{N,\rho}^{(\theta)}(\lambda,\mu)
  :=\frac{1}{\mathcal H_\theta(\rho,1^N)}
    \frac{j_\lambda(\theta)}{j_\mu(\theta)}
    \frac{J_\mu(1^N;\theta)}{J_\lambda(1^N;\theta)}
    J_{\mu/\lambda}(\rho;\theta),
  \label{eq:general-Jack-transition}
\end{equation}
started at $\varnothing$.
The Jack Plancherel process is a special case where $\gamma=1, \alpha_1=\beta_1=0$.
Note that besides Jack-positivity, the kernel is well-defined only if the normalizing term \[\mathcal H_\theta(\rho,1^N)
  =\exp\left(
    \theta N\sum_{d\geq1}\frac{p_d(\rho)}d
  \right)<\infty,\]
  which implies that $\limsup_{d\to\infty}\kappa_d^{1/d}<1$.
Write $n\rho$ for the union of $n$ copies of $\rho$, so that
$p_d(n\rho)=n\kappa_d$. At time $n$, the Jack process has marginal $\mathbb P_{1^N,n\rho}^{(\theta)}$, the Jack measure defined in (\ref{eq_jackmeasure}).

For $\theta=1$, the Plancherel, pure-$\alpha$ ($\alpha_1=\delta<1,\alpha_2=\alpha_3=\ldots=0)$, and pure-$\beta$ ($\beta_1=\delta<1,\beta_2=\beta_3=\ldots=0$)
specializations give nonintersecting Poisson, geometric, and Bernoulli
walks, respectively. These models are closely
connected with last-passage percolation, and play essential roles in the KPZ universality class, and
their multi-time Airy edge limits were well known, see  \cite{JohanssonPNG} and \cite{DauvergneNicaVirag}, while for general $\beta$ it was not yet proven. For general specialization $\rho$, under the coordinates $\ell_i=\lambda_{N-i+1}+(i-1)\theta$, the Jack processes give a large class of discrete particle systems parameterized by $\beta=2\theta>0$, but 
even their one-time marginals need not have the discrete log gas formulation in \cite{BGG}, so the results in \cite{GH} do not apply directly.

We impose the following conditions on the fixed specialization:
\begin{equation}
  \alpha_1<1,\qquad \kappa_1>0,\qquad
  \kappa_d\geq0\quad(d\geq2).
  \label{eq:general-Jack-parameter-conditions}
\end{equation}
The first condition is precisely finiteness of the normalization in
\eqref{eq:general-Jack-transition}, and the nonnegativity of $\kappa_d$ is the technical restriction in our approach. For even $d$, this in particular implies
$\theta\beta_1\leq\alpha_1$, hence
\[
  R:=\limsup_{d\to\infty}\kappa_d^{1/d}
    =\max\{\alpha_1,\theta\beta_1\}=\alpha_1<1.
\]

We extract the moments of the Jack process at time $s_{1,N}\ge\ldots\ge s_{m,N}\in \mathbb Z_{\ge 1}$, using again the action of Nazarov-Sklyanin operator on the corresponding Jack generating function, see Sections \ref{subsec:NS-operator} to \ref{subsec:Jack-Plancherel-process} for the details. In particular, Theorem \ref{thm:multitimemoment} holds in the exact same way for general Jack processes.
By the Cauchy identity (\ref{eq:Jack-Cauchy-identity}), one can see after a calculation that at time $n$,
\[\mathcal G_{N,n}(\boldsymbol p;\theta)
  =\exp\left(\theta n\sum_{d\geq1}\kappa_d(p_d-N)/d\right).\]

Set
\begin{equation}
\begin{gathered}
  s_{p,N}:=\lfloor2N+\tau_pN^{2/3}\rfloor,\qquad
  a_{p,N}:=s_{p,N}/N,\\
  V_{p,N}(z):=\sum_{d\geq1}z^d+
       a_{p,N}\sum_{d\geq1}\kappa_dz^{-d},\qquad R<z<1.
\end{gathered}
  \label{eq:general-Jack-symbol}
\end{equation}
A simple calculation yields that for large $N$, there is a unique
$z_{p,N}\in(R,1)$ with $V'_{p,N}(z_{p,N})=0$.
Moreover, there exists $z_c\in (R,1)$, $\mu_+>0$ such that
$z_{p,N}\to z_c$ and
$\mu_{p,N}:=V_{p,N}(z_{p,N})\to\mu_+$.

As in Section~\ref{subsec:decorated-path-expansion}, we construct a random walk with i.i.d., mean 0, and finite variance
increments $X_{p,N}$, which have probabilities
\begin{equation*}
  \mathbb P(X_{p,N}=d)=\frac{z_{p,N}^d}{\mu_{p,N}},\qquad
  \mathbb P(X_{p,N}=-d)
    =\frac{a_{p,N}\kappa_dz_{p,N}^{-d}}{\mu_{p,N}},\qquad d\geq1.
\end{equation*}
It is straightforward to check that they satisfy all assumptions of Section~\ref{subsec:conditional-walk-bridges}. Let $\sigma^2>0$ be the limiting
variance and $a_1,a_2>0$ the constants of
Proposition~\ref{prop:conditional-bridge-local-limits}.
As discussed in \eqref{eq:finite-N-cross-time-factor}, the multi-time factor in the action is still of the form
\[
  \Xi_N=\prod_{j=1}^r
  \left(\frac{z_{\ell_+^j,N}}{z_{\ell_-^j,N}}\right)^{d_j}.
\]
Put $V(z):=\lim_{N\to\infty}V_{p,N}(z)$, and let $z(a)$ be the
critical point in \eqref{eq:general-Jack-symbol} with $a_{p,N}=a$,
so that $z(2)=z_c$. The implicit function theorem gives
\begin{align*}
  z'(2)&=\frac{\sum_{d\geq1}d\kappa_dz_c^{-d-1}}{V''(z_c)}
       =\frac{1}{2(1-z_c)^2V''(z_c)},\\
  N^{1/3}\log\frac{z_{q,N}}{z_{p,N}}
       &\longrightarrow\frac{z'(2)}{z_c}(\tau_q-\tau_p).
\end{align*}
Thus, when $d_j/(\sigma N^{1/3})\to h^j$, the multiplier $\Xi_N$
converges to \eqref{eq:intro-multitime-factor} with
\begin{equation}
  \tilde c:=\frac{\sigma z'(2)}{z_c}
  =\frac{\sigma}{2z_c(1-z_c)^2V''(z_c)}>0.
  \label{eq:general-Jack-time-coefficient}
\end{equation}
For the general Jack processes, set
\begin{equation*}
\begin{aligned}
  \widetilde{\mathbf L}_{b,r}^{\boldsymbol\ell}
    [\boldsymbol k;\boldsymbol\tau]
  &:=\frac{(\sigma a_1a_2)^m}{r!}
  \int_{\mathcal D_{\boldsymbol k}(\boldsymbol\ell)}
  \prod_{j=1}^r
  \exp\!\left\{\tilde c
    (\tau_{\ell_+^j}-\tau_{\ell_-^j})h^j\right\}
\prod_{p=1}^m\mathbf D_p\mathbf A_{b,p}\,
  \mathrm d\boldsymbol\nu_r,
  \\
  \widetilde{\mathbf L}_{b,r}[\boldsymbol k;\boldsymbol\tau]
  &:=\sum_{\boldsymbol\ell\in\mathfrak A_{m,r}}
  \widetilde{\mathbf L}_{b,r}^{\boldsymbol\ell}
    [\boldsymbol k;\boldsymbol\tau],\qquad
  \widetilde{\mathbf L}_b[\boldsymbol k;\boldsymbol\tau]:=\sum_{r=0}^{\infty}
  \widetilde{\mathbf L}_{b,r}[\boldsymbol k;\boldsymbol\tau].
\end{aligned}
\end{equation*}
Here $\mathbf D_p$ and $\mathbf A_{b,p}$ are defined in the same way as in (\ref{eq:intro-transition-factor}) and (\ref{eq:intro-area-factor}), with the present $\sigma,\mu_+$. The changes from
\eqref{eq:intro-multitime-fixed-functional} are the prefactor
$(\sigma a_1a_2)^m$ and the time coefficient $\tilde c$ in
\eqref{eq:general-Jack-time-coefficient}.

\begin{theorem}[Multi-time moment convergence for general Jack processes]
\label{thm:general-Jack-moment-convergence}
Fix $\theta>0$ and a Jack-positive specialization $\rho$ satisfying
\eqref{eq:general-Jack-parameter-conditions}. For $m\geq1$,
$\boldsymbol k\in\mathbb R_{>0}^m$, and
$\tau_1\geq\cdots\geq\tau_m$, put $K_p=\lfloor k_pN^{2/3}\rfloor$
and $b=\theta^{-1}-1$. Then $\widetilde{\mathbf L}_b[\boldsymbol k;\boldsymbol\tau]$
converges absolutely, and
\begin{align*}
  &\lim_{N\to\infty}\mathbb E\left[
    \prod_{p=1}^m N\int
    \left(\frac{x}{\mu_{p,N}}\right)^{K_p}
    \mathrm d\mu_{\mathrm{PP}}[
       \lambda^{(N)}(s_{p,N})](x)\right]
  \notag\\
  &\qquad=(1+\mu_+)^m\widetilde{\mathbf L}_b[\boldsymbol k;\boldsymbol\tau].
\end{align*}
The convergence is locally uniform in the ordered time vector
$\boldsymbol\tau$.
\end{theorem}
\noindent{Sketch of proof.} The proof follows almost in the same way as in Section \ref{sec:combinatorial-asymptotics}. In Section~\ref{subsec:decorated-path-expansion}, among the four effects of an entry action, only in effect (iv), the derivatives in $p_d$ acting on $\mathcal G_{N,n}(\boldsymbol p;\theta)$ now have the freedom $d\ge 1$, which correspond to the negative increments in $\mathbb Z_{\le -1}$ of the canonical steps. All the other combinatorial features of the action, including the  weights of the
paired-jump and the height marks are unchanged. After that,
the reconstruction and
summability arguments in Section~\ref{subsec:uniform-summability}
hold universally for the decorated paths, as long as the canonical walks satisfy all the assumptions in Section \ref{subsec:conditional-walk-bridges}. Finally, the convergence in 
Section~\ref{subsec:passage-to-Brownian-functional} also holds in the same way, while one has to replace the parameters in the limiting functional by their more general form, as listed previously in this Section. The final sum over the
$I^{(K_p)}$ and $I^{(K_p+1)}$ terms again contributes $(1+\mu_+)^m$.\qed

Theorem \ref{thm:general-Jack-moment-convergence} covers a wide class of discrete dynamical $\beta$-ensembles: every nonzero specialization with all $\beta_i=0$ and
finitely many $\alpha_i<1$, including the $\beta$-nonintersecting Poisson (Plancherel)
and $\beta$-nonintersecting geometric walks, where $\beta=2\theta$. However, the non-intersecting Bernoulli walks are not covered, because their even $\kappa_d$ are
negative, so the canonical weights do not define a legitimate random walk.
It would be interesting to overcome this technical difficulty.

\section{Identification of the one-time limit}
\label{sec:airy-identification}

In the remainder of this text, we specialize to
$\tau_1=\cdots=\tau_m=0$ for the limit of the Jack Plancherel process.  Thus
$s_{p,N}=2N$, $\mu_{p,N}=\mu_+$, and
$\mathbf L_b[\boldsymbol k;\boldsymbol0]=\mathbf L_b[\boldsymbol k]$. Write
\[
  \lambda:=\lambda^{(N)}(2N).
\]
Then $\lambda$ has law $\mathbb P_{N,2N}^{(\theta)}$, and
$\mathcal M_N(\boldsymbol k;\boldsymbol0)$ is exactly the one-time
deformed moment considered below.

We separate the argument into two steps.  First we use only a weak
rigidity and level-repulsion input to remove the Perelomov--Popov
deformation.  We then prove tightness, pass to a subsequential limit, and
identify the mixed moments of its exponential linear statistics.  The
Airy$_\beta$ edge theorem is used only in the final paragraph, where
$\beta\geq1$.

Put
\begin{equation*}
  Y_{i,N}:=\theta N y_i=\lambda_i-\theta(i-1),
  \qquad
  u_{i,N}:=N^{2/3}\left(\frac{y_i}{\mu_+}-1\right).
\end{equation*}
We also write $\mathcal X_N:=\sum_{i=1}^N\delta_{u_{i,N}}$ for the rescaled edge point process.
Recall the Perelomov--Popov weights
\begin{equation}
  \omega_{i,N}
  =\prod_{j\neq i}
    \frac{Y_{i,N}-Y_{j,N}+\theta}{Y_{i,N}-Y_{j,N}}
  =\prod_{j\neq i}
    \frac{y_i-y_j+N^{-1}}{y_i-y_j}.
  \label{eq:edge-PP-weights}
\end{equation}
The shifted-lattice constraint gives
\[
  Y_{i,N}-Y_{j,N}\geq\theta(j-i),
  \qquad i<j,
\]
so all the weights are nonnegative.  The rational identity
\[
  \frac{\prod_{j=1}^N(z-Y_{j,N}+\theta)}
       {\prod_{j=1}^N(z-Y_{j,N})}
  =
  1+\theta\sum_{i=1}^N\frac{\omega_{i,N}}{z-Y_{i,N}}
\]
gives the total mass directly: the coefficient of $z^{-1}$ in the expansion
at $z=\infty$ is $N\theta$ on the left and
$\theta\sum_{i=1}^N\omega_{i,N}$ on the right.  Hence
\begin{equation}
  \sum_{i=1}^N\omega_{i,N}=N.
  \label{eq:PP-weights-total-mass}
\end{equation}
This is the standard total-mass identity for the Perelomov--Popov weights;
see e.g. \cite[Proposition~1]{ZografosPP}.

\subsection{A weak rigidity input}

The following lemma is the only microscopic input needed before the final
Airy identification.  It uses rigidity and weak level repulsion from
\cite{GH}, but not their edge-convergence theorem.

\begin{lemma}[Edge weight and reciprocal-gap bounds]
\label{lem:edge-reciprocal-gap}
Fix $0<\mathfrak b<1/13$ and put $L=\lfloor N^{\mathfrak b}\rfloor$.
Then the following statements hold.
\begin{enumerate}[label=\textnormal{(\roman*)}]
  \item The Perelomov--Popov weights satisfy
  \begin{equation*}
    \max_{1\leq i\leq L/3}
    \left|\omega_{i,N}-z_c^{-1}\right|
    \xrightarrow{\mathbb P}0.
  \end{equation*}
  \item There is a constant $C=C(\theta)>0$ such that, with probability
  tending to one,
  \begin{equation*}
    \sum_{j\neq i}\frac1{|Y_{i,N}-Y_{j,N}|}
    \leq C(1+\log i)
    \leq Ci^{2/3},
    \qquad \frac L3<i\leq10^{-2}N.
  \end{equation*}
\end{enumerate}
\end{lemma}

\begin{proof}
Choose
\[
  3\mathfrak b<d<\frac13-\frac{4\mathfrak b}{3},
  \qquad
  0<\mathfrak a<\min\{d,\mathfrak b/2\},
  \qquad s=N^{-d}.
\]
The interval for $d$ is nonempty because $\mathfrak b<1/13$.
The bound \eqref{eq:unconditional-weak-level-repulsion} in
Appendix~\ref{subsec:good-exterior-conditioning} implies that, with
probability $1-o(1)$,
\begin{equation}
  Y_{k,N}-Y_{k+1,N}\geq
  \delta_N:=sL^{-1/3}N^{1/3},
  \qquad 1\leq k\leq L.
  \label{eq:edge-block-minimum-gap}
\end{equation}
Indeed, the choices above give
$L^{4/3}N^{-1/3}\ll s\ll\min\{L^{-2},N^{-\mathfrak a}\}$ and
$sL^3=o(1)$.

We first prove \textnormal{(i)}.  For $i\leq L/3$, split the product
\eqref{eq:edge-PP-weights} into $j\leq L$ and $j>L$.  By
\eqref{eq:edge-block-minimum-gap},
\[
  \sup_{i\leq L/3}
  \sum_{\substack{j\leq L\\j\neq i}}
  \frac1{|Y_{i,N}-Y_{j,N}|}
  \leq\frac{C\log L}{\delta_N}=o(1).
\]
Thus the factors with $j\leq L$ contribute $1+o(1)$, uniformly in
$i\leq L/3$.

For $j>L$, the right-edge versions of
\cite[(4.14)--(4.15)]{GH} identify
\[
  \theta\sum_{j>L}\frac1{Y_{i,N}-Y_{j,N}}
  =\frac1N\sum_{j>L}\frac1{y_i-y_j}
\]
with the Hilbert transform of the equilibrium measure, up to
$o(1)$ uniformly for $i\leq L/3$, with probability $1-o(1)$.%
\footnote{The required high-probability restriction on the interior
configuration is first obtained in \cite{GH} under the perturbed
conditional measure.  Appendix~\ref{subsec:good-exterior-conditioning}
transfers it to the original ensemble with an $o(1)$ loss; once this event
and the good exterior condition hold, \cite[(4.14)--(4.15)]{GH} are
deterministic estimates.}
Rigidity gives $y_i=Y_{i,N}/(\theta N)\to\mu_+$ uniformly for
$i\leq L/3$.  Let $\mu$ denote the equilibrium measure in the $\ell/N$-coordinates, $\mu_y$ denote its image in the $y$-coordinates, and $G_\mu(z):=\int(z-x)^{-1}\mu(\mathrm dx)$ be its Stieltjes transform as in Appendix~\ref{app:discrete-beta-edge}.  Then the Hilbert transform converges to
\[
  \int\frac{\mu_y(\mathrm dv)}{\mu_+-v}
  =\theta G_\mu(B)=\log z_c^{-1},
\]
where the last two identities are
\eqref{eq:y-equilibrium-measure-shift} and
\eqref{eq:edge-PP-constant}.  Since $j>L>i$, all denominators in the
far sum have the same sign, so the same comparison gives a uniform bound
for its absolute sum.  Equation~\eqref{eq:edge-block-minimum-gap} also
makes the largest reciprocal gap $o(1)$.  Therefore the quadratic
remainder in $\log(1+x)=x+O(x^2)$ is $o(1)$, and
\[
  \log\omega_{i,N}=\log z_c^{-1}+o(1)
\]
uniformly for $i\leq L/3$.  This proves \textnormal{(i)}.

For \textnormal{(ii)}, let $L/3<i\leq10^{-2}N$ and split the sum at
$j=2i$.  The shifted-lattice constraint gives
\[
  \sum_{\substack{j\leq2i\\j\neq i}}
  \frac1{|Y_{i,N}-Y_{j,N}|}
  \leq\frac1\theta
  \sum_{\substack{j\leq2i\\j\neq i}}\frac1{|i-j|}
  \leq C(1+\log i).
\]
For $j>2i$, rigidity and the square-root behavior of the classical
locations imply, uniformly while $j$ remains in a fixed edge band,
\[
  |Y_{i,N}-Y_{j,N}|
  \geq cN^{1/3}\bigl(j^{2/3}-i^{2/3}\bigr)
  \geq cN^{1/3}j^{2/3}.
\]
Consequently this part contributes at most
$CN^{-1/3}\sum_{j\leq cN}j^{-2/3}=O(1)$.
Outside the fixed edge band the separation is of order $N$, so the
remaining $O(N)$ terms also contribute $O(1)$.  This proves
\textnormal{(ii)}.
\end{proof}

\subsection{Tightness and a subsequential point-process limit}

For $k>0$
put
\[
  q_N(k):=\left\lfloor\frac{kN^{2/3}}2\right\rfloor
\]
and, for $u\in\mathbb R$, define
\begin{align*}
  \Phi_{N,k}^{+}(u)
  &:=\frac12\left(1+N^{-2/3}u\right)^{2q_N(k)},\notag\\
  \Phi_{N,k}(u)
  &:=\Phi_{N,k}^{+}(u)
    +\frac12\left(1+N^{-2/3}u\right)^{2q_N(k)+1}.
\end{align*}
Let $u_{i,N}$ satisfy $1+N^{-2/3}u_{i,N}=y_i/\mu_+$.  Introduce the deformed and
ordinary statistics
\begin{equation}
  T_N(k):=\sum_{i=1}^N\omega_{i,N}\Phi_{N,k}(u_{i,N}),
  \qquad
  R_N(k):=\sum_{i=1}^N\Phi_{N,k}(u_{i,N}).
  \label{eq:deformed-and-ordinary-edge-statistics}
\end{equation}
Expanding each factor in $T_N(k)$ into its even and odd powers gives a
finite linear combination of the statistics in
\eqref{eq:deformed-moment-definition}, with every exponent changed by
at most one.  The estimates in Section~\ref{sec:combinatorial-asymptotics}
are uniform under such bounded changes.  Hence
Theorem~\ref{thm:deformed-moment-convergence} gives, for
$k_1,\ldots,k_m>0$,
\begin{equation}
  \lim_{N\to\infty}
  \mathbb E\left[\prod_{p=1}^mT_N(k_p)\right]
  =(1+\mu_+)^m\mathbf L_b[\boldsymbol k].
  \label{eq:symmetrized-deformed-moment-limit}
\end{equation}

Put $L_0:=\lfloor L/3\rfloor$, with $L$ as in
Lemma~\ref{lem:edge-reciprocal-gap}, and write
$T_N^{\leq L_0}(k),T_N^{>L_0}(k)$ and
$R_N^{\leq L_0}(k),R_N^{>L_0}(k)$ for the restrictions of the sums in
\eqref{eq:deformed-and-ordinary-edge-statistics}.

We first record the two tail estimates needed below.  With probability
$1-o(1)$,
\[
  u_{i,N}\leq-ci^{2/3}\quad
  (L_0<i\leq10^{-2}N),
  \qquad
  \frac{y_i}{\mu_+}\leq1-c\quad
  (i>10^{-2}N,\ y_i\geq0).
\]
For $y_i<0$, one has deterministically
$|y_i|/\mu_+\leq\mu_+^{-1}$.  Hence, for every fixed $k>0$,
\begin{equation}
  R_N^{>L_0}(k)
  \leq C\sum_{i>L_0}e^{-cki^{2/3}}+o_{\mathbb P}(1)
  \xrightarrow{\mathbb P}0.
  \label{eq:ordinary-tail-vanishing}
\end{equation}
For the deformed statistic, on the same event,
\[
  \Phi_{N,k}(u_{i,N})
  \leq
  \begin{cases}
    C e^{-ckL_0^{2/3}}\Phi_{N,k/2}(u_{i,N}),&i>L_0,\ y_i\geq0, \\
    \mu_+^{-2q_N(k)},&y_i<0.
  \end{cases}
\]
The second line uses $y_i\geq-1$.  Together with
\eqref{eq:PP-weights-total-mass}, this gives, with probability $1-o(1)$,
\begin{equation*}
  |T_N^{>L_0}(k)|
  \leq
  C e^{-ckL_0^{2/3}}T_N(k/2)
  +2N\mu_+^{-2q_N(k)}.
\end{equation*}  By
\eqref{eq:symmetrized-deformed-moment-limit}, $T_N(k/2)$ is tight, and
therefore
\begin{equation}
  T_N^{>L_0}(k)\xrightarrow{\mathbb P}0.
  \label{eq:deformed-tail-vanishing}
\end{equation}

Put
\[
  \Delta_N:=\max_{1\leq i\leq L_0}
  |\omega_{i,N}-z_c^{-1}|.
\]
Part~\textnormal{(i)} of Lemma~\ref{lem:edge-reciprocal-gap} gives
$\Delta_N\to0$ in probability.  On
$\{\Delta_N\leq z_c^{-1}/2\}$,
\[
  R_N^{\leq L_0}(k)\leq2z_cT_N^{\leq L_0}(k)\leq2z_cT_N(k).
\]
The variables $T_N(k)$ are tight by
\eqref{eq:symmetrized-deformed-moment-limit}; hence
$R_N^{\leq L_0}(k)$ is tight.  Together with
\eqref{eq:ordinary-tail-vanishing}, this proves tightness of $R_N(k)$.

 Let
\[
  \mathcal D:=\mathbb N\cup\{n^{-1}:n\in\mathbb N\},
  \qquad
  R_N^+(k):=\sum_{i=1}^N\Phi_{N,k}^+(u_{i,N}).
\]
Since $y_i/\mu_+\geq-\mu_+^{-1}$ and
\[
  \Phi_{N,k}(u_{i,N})
  =\Phi_{N,k}^+(u_{i,N})
   \left(1+\frac{y_i}{\mu_+}\right),
\]
we have
\[
  0\leq R_N^+(k)
  \leq(1-\mu_+^{-1})^{-1}R_N(k).
\]
Thus the two families $\{R_N(k):k\in\mathcal D\}$ and
$\{R_N^+(k):k\in\mathcal D\}$ are jointly tight.

Starting from an arbitrary subsequence, a diagonal extraction gives a further
subsequence along which all these variables converge jointly in distribution.
By the Skorokhod representation theorem, we may realize this convergence almost
surely.  For almost every outcome, the hypotheses of
\cite[Proposition~6.2]{KX} are then satisfied.  Applying that proposition
pathwise to the ordered array $(u_{i,N})_{i=1}^N$ produces an ordered sequence
\[
  u_1\geq u_2\geq\cdots\longrightarrow-\infty
\]
such that $u_{i,N}\to u_i$ in $\mathbb R\cup\{-\infty\}$ for every fixed
$i$ and, simultaneously, for every $k>0$,
\begin{equation*}
  R_N(k)\longrightarrow
  \mathcal Z_{\mathcal X}(k):=
  \sum_{i\geq1}e^{ku_i}
  \qquad\text{almost surely},
\end{equation*}
where $\mathcal X:=\sum_{i\geq1}\delta_{u_i}$.

Fix now $k_1,\ldots,k_m>0$.  Passing to a further subsequence if necessary,
part~\textnormal{(i)} of Lemma~\ref{lem:edge-reciprocal-gap},
\eqref{eq:ordinary-tail-vanishing}, and
\eqref{eq:deformed-tail-vanishing} hold almost surely for these finitely many
parameters.  For each $p$,
\[
  R_N^{\leq L_0}(k_p)
  =R_N(k_p)-R_N^{>L_0}(k_p)
  \longrightarrow\mathcal Z_{\mathcal X}(k_p)
\]
almost surely, and
\[
  \left|
    T_N^{\leq L_0}(k_p)-z_c^{-1}R_N^{\leq L_0}(k_p)
  \right|
  \leq\Delta_NR_N^{\leq L_0}(k_p)
  \longrightarrow0.
\]
Together with $T_N^{>L_0}(k_p)\to0$, this proves
\begin{equation}
  T_N(k_p)\longrightarrow
  z_c^{-1}\mathcal Z_{\mathcal X}(k_p)
  \qquad\text{almost surely}.
  \label{eq:deformed-statistic-subsequence-limit}
\end{equation}

Equation~\eqref{eq:symmetrized-deformed-moment-limit}, applied with every
$k_p$ repeated twice, gives
\[
  \sup_N
  \mathbb E\left[
    \left(\prod_{p=1}^mT_N(k_p)\right)^2
  \right]<\infty.
\]
Thus these products are uniformly integrable.  Multiplying
\eqref{eq:deformed-statistic-subsequence-limit} over $p=1,\ldots,m$, taking expectations, and using
\eqref{eq:symmetrized-deformed-moment-limit}, we obtain
\begin{equation}
  \mathbb E\left[
    \prod_{p=1}^m\mathcal Z_{\mathcal X}(k_p)
  \right]
  =z_c^m(1+\mu_+)^m\mathbf L_b[\boldsymbol k]
  =C_2^m\mathbf L_b[\boldsymbol k].
  \label{eq:subsequential-limit-functional}
\end{equation}
Consequently, for every $\beta>0$, every sequence has a subsequence whose
edge process converges to a point process with the explicit transform
\eqref{eq:subsequential-limit-functional}.  Up to this point we used only
the weak estimates in Lemma~\ref{lem:edge-reciprocal-gap}, the global
equilibrium law, and the moment calculation of
Section~\ref{sec:combinatorial-asymptotics}; no pre-existing identification
of the edge process was invoked.

\subsection{Matching with the Airy$_\beta$ process}

To upgrade the subsequential conclusion above to weak convergence of the
entire sequence $\mathcal X_N$, one needs to know that the mixed
Laplace-transform expression in
\eqref{eq:subsequential-limit-functional} uniquely determines the law of the
limiting point process.  This determinacy is not immediate from our explicit
formula, whose available bounds grow too rapidly to verify the standard
moment-determinacy criteria directly.  For $\beta\geq1$, however, the
limiting process is already known: the available right-edge theorem identifies
the weak limit of $\mathcal X_N$ itself, so no abstract determinacy argument is
needed.

\begin{proof}[Proof of Theorem~\ref{thm:main-Airy-Laplace}]
Assume now that $\beta\geq1$.  The right-edge theorem of \cite{GH},
whose hypotheses and normalization are checked in
Appendix~\ref{app:discrete-beta-edge}, gives
\[
  \left(
    \mathfrak c_2N^{2/3}(y_i-\mu_+)
  \right)_{i\geq1}
  \Longrightarrow
  \left(\mathcal A_i^{(\beta)}\right)_{i\geq1}.
\]
Equivalently,
\[
  \mathcal X_N
  \Longrightarrow
  \sum_{i\geq1}
  \delta_{\mathcal A_i^{(\beta)}/(\mu_+\mathfrak c_2)}.
\]
Therefore the subsequential process $\mathcal X$ in
\eqref{eq:subsequential-limit-functional} is the rescaled
Airy$_\beta$ process.  Substitution into that identity yields
\[
  \mathbb E\left[
    \prod_{p=1}^m
    \mathcal Z_\beta\left(
      \frac{k_p}{\mu_+\mathfrak c_2}
    \right)
  \right]
  =C_2^m\mathbf L_b[\boldsymbol k],
\]
which is \eqref{eq:intro-main-Laplace-formula}.  Absolute convergence of
the functional was proved in
Section~\ref{subsec:uniform-summability}.
\end{proof}

\begin{remark}
The decorated-path asymptotics, the removal of the
Perelomov--Popov weights, and the construction of a subsequential edge
process hold for every $\beta>0$.  The restriction $\beta\geq1$ enters
only in the last matching step, through the presently available
right-edge theorem for the discrete beta ensemble.
\end{remark}

\section{The $\beta$-topological expansion}
\label{sec:beta-topological}

\subsection{Map and measure of non-orientability}
We recall the notion of map on surfaces and the measure of non-orientability, following the conventions of \cite[Sections~2 and~4.1]{LaCroix}.
A map is a cellular embedding of a connected graph in a compact connected
surface, which need not be orientable.  Equivalently, it is a ribbon graph in
which edge ribbons are allowed to be twisted.  Each edge ribbon has four local
corners, determined by a choice of one attachment end and one boundary side.
A flag is such a corner, and a rooted map is a map with a distinguished flag,
modulo automorphisms.

It is convenient
to represent the root flag by an arrow lying in
the root face and pointing to the root vertex
\cite[Section~2.4]{LaCroix}.  More precisely, the drawing supplies a local
clockwise direction, and if the successive boundary occurrences of the root
face are indexed clockwise as $c_0,\ldots,c_{d-1}$, an arrow pointing to
$c_i$ denotes the flag immediately clockwise from it, adjacent to the boundary
side from $c_i$ to $c_{i+1}$, with indices taken modulo $d$.  For a rooted map $\mathcal M$, its \emph{measure of non-orientability} is a nonnegative integer $\eta(\mathcal{M})$ determined by the following algorithm. Set
$\eta(\mathcal M)=0$ if $\mathcal{M}$ has no edge, and otherwise let $e$ be its root
edge and delete the ribbon associated with $e$.

If $e$ is a cut edge, its deletion disconnects the ribbon graph into two
components.  The original root arrow remains a root of the component that
contains it, while the second component is rooted by an arrow drawn along the
deleted ribbon and pointing toward that component.  This gives a canonically
ordered pair $(\mathcal M_1,\mathcal M_2)$ of rooted maps.  In this case $e$
is called a \emph{bridge}, and
\[
  \eta(\mathcal M)=\eta(\mathcal M_1)+\eta(\mathcal M_2).
\]
If $e$ is not a cut edge, deletion leaves one rooted map $\mathcal M'$.  The
remaining three cases are defined by the change in the number $F$ of faces:
\[
\begin{array}{c|c|c}
\text{type of }e & F(\mathcal M')-F(\mathcal M)
  & \eta(\mathcal M)-\eta(\mathcal M')\\ \hline
\text{border} & -1 & 0\\
\text{cross-border} & 0 & 1.
\end{array}
\]
Thus a border separates two distinct faces, which merge after deletion,
whereas deleting a cross-border preserves the number of faces.  

The remaining possibility is $F(\mathcal M')=F(\mathcal M)+1$, and $e$ is
then called a \emph{handle}.  For fixed $\mathcal M'$, the two possible
local reattachments of the root ribbon differ by reversing one attachment;
twisting the root ribbon therefore defines an involution
$\mathcal M\leftrightarrow\tau(\mathcal M)$.  Both root edges are handles,
and both delete to the same $\mathcal M'$.  The measures of non-orientability of these two maps satisfy
\begin{equation}
  \{\eta(\mathcal M),\eta(\tau(\mathcal M))\}
  =\{\eta(\mathcal M'),\eta(\mathcal M')+1\}.
  \label{eq:LaCroix-handle-rule}
\end{equation}
Note that when neither member of the pair is orientable, there is no canonical choice to allocate the two values in \eqref{eq:LaCroix-handle-rule}.  If one
member is orientable, it receives the smaller value, necessarily zero.
\footnote{To define $\eta$ map by map, one must choose consistently which member of
each ambiguous handle pair receives the smaller value.  La Croix gives one
such rule using a canonical spanning tree in
\cite[Section~6.4]{LaCroix}. } This choice does not affect the map generating
series.  Indeed, each handle pair contributes
\[
  b^{\eta(\mathcal M')}+b^{\eta(\mathcal M')+1}
  =(1+b)b^{\eta(\mathcal M')},
\]
independently of which member receives the smaller value.  Hence, for every
family of maps preserved by $\tau$---in particular, maps with fixed face
degrees---the number of maps having any prescribed value of $\eta$ is
independent of the tie-breaking rule.

If $\mathcal M$ has Euler characteristic $2-2g$, we call $g$ its genus.
The above recursion then implies
$0\leq\eta(\mathcal M)\leq2g$.  Moreover, $\eta(\mathcal M)=0$ if and
only if the surface in which $\mathcal M$ is 
embedded admits an orientation.

\begin{figure}[t]
\centering
\begin{tikzpicture}[
  >=stealth,
  every node/.style={font=\small},
  source/.style={draw,rounded corners,align=center,minimum height=8mm,
    text width=4.0cm,fill=gray!8},
  outcome/.style={draw,rounded corners,align=center,minimum height=12mm,
    text width=3.2cm}
]
\node[source] (map) at (0,0)
  {$\mathcal M$ with root edge $e$\\ delete the ribbon of $e$};
\node[outcome] (bridge) at (-4.1,-2.0)
  {\textbf{bridge}\\two rooted components
   $\mathcal M_1,\mathcal M_2$\\
   $\eta=\eta_1+\eta_2$};
\node[source] (single) at (2.0,-2.0)
  {one rooted map $\mathcal M'$\\compare the face numbers};
\draw[->] (map) -- node[above left,align=center]{disconnects} (bridge);
\draw[->] (map) -- node[above right,align=center]{remains connected} (single);

\node[outcome] (border) at (-2.8,-4.6)
  {\textbf{border}\\$F(\mathcal M')=F(\mathcal M)-1$\\
   $\eta(\mathcal M)=\eta(\mathcal M')$};
\node[outcome] (cross) at (1.4,-4.6)
  {\textbf{cross-border}\\$F(\mathcal M')=F(\mathcal M)$\\
   $\eta(\mathcal M)=\eta(\mathcal M')+1$};
\node[outcome] (handle) at (5.6,-4.6)
  {\textbf{handle}\\$F(\mathcal M')=F(\mathcal M)+1$\\
   $\mathcal M\leftrightarrow\tau(\mathcal M)$};
\draw[->] (single) -- (border);
\draw[->] (single) -- (cross);
\draw[->] (single) -- (handle);
\end{tikzpicture}
\caption{The root-edge deletion recursion.  }%

\label{fig:LaCroix-root-edge-deletion}
\end{figure}
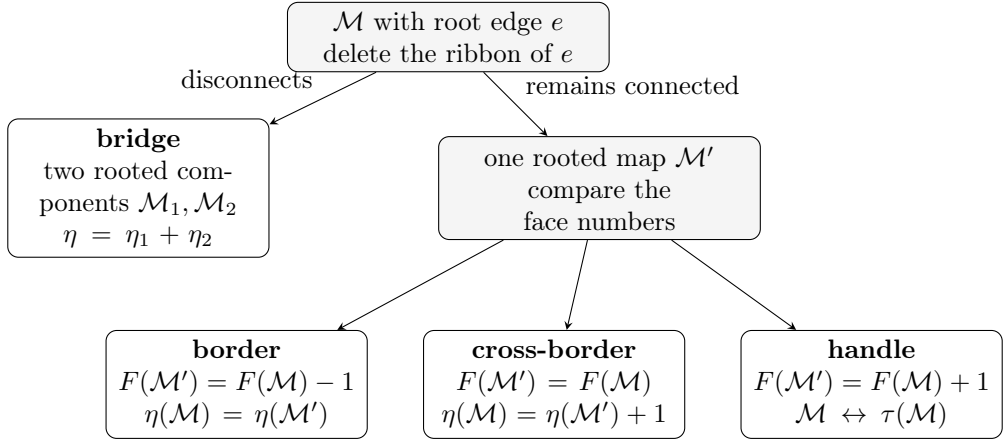

\subsection{Marginal map series}

Fix $E\in\mathbb Z_{\geq1}$ and a face-degree partition
$\nu=(k_1,\ldots,k_m)\vdash2E$.  Recall that
$m_d(\nu)=\#\{p:k_p=d\}$, and write
\[
  y_\nu:=\prod_{p=1}^m y_{k_p},
  \qquad
  |\operatorname{Aut}(\nu)|
  :=\prod_{d\geq1}m_d(\nu)!.
\]
For $g\in\frac12\mathbb Z_{\geq0}$ and
$\eta\in\{0,\ldots,2g\}$, let $r_{g,\nu,\eta}$ be the number of
connected rooted maps with $E$ edges, $m$ faces of degrees $\nu$, Euler
characteristic $2-2g$, and non-orientability statistic $\eta$.
Here and throughout, $g$ is the genus in the convention $\chi=2-2g$; it
may be half-integral in the non-orientable case and agrees with the usual
genus in the orientable case.

Euler's relation shows that such a map has
\[
  v=E-m+2-2g
\]
vertices.  Conversely, for $v\geq1$, put
$g=(E-m+2-v)/2$ and define
\begin{equation*}
  d_{v,\nu}(b)
  :=
  \sum_{\eta=0}^{2g}r_{g,\nu,\eta}b^\eta,
\end{equation*}
with $d_{v,\nu}(b):=0$ if
$g\notin\frac12\mathbb Z_{\geq0}$.  These are precisely the marginal
$b$-polynomials, whose coefficient of $b^\eta$ counts the
maps with fixed $(v,\nu,\eta)$, see
\cite[Corollary~4.17]{LaCroix}.

\begin{definition}[Marginal map series]
\label{def:LaCroix-marginal-map-series}
For $\boldsymbol y=(y_1,y_2,\ldots)$, define
\begin{equation*}
  \mathcal M(x,\boldsymbol y,z;b)
  :=
  \sum_{E\geq1}
  \sum_{\nu\vdash2E}
  \sum_{v\geq1}
  d_{v,\nu}(b)x^v y_\nu z^E.
\end{equation*}
\end{definition}

Let $\langle\,\cdot\,\rangle_{N,\beta}$ denote expectation under the
Gaussian $\beta$ ensemble
\begin{equation*}
  \frac1{Z_{N,\beta}}
  \prod_{i<j}|u_i-u_j|^\beta
  \exp\left(-\frac\beta4\sum_{i=1}^N u_i^2\right)
  \prod_{i=1}^N\mathrm du_i,
  \qquad
  b=\frac2\beta-1,
\end{equation*}
and put $p_k(\boldsymbol u)=\sum_{j=1}^N u_j^k$.

\begin{theorem}[Gaussian $\beta$-ensemble realization]
\label{thm:LaCroix-Gaussian-realization}
For every $N\in\mathbb Z_{\geq1}$ and $b>-1$,
\begin{equation}
  \mathcal M(N,\boldsymbol y,z;b)
  =
  2(1+b)z\frac{\partial}{\partial z}
  \log\left\langle
    \exp\left\{
      \frac1{1+b}\sum_{k\geq1}
      \frac{y_k z^{k/2}}{k}p_k(\boldsymbol u)
    \right\}
  \right\rangle_{N,\beta}.
  \label{eq:LaCroix-Gaussian-series}
\end{equation}
Equivalently, for every $\nu=(k_1,\ldots,k_m)\vdash2E$,
\begin{align}
  \frac{2E(1+b)^{1-m}}
       {|\operatorname{Aut}(\nu)|\prod_{p=1}^m k_p}
  \kappa_{N,\beta}(p_{k_1},\ldots,p_{k_m})
  &=
  \sum_{v\geq1}N^v d_{v,\nu}(b)
  \notag\\
  &=
  \sum_{\substack{g\in\frac12\mathbb Z_{\geq0}\\
                    E-m+2-2g\geq1}}
  N^{E-m+2-2g}
  \sum_{\eta=0}^{2g}r_{g,\nu,\eta}b^\eta.
  \label{eq:LaCroix-coefficient-match}
\end{align}
\end{theorem}

\begin{proof}
Equation~\eqref{eq:LaCroix-Gaussian-series} is a combination of
\cite[Corollary~3.19]{LaCroix} and
\cite[Corollary~4.17]{LaCroix}.
In the logarithm in \eqref{eq:LaCroix-Gaussian-series}, the coefficient of
$y_\nu z^E$ is
\[
  \frac{\kappa_{N,\beta}(p_{k_1},\ldots,p_{k_m})}
  {|\operatorname{Aut}(\nu)|(1+b)^m\prod_{p=1}^m k_p}.
\]
The operator $2(1+b)z\partial_z$ multiplies this coefficient by
$2E(1+b)$.  Comparing with Definition~\ref{def:LaCroix-marginal-map-series}
proves \eqref{eq:LaCroix-coefficient-match}.
\end{proof}

We now normalize $r_{g,\nu,\eta}$, so that the coefficient identity can be
compared directly with the joint cumulants of power sums under edge scaling.  Define
\begin{equation}
  \operatorname{Map}_{g,\eta}(\boldsymbol k)
  :=
  \frac{|\operatorname{Aut}(\nu)|\prod_{p=1}^m k_p}{2E}
  r_{g,\nu,\eta},
  \qquad
  \operatorname{Map}^{(b)}_g(\boldsymbol k)
  :=
  \sum_{\eta=0}^{2g}
  \operatorname{Map}_{g,\eta}(\boldsymbol k)b^\eta.
  \label{eq:normalized-map-count}
\end{equation}
For
\begin{equation*}
  P_k^{(N)}
  :=\sum_{j=1}^N\left(\frac{u_j}{2\sqrt N}\right)^k,
\end{equation*}
Theorem~\ref{thm:LaCroix-Gaussian-realization} becomes
\begin{equation}
  (1+b)^{1-m}
  \kappa_{N,\beta}
  \bigl(P_{k_1}^{(N)},\ldots,P_{k_m}^{(N)}\bigr)
  =
  4^{-E}
  \sum_{g\in\frac12\mathbb Z_{\geq0}}
  N^{2-m-2g}\operatorname{Map}^{(b)}_g(\boldsymbol k),
  \label{eq:Gaussian-cumulant-map-expansion}
\end{equation}
where terms corresponding to nonexistent maps are zero.  The factor $4^{-E}$
comes from $\sum_{p=1}^m k_p=2E$ and the rescaling
$p_k\mapsto(2\sqrt N)^{-k}p_k$.

Finally, $\operatorname{Map}^{(b)}_g(\boldsymbol k)$ admits two useful
decompositions.  Its monomial expansion in \eqref{eq:normalized-map-count}
records the fixed value of $\eta$. And the stronger decomposition
\cite[Theorem~4.22]{LaCroix} introduces nonnegative integers
$\widehat r_{g,\nu,i}$ through
\begin{equation*}
  \sum_{\eta=0}^{2g}r_{g,\nu,\eta}b^\eta
  =
  \sum_{i=0}^{\lfloor g\rfloor}
  \widehat r_{g,\nu,i}\,
  b^{2g-2i}(1+b)^i.
\end{equation*}
Indeed, $2^i\widehat r_{g,\nu,i}$ is the number of rooted maps whose
iterated root-edge deletion encounters exactly $i$ handles.  Thus $b$ marks
cross-borders, while $1+b$ is the effective handle weight.  Define
\begin{equation*}
  \operatorname{Map}^{\mathrm{handle}}_{g,i}(\boldsymbol k)
  :=
  \frac{|\operatorname{Aut}(\nu)|\prod_{p=1}^m k_p}{2E}
  \widehat r_{g,\nu,i}.
\end{equation*}
Then
\begin{equation}
  \operatorname{Map}^{(b)}_g(\boldsymbol k)
  =
  \sum_{i=0}^{\lfloor g\rfloor}
  \operatorname{Map}^{\mathrm{handle}}_{g,i}(\boldsymbol k)
  b^{2g-2i}(1+b)^i.
  \label{eq:normalized-LaCroix-basis}
\end{equation}
Equivalently, $\operatorname{Map}^{\mathrm{handle}}_{g,i}$ is $2^{-i}$ times
the same normalization of the number of maps with exactly $i$ encountered
handles.  Substituting \eqref{eq:normalized-LaCroix-basis} into
\eqref{eq:Gaussian-cumulant-map-expansion} gives
\begin{equation*}
  (1+b)^{1-m}
  \kappa_{N,\beta}
  \bigl(P_{k_1}^{(N)},\ldots,P_{k_m}^{(N)}\bigr)
  =
  4^{-E}
  \sum_{g\in\frac12\mathbb Z_{\geq0}}
  N^{2-m-2g}
  \sum_{i=0}^{\lfloor g\rfloor}
  \operatorname{Map}^{\mathrm{handle}}_{g,i}(\boldsymbol k)
  b^{2g-2i}(1+b)^i.
\end{equation*}

\subsection{The handle basis in the Brownian functional}

Let $\mathbf L_{b,r}^{\mathrm{conn}}[\boldsymbol k]$ be the sum in
\eqref{eq:intro-fixed-functional} over assignments whose multigraph on $[m]$
is connected, and put
\begin{equation*}
  \mathbf L_b^{\mathrm{conn}}[\boldsymbol k]
  :=\sum_{r\geq m-1}\mathbf L_{b,r}^{\mathrm{conn}}[\boldsymbol k].
\end{equation*}
By Theorem \ref{thm:main-Airy-Laplace} and factorization over connected components, we have for $\beta\ge 1$  (equivalently, $b\in(-1,1]$),
\begin{equation}
  \kappa\bigl(\mathcal Z_\beta(s_1),\ldots,
  \mathcal Z_\beta(s_m)\bigr)
  =C_2^m\mathbf L_b^{\mathrm{conn}}
  [\mu_+\mathfrak c_2\boldsymbol s].
  \label{eq:Airy-cumulant-connected-functional}
\end{equation}

The dependence of the functional on $b$ can be separated exactly.  Every jump
pair contributes one factor $1+b$, and expanding the area exponentials to
total degree $\rho$ contributes $b^\rho$.  We therefore define nonnegative
Brownian coefficients $\mathbf B_{r,\rho}^{\mathrm{conn}}$ by
\begin{equation*}
  (1+b)^{-r}\mathbf L_{b,r}^{\mathrm{conn}}[\boldsymbol k]
  =\sum_{\rho\geq0}b^\rho
    \mathbf B_{r,\rho}^{\mathrm{conn}}[\boldsymbol k].
\end{equation*}
Equivalently, $\mathbf B_{r,\rho}^{\mathrm{conn}}$ is obtained by retaining
exactly $r$ jump pairs and total area degree $\rho$ in the Brownian integral.
Brownian scaling gives
\begin{equation}
  \mathbf B_{r,\rho}^{\mathrm{conn}}[a\boldsymbol k]
  =a^{3r+\frac32\rho-\frac32m}
   \mathbf B_{r,\rho}^{\mathrm{conn}}[\boldsymbol k],
  \qquad a>0.
  \label{eq:Brownian-r-rho-homogeneity}
\end{equation}

Now use connectedness, which forces $r\geq m-1$, and set
\begin{equation}
  i=r-m+1\ge 0,
  \qquad
  2g=2i+\rho.
  \label{eq:Brownian-topological-indices}
\end{equation}
Then $i$ is the cycle rank of the connected assignment multigraph, and
\begin{align}
  (1+b)^{1-m}\mathbf L_b^{\mathrm{conn}}[\boldsymbol k]
  &=%
  \sum_{g\in\frac12\mathbb Z_{\geq0}}
  \sum_{i=0}^{\lfloor g\rfloor}
  b^{2g-2i}(1+b)^i
  \mathbf B_{m-1+i,\,2g-2i}^{\mathrm{conn}}[\boldsymbol k],
  \label{eq:Brownian-LaCroix-basis}\\
  \mathbf B_{m-1+i,\,2g-2i}^{\mathrm{conn}}[a\boldsymbol k]
  &=%
  a^{3g-3+\frac32m}
  \mathbf B_{m-1+i,\,2g-2i}^{\mathrm{conn}}[\boldsymbol k].
  \label{eq:Brownian-topological-homogeneity}
\end{align}
Thus the Brownian expansion already comes in precisely the handle basis
$b^{2g-2i}(1+b)^i$: the area degree matches the cross-border exponent
$2g-2i$, while the cycle rank matches the handle index $i$. We call $3g-3+\frac32m$ its homogeneous degree under the simultaneous scaling $\boldsymbol k\mapsto a\boldsymbol k$.  For an absolutely convergent expansion into terms homogeneous under this scaling, let $\operatorname{Hom}_d$ denote the component of degree $d$.

\subsection{Asymptotic map counts}

Fix $\boldsymbol k\in\mathbb R_{>0}^m$, and choose integers $K_{p,N}$ such
that
\begin{equation}
  K_{p,N}=k_pN^{2/3}+O(1),
  \qquad
  \sum_{p=1}^m K_{p,N}\ \text{is even}.
  \label{eq:map-edge-powers}
\end{equation}
Write $2E_N=\sum_pK_{p,N}$.  We start with an elementary
coefficient-extraction principle.

\begin{lemma}[Positive power-series extraction]
\label{lem:positive-power-series-extraction}
Let
$F_N(q)=\sum_{n\geq0}a_{N,n}q^{n}$ have nonnegative coefficients.  Suppose that
on an open interval $I\subset(0,1)$,
\[
  F_N(q)\longrightarrow F(q)=\sum_{n\geq0}a_{n}q^{n}<\infty,
\]
and $a_{n}\geq0$ for every $n$.
Then $a_{N,n}\to a_{n}$ for every fixed $n$.
\end{lemma}

\begin{proof}
Choose $q_0$ in the interior of $I$.  If $F(q_0)=0$, then nonnegativity gives
\[
  0\leq a_{N,n}q_0^{n}\leq F_N(q_0)\longrightarrow0
\]
for every fixed $n$, so the conclusion is immediate.  Suppose henceforth
that $F(q_0)>0$.  Then $F_N(q_0)>0$ for all sufficiently large $N$.  Regard
$\sum_{n} a_{N,n}q_0^{n}\delta_{n}$ and
$\sum_{n} a_{n}q_0^{n}\delta_{n}$ as finite measures on
$\mathbb Z_{\geq0}$.  After normalizing them to have total mass one, their
Laplace transforms are
\[
  \frac{F_N(q_0e^{-s})}{F_N(q_0)}
  \longrightarrow
  \frac{F(q_0e^{-s})}{F(q_0)}
\]
in a neighborhood of the origin.  By \cite[Theorem~3]{Cur}, the normalized
measures converge weakly.  Since their support is discrete, the mass of
every singleton converges, and hence $a_{N,n}\to a_{n}$.
\end{proof}

\begin{theorem}[Asymptotic marginal map coefficients]
\label{thm:asymptotic-LaCroix-coefficients}
For every $g\in\frac12\mathbb Z_{\geq0}$ and
$0\leq i\leq\lfloor g\rfloor$, the following limit exists and is independent
of the approximating integer sequences in~\eqref{eq:map-edge-powers}:
\begin{equation}
  \boxed{
  \lim_{N\to\infty}
  N^{2-m-2g}4^{-E_N}
  \operatorname{Map}^{\mathrm{handle}}_{g,i}
  (K_{1,N},\ldots,K_{m,N})
  =
  2C_2^m
  \mathbf B_{m-1+i,\,2g-2i}^{\mathrm{conn}}
  \left[\frac{\mu_+\mathfrak c_2}{2}\boldsymbol k\right].
  }
  \label{eq:asymptotic-LaCroix-coefficient}
\end{equation}
In particular, the asymptotic number of maps with genus $g$ and $i$ handles has an explicit finite
nonnegative Brownian integral formula.
\end{theorem}

\begin{proof}
For $b\in[0,1]$, define
\begin{equation*}
  A_{N,g}^{(b)}(\boldsymbol k)
  :=N^{2-m-2g}4^{-E_N}
    \operatorname{Map}^{(b)}_g
    (K_{1,N},\ldots,K_{m,N}).
\end{equation*}
Let \(q\in(0,1)\).  Set
\[
  L_N^-:=\left\lceil\frac Nq\right\rceil,\qquad
  L_N^+:=\left\lfloor\frac Nq\right\rfloor,
  \qquad q_N^\pm:=\frac{N}{L_N^\pm},
\]
and in the argument below let \(L_N\) denote either \(L_N^-\) or \(L_N^+\), with \(q_N=N/L_N\).  Introduce the positive power series
\[
  F_N^{(b)}(z):=\sum_g A_{N,g}^{(b)}(\boldsymbol k)z^{2g}.
\]
Applying~\eqref{eq:Gaussian-cumulant-map-expansion} at matrix size
\(L_N\), while keeping the powers \(K_{p,N}\) fixed, gives the exact
identity
\begin{equation}
  F_N^{(b)}(q_N)
  =
  q_N^{2-m}(1+b)^{1-m}
  \kappa_{L_N,\beta}
  \bigl(P_{K_{1,N}}^{(L_N)},\ldots,
        P_{K_{m,N}}^{(L_N)}\bigr).
  \label{eq:auxiliary-dimension-transform}
\end{equation}

We now pass to the limit in the cumulant on the right-hand side of
\eqref{eq:auxiliary-dimension-transform}.  Recall that a joint cumulant is
a universal finite polynomial in the joint moments of its nonempty
subcollections.

By~\eqref{eq:map-edge-powers} and
\(L_N=N/q+O(1)\),
\[
  K_{p,N}=q^{2/3}k_pL_N^{2/3}+O(1).
\]
Pass first to a subsequence on which the parity vector
\((K_{1,N}\bmod2,\ldots,K_{m,N}\bmod2)\) is fixed.  By
\cite[Theorem~2.8, Corollary~2.10, and Remarks~2.11 and~4.21]{GorinShkolnikovAiry}, the even and odd
high-power traces of the Gaussian beta ensemble converge jointly in
moments.  In the present normalization their limits may be written as
\[
  \mathcal Z_\beta^{(+)}\left(\frac{q^{2/3}k_p}{2}\right)
  +(-1)^{K_{p,N}}
   \mathcal Z_\beta^{(-)}\left(\frac{q^{2/3}k_p}{2}\right),
\]
where \(\mathcal Z_\beta^{(+)}\) and
\(\mathcal Z_\beta^{(-)}\) arise from two independent\footnote{By ~\cite[Remark~4.21]{GorinShkolnikovAiry}, the even and odd
high-power traces admit joint limiting representations in terms of the
two independent Brownian noises \(W_\xi\) and \(W_{\mathfrak a}\).
For the Gaussian \(\beta\)-ensemble, 
\cite[Lemma~2.2]{GorinShkolnikovAiry} gives
$
\frac{s_{\mathfrak a}}{2}=s_\xi.
$
Hence the sum and difference of the two limiting fields are driven by
the independent Brownian motions
$
W_\pm:=\frac{W_\xi\pm W_{\mathfrak a}}{\sqrt{2}}.
$
As observed in Remark~2.11 of~\cite{GorinShkolnikovAiry}, these
correspond to the independent right- and left-edge Airy$_\beta$ limits.} copies of the
Airy$_\beta$ point process.  Multilinearity of cumulants and independence
therefore multiply the common Airy cumulant by
\[
  1+(-1)^{\sum_{p=1}^mK_{p,N}}=2.
\]
Every subsequence has a further parity-fixed subsequence with this same
limit, so the full sequence converges.  Hence
\begin{equation}
  \kappa_{L_N,\beta}
  \bigl(P_{K_{1,N}}^{(L_N)},\ldots,
        P_{K_{m,N}}^{(L_N)}\bigr)
  \longrightarrow
  2\,\kappa\left(
    \mathcal Z_\beta\left(\frac{q^{2/3}k_1}{2}\right),\ldots,
    \mathcal Z_\beta\left(\frac{q^{2/3}k_m}{2}\right)
  \right).
  \label{eq:Gaussian-cumulant-Airy-limit}
\end{equation}

By \eqref{eq:Airy-cumulant-connected-functional} and \eqref{eq:Gaussian-cumulant-Airy-limit},
\eqref{eq:auxiliary-dimension-transform} converges to
\begin{equation*}
  2q^{2-m}(1+b)^{1-m}C_2^m
  \mathbf L_b^{\mathrm{conn}}
  \left[
    \frac{\mu_+\mathfrak c_2q^{2/3}}{2}\boldsymbol k
  \right].
\end{equation*}
By~\eqref{eq:Brownian-LaCroix-basis} and
\eqref{eq:Brownian-topological-homogeneity}, this expression equals
\begin{equation}
  F^{(b)}(q):=
  \sum_{g\in\frac12\mathbb Z_{\geq0}}
  A_g^{(b)}(\boldsymbol k)q^{2g},
  \qquad
  A_g^{(b)}(\boldsymbol k)
  :=2C_2^m\sum_{i=0}^{\lfloor g\rfloor}
  b^{2g-2i}(1+b)^i
  \mathbf B_{m-1+i,\,2g-2i}^{\mathrm{conn}}
  \left[\frac{\mu_+\mathfrak c_2}{2}\boldsymbol k\right].
  \label{eq:limiting-genus-polynomial}
\end{equation}

Applying the preceding conclusion to the two choices
\(L_N=L_N^-\) and \(L_N=L_N^+\) gives
\[
  F_N^{(b)}(q_N^-)\longrightarrow F^{(b)}(q),
  \qquad
  F_N^{(b)}(q_N^+)\longrightarrow F^{(b)}(q).
\]
Since all coefficients are nonnegative,
$F_N^{(b)}(q_N^-)\leq F_N^{(b)}(q)\leq F_N^{(b)}(q_N^+)$.
Absolute convergence of the Brownian functional implies
$F^{(b)}(q)<\infty$ for $q\in(0,1)$.  Thus
$F_N^{(b)}(q)\to F^{(b)}(q)$ throughout this interval, and
Lemma~\ref{lem:positive-power-series-extraction} gives
\begin{equation}
  A_{N,g}^{(b)}(\boldsymbol k)
  \longrightarrow A_g^{(b)}(\boldsymbol k)
  \label{eq:fixed-b-genus-limit}
\end{equation}
for every fixed $g$ and every $b\in[0,1]$.  Both sides of
\eqref{eq:fixed-b-genus-limit} are polynomials in $b$ with nonnegative
coefficients.  A second application of
Lemma~\ref{lem:positive-power-series-extraction}, now with $b$ as the
variable, gives coefficientwise convergence in the monomial basis.
Equation~\eqref{eq:asymptotic-LaCroix-coefficient} then follows from
the fixed triangular change to the handle basis
$b^{2g-2i}(1+b)^i$ in
\eqref{eq:normalized-LaCroix-basis} and
\eqref{eq:limiting-genus-polynomial}, which finishes the proof.
\end{proof}

The original monomial coefficients are immediate consequences.

\begin{corollary}[Fixed non-orientability]
\label{cor:asymptotic-nonorientability}
For $g\in\frac12\mathbb Z_{\geq0}$ and $0\leq\eta\leq2g$, the limit
\begin{equation*}
  \mathfrak M_{g,\eta}(\boldsymbol k)
  :=\lim_{N\to\infty}
  N^{2-m-2g}4^{-E_N}
  \operatorname{Map}_{g,\eta}
  (K_{1,N},\ldots,K_{m,N})
\end{equation*}
exists and is
\begin{equation*}
  \boxed{
  \mathfrak M_{g,\eta}(\boldsymbol k)
  =2C_2^m
  \sum_{i=0}^{\lfloor g\rfloor}
  \binom{i}{\eta-(2g-2i)}
  \mathbf B_{m-1+i,\,2g-2i}^{\mathrm{conn}}
  \left[\frac{\mu_+\mathfrak c_2}{2}\boldsymbol k\right],
  }
\end{equation*}
where the binomial coefficient is zero unless its lower argument belongs to
$\{0,\ldots,i\}$.  Equivalently,
\begin{equation}
  \mathfrak M_{g,\eta}(\boldsymbol k)
  =[b^\eta]\operatorname{Hom}_{\,3g-3+\frac32m}
  \left\{
    2(1+b)^{1-m}C_2^m
    \mathbf L_b^{\mathrm{conn}}
    \left[\frac{\mu_+\mathfrak c_2}{2}\boldsymbol k\right]
  \right\}.
  \label{eq:asymptotic-map-Brownian-formula}
\end{equation}
For fixed $(g,\eta,m)$, both formulas contain only finitely many nonnegative
Brownian integrals.
\end{corollary}

\begin{proof}
Expand $(1+b)^i$ in~\eqref{eq:normalized-LaCroix-basis} and apply
Theorem~\ref{thm:asymptotic-LaCroix-coefficients}.  Formula
\eqref{eq:asymptotic-map-Brownian-formula} is the same identity grouped by
the homogeneous degree in~\eqref{eq:Brownian-topological-homogeneity}.
\end{proof}

\begin{remark}
At $\beta=2$, hence $b=0$, only the terms with $2g-2i=0$ survive.  Thus $g$
is an integer, $i=g$, there are no area marks, and the number of jump pairs is
$r=m-1+g$.  On the Brownian side, orientable genus is therefore exactly the
cycle rank of the connected assignment multigraph.  This is the form used in
the next section to pass from the Airy point process to intersection numbers.
\end{remark}

\section{A probabilistic model for intersection numbers}
\label{sec:intersection-numbers}

We now specialize to $\beta=2$.  The purpose of this section is to combine
the connected Brownian expansion with Okounkov's Airy representation \cite{OkounkovIntersection} of the
Witten--Kontsevich $m$-point functions.  The resulting formula is finite at
every fixed $(g,m)$, has no alternating signs, and is indexed by ordinary
connected multigraphs rather than ribbon graphs.

Recall that for $g,m\in \mathbb Z_{\ge 0}$, the pair $(g,m)$ is stable when $2g-2+m>0$. Let $\overline{\mathcal M}_{g,m}$ denote the Deligne--Mumford moduli space of stable genus-$g$ curves with $m$ marked points. Let
$\psi_p=c_1(\mathcal L_p)$ be the first Chern class of the cotangent line at
the $p^{\mathrm{th}}$ marked point on $\overline{\mathcal M}_{g,m}$.  For a stable pair and $(d_1,\ldots,d_m)\in \mathbb Z_{\geq0}^{m}$,
define the Witten--Kontsevich intersection number by
\begin{equation*}
  \left\langle\tau_{d_1}\cdots\tau_{d_m}\right\rangle_g
  :=\int_{\overline{\mathcal M}_{g,m}}
  \psi_1^{d_1}\cdots\psi_m^{d_m}\in \mathbb{Q},
  \qquad
  d_1+\cdots+d_m=3g-3+m.
\end{equation*}
The genus-$g$ $m$-point polynomial is
\begin{equation*}
  F_{g,m}(\boldsymbol x)
  =
  \sum_{\substack{d_1,\ldots,d_m\geq0\\
  d_1+\cdots+d_m=3g-3+m}}
  \left\langle\tau_{d_1}\cdots\tau_{d_m}\right\rangle_g
  \prod_{p=1}^m x_p^{d_p}.
\end{equation*}
For the two unstable pairs appearing in the Airy generating series, we adopt
the conventions implicit in Okounkov's full point function
\cite[Sections~2.6.1--2.6.3]{OkounkovIntersection}:
\begin{equation}
  F_{0,1}(x)=x^{-2},
  \qquad
  F_{0,2}(x_1,x_2)=(x_1+x_2)^{-1}.
  \label{eq:unstable-intersection-conventions}
\end{equation}
These complete the $m$-point function but are not intersection numbers on a
stable moduli space.  Write
\[
  \mathcal F_m(\boldsymbol x)=\sum_{g\geq0}F_{g,m}(\boldsymbol x),
  \qquad
  Z(s)=\sum_{i\geq1}e^{s\mathcal A_i^{(2)}}.
\]
By \cite[Theorem~1]{OkounkovIntersection}, in our normalization,
\begin{equation}
  \mathcal F_m(\boldsymbol x)
  =
  \frac{(2\pi)^{m/2}}{\sqrt{x_1\cdots x_m}}
  \kappa\left(
    Z\left(\frac{x_1}{2^{1/3}}\right),\ldots,
    Z\left(\frac{x_m}{2^{1/3}}\right)
  \right).
  \label{eq:Okounkov-Airy-correspondence}
\end{equation}
Indeed, Okounkov's connected transform is the alternating sum over set
partitions of cyclic Airy-kernel traces; this is precisely the joint cumulant
of the Airy point-process Laplace statistics $Z(s_p)$.  We next give an
alternative expression in terms of connected Brownian integrals.

\subsection{The positive multigraph formula}

For $\boldsymbol\ell\in\mathfrak A_{m,r}$, let
$\mathfrak I_{\boldsymbol\ell}(\boldsymbol x)$ be the raw Brownian integral
obtained from~\eqref{eq:intro-fixed-functional} at $b=0$ after removing all
normalizing constants.  Explicitly,
\begin{equation*}
  \mathfrak I_{\boldsymbol\ell}(\boldsymbol x)
  =
  \int_{\mathcal D_{\boldsymbol x}(\boldsymbol\ell)}
  \prod_{p=1}^m\mathbf D_p
  \prod_{j=1}^r
  h^j\,\mathrm dt_-^j\,\mathrm dt_+^j\,\mathrm dh^j\,
  \mathrm dH_-^j\,\mathrm dH_+^j.
\end{equation*}
Every factor in the integrand is nonnegative.  Since the area weights are
identically one at $\beta=2$,
\begin{equation}
  \mathbf L_{0,r}^{\boldsymbol\ell}[\boldsymbol x]
  =
  \frac1{r!}(2P_{-1}\sigma)^{-m}
  \left(\frac{\sigma^2}{\mu_+^2}\right)^r
  \mathfrak I_{\boldsymbol\ell}(\boldsymbol x).
  \label{eq:L0-raw-integral}
\end{equation}

For $\boldsymbol\ell\in\mathfrak A_{m,r}$, write
$G(\boldsymbol\ell)$ for the multigraph (possibly with loops) on $[m]$
whose $j^{\mathrm{th}}$ labeled edge
joins $\ell_-^j$ and $\ell_+^j$.  The raw integral and its normalization in
\eqref{eq:L0-raw-integral} factorize over the connected components of
$G(\boldsymbol\ell)$.  The standard moment--cumulant recursion, applied
inductively in $m$, therefore identifies the joint cumulant with the unique
one-block term, namely the sum over connected assignments.  Thus, for $\beta=2$, Theorem \ref{thm:main-Airy-Laplace} specializes to 
\begin{equation}
  \kappa(Z(s_1),\ldots,Z(s_m))
  =
  C_2^m
  \sum_{r\geq m-1}
  \mathbf L_{0,r}^{\mathrm{conn}}
  [\mu_+\mathfrak c_2\boldsymbol s].
  \label{eq:Airy2-connected-jump-expansion}
\end{equation}

Recall that Section~\ref{sec:beta-topological} defines the genus $g$ by
$\chi=2-2g$ for orientable and non-orientable maps alike.  At $\beta=2$ only
orientable maps remain, so $g\in\mathbb Z_{\geq0}$ and $\rho=0$.  Hence

\eqref{eq:Brownian-topological-indices} gives
\begin{equation*}
  g=r-m+1.
\end{equation*}
For a connected assignment multigraph $G(\boldsymbol\ell)$, the graph has $m$ vertices and $r$ edges, so this number is precisely
its first Betti number $b_1(G(\boldsymbol\ell))$, which counts the number of independent
cycles.

\begin{lemma}[Homogeneity]
\label{lem:intersection-homogeneity}
For $r\geq m-1$ and $a>0$,
\begin{equation}
  \mathbf L_{0,r}^{\mathrm{conn}}[a\boldsymbol x]
  =
  a^{3r-\frac32m}
  \mathbf L_{0,r}^{\mathrm{conn}}[\boldsymbol x].
  \label{eq:beta2-connected-homogeneity}
\end{equation}
\end{lemma}

\begin{proof}
At $b=0$, $\mathbf L_{0,r}^{\mathrm{conn}}=\mathbf B_{r,0}^{\mathrm{conn}}$, so the claim
is \eqref{eq:Brownian-r-rho-homogeneity} with $\rho=0$.
\end{proof}

\begin{theorem}[Positive Brownian multigraph formula]
\label{thm:intersection-Brownian-multigraph}
Let $g\in\mathbb Z_{\geq0}$, $m\geq1$, and
$\boldsymbol x\in\mathbb R_{>0}^m$.  For the unstable pairs $(0,1)$ and
$(0,2)$, interpret $F_{g,m}$ by
\eqref{eq:unstable-intersection-conventions}.  Then
\begin{equation}
  \boxed{
  F_{g,m}(\boldsymbol x)
  =
  \frac{(\pi/2)^{m/2}}
  {(g+m-1)!\sqrt{x_1\cdots x_m}}
  \sum_{\boldsymbol\ell\in
  \mathfrak A_{m,g+m-1}^{\mathrm{conn}}}
  \mathfrak I_{\boldsymbol\ell}(\boldsymbol x).
  }
  \label{eq:positive-intersection-formula}
\end{equation}

Equivalently, let $M=(m_{pq})_{1\leq p\leq q\leq m}$ range over connected
multigraphs on the labeled vertex set $[m]$, with loops allowed and with
unlabeled edge copies, where $m_{pq}$ is the multiplicity of the edge
$\{p,q\}$. Suppose that $b_1(M)=g$.  Define $\mathfrak I_M$ to be
$\mathfrak I_{\boldsymbol\ell}$ for any labeling of the edge copies with
these multiplicities.  Permutation symmetry of the jump-pair variables
makes this independent of the chosen labeling.  Then
\begin{equation}
  F_{g,m}(\boldsymbol x)
  =
  \frac{(\pi/2)^{m/2}}{\sqrt{x_1\cdots x_m}}
  \sum_{\substack{M\text{ connected on }[m]\\ b_1(M)=g}}
  \frac{\mathfrak I_M(\boldsymbol x)}
  {\prod_{p\leq q}m_{pq}!}.
  \label{eq:unlabeled-multigraph-intersection-formula}
\end{equation}
\end{theorem}

\begin{proof}
Insert~\eqref{eq:Airy2-connected-jump-expansion} and
\eqref{eq:L0-raw-integral} into
\eqref{eq:Okounkov-Airy-correspondence}, and use
\eqref{eq:beta2-connected-homogeneity} to retain the unique value
$r=g+m-1$.  Set $A:=\mu_+\mathfrak c_2/2^{1/3}$ (which collects the coefficients of $\boldsymbol{x}$ in \eqref{eq:Okounkov-Airy-correspondence} and~\eqref{eq:Airy2-connected-jump-expansion}).  The remaining constant
multiplying the Brownian integral sum separates into its $r$- and
$m$-dependent parts as
\[
  \frac1{r!}
  \left[\frac{\sigma^2}{\mu_+^2}A^3\right]^r
  \left[(2\pi)^{1/2}C_2(2P_{-1}\sigma)^{-1}A^{-3/2}\right]^m.
\]
The two brackets are respectively the factors per jump pair and per marked
point.  From \eqref{eq:intro-functional-constants} and
\eqref{eq:intro-edge-conversion-constants},
\begin{equation*}
  \frac{\sigma^2}{\mu_+^2}A^3=1,
  \qquad
  (2\pi)^{1/2}C_2(2P_{-1}\sigma)^{-1}A^{-3/2}
  =\left(\frac\pi2\right)^{1/2}.
\end{equation*}
Thus the prefactor is $(\pi/2)^{m/2}/r!$.  Taking
$r=g+m-1$ proves \eqref{eq:positive-intersection-formula}.

For fixed multiplicities $(m_{pq})$, the number of labelings of the
$g+m-1$ edge copies is
$(g+m-1)!/\prod_{p\leq q}m_{pq}!$.  Grouping the first formula by these
multiplicities gives~\eqref{eq:unlabeled-multigraph-intersection-formula}.
\end{proof}

\begin{corollary}[Individual intersection numbers]
\label{cor:individual-intersection-number}
For a stable pair $(g,m)$ and
$d_1+\cdots+d_m=3g-3+m$,
\begin{equation*}
  \left\langle\tau_{d_1}\cdots\tau_{d_m}\right\rangle_g
  =
  \frac{(\pi/2)^{m/2}}{(g+m-1)!}
  [\boldsymbol x^{\boldsymbol d}]
  \left[
    \frac1{\sqrt{x_1\cdots x_m}}
    \sum_{\boldsymbol\ell\in
    \mathfrak A_{m,g+m-1}^{\mathrm{conn}}}
    \mathfrak I_{\boldsymbol\ell}(\boldsymbol x)
  \right].
\end{equation*}

\end{corollary}
We note that an individual
multigraph summand need not separately have a polynomial expansion in the
$x_p$; positivity here is positivity of the finite integral representation
on $\mathbb R_{>0}^m$, not a claim of coefficientwise positivity graph by
graph.

\subsection{A possible connection to Airy topological recursion}

The preceding identification of orientable genus with the cycle rank is
particularly suggestive in view of the
topological recursion for the Airy spectral curve.  Recall that, for
\[
  x(z)=\frac{z^2}{2},\qquad y(z)=z,
\]
the Eynard--Orantin recursion takes the form
\begin{align}
  \omega_{g,m+1}(z_0,\boldsymbol z)
  =
  \mathop{\rm Res}_{z=0} K(z_0,z)
  \bigg[
    &\omega_{g-1,m+2}(z,-z,\boldsymbol z)
    \notag\\
    &+
    \sum_{\substack{g_1+g_2=g\\ I\sqcup J=\{1,\ldots,m\}}}^{\prime}
    \omega_{g_1,|I|+1}(z,\boldsymbol z_I)
    \omega_{g_2,|J|+1}(-z,\boldsymbol z_J)
  \bigg],
  \label{eq:Airy-topological-recursion}
\end{align}
where $\boldsymbol z=(z_1,\ldots,z_m)$ and
\[
  K(z_0,z)=\frac{\mathrm dz_0}
  {2z(z_0^2-z^2)\,\mathrm dz}
\]
is the Airy recursion kernel.  The prime means that terms containing an
$\omega_{0,1}$ factor are omitted,
i.e. $(g_1,I)\neq(0,\varnothing)$ and
$(g_2,J)\neq(0,\varnothing)$.  The Laurent coefficients of the differentials
$\omega_{g,m}$ are the Witten--Kontsevich intersection numbers; see
\cite{Kontsevich,EO}.  The two terms in
\eqref{eq:Airy-topological-recursion} have the familiar topological
interpretation of cutting a nonseparating cycle, which lowers the genus by
one, and cutting a separating connection, which produces two connected
components.

Our Brownian multigraph formula has a direct graphical analogue of these two
cases.  
  Since $g=b_1(G)$, removing an edge which lies on a cycle
keeps the graph connected and changes $g$ to $g-1$, whereas removing a
bridge separates the graph into two components whose cycle ranks add to
$g$.  More significantly, the Brownian integrals themselves appear
compatible with such an edge-cutting interpretation, as one can see from the direct calculations in the first cases
$(g,m)=(1,1)$ and $(0,3)$: cutting an edge produces simple open-edge factors, and the
sum over positive graph contributions collapses to the corresponding Airy
intersection correlator.

These observations suggest that
\eqref{eq:positive-intersection-formula} may admit a direct probabilistic
presentation of the Airy topological recursion, in which the two terms of
\eqref{eq:Airy-topological-recursion} arise from cutting respectively a
cycle edge and a bridge in the Brownian multigraph.  Establishing such a
recursion would require an identity for the open-edge Brownian amplitudes;
we do not pursue this here.

\appendix

\section{The discrete beta ensemble at time $2N$}
\label{app:discrete-beta-edge}

This appendix verifies the hypotheses needed for the edge input in
Section~\ref{sec:airy-identification}.  It also records the normalization of
the right-edge scaling.  The choice of time $2N$, rather than the critical
time $N$, keeps the opposite edge away from the hard wall and permits a
direct application of \cite{GH}.

\subsection{Discrete $\beta$-ensemble}
For $\beta=2\theta$, a discrete beta ensemble is a
probability law on the shifted-lattice configurations
\[
  a(N)<\ell_1<\cdots<\ell_N<b(N),\qquad
  \ell_1-a(N),\ b(N)-\ell_N\in\mathbb Z_{>0},
\]
such that
\[
  \ell_{i+1}-\ell_i-\theta\in\mathbb Z_{\geq0},
  \qquad 1\leq i<N.
\]
Its probability mass function has the form
\begin{equation}
  \mathbb P_N(\boldsymbol\ell)
  =\frac1{Z_N}
   \prod_{1\leq i<j\leq N}
   \frac{\Gamma(\ell_j-\ell_i+1)\Gamma(\ell_j-\ell_i+\theta)}
        {\Gamma(\ell_j-\ell_i)\Gamma(\ell_j-\ell_i+1-\theta)}
   \prod_{i=1}^N w_N(\ell_i),
  \qquad
    w_N(x)=\exp\left(-NV_N\left(\frac{x}{N}\right)\right).
  \label{eq:appendix-general-discrete-beta}
\end{equation}
Ensembles of the form \eqref{eq:appendix-general-discrete-beta} are discrete analogues of continuous log gases; see, e.g., \cite{BGG,GH,DD,DK,DZ,BorotGorinGuionnet}.
Introduce the increasing particle coordinates
\begin{equation}
  \ell_i=\lambda_{N-i+1}+(i-1)\theta,
  \qquad i=1,\ldots,N.
  \label{eq:appendix-increasing-particles}
\end{equation}
Then the time-$2N$ Jack--Plancherel law becomes
\begin{equation*}
  \mathbb P_N(\boldsymbol\ell)
  =\frac1{Z_N}
  \prod_{1\leq i<j\leq N}
  \frac{\Gamma(\ell_j-\ell_i+1)
        \Gamma(\ell_j-\ell_i+\theta)}
       {\Gamma(\ell_j-\ell_i)
        \Gamma(\ell_j-\ell_i+1-\theta)}
  \prod_{i=1}^N\frac{(2\theta N)^{\ell_i}}{\Gamma(\ell_i+1)},
\end{equation*}
and fits into the discrete beta ensemble framework.   Its one-particle weight and rescaled potential are
\begin{equation}
  w_N(x)=\frac{(2\theta N)^x}{\Gamma(x+1)},\qquad
  V_N(u)=\frac1N\log\Gamma(Nu+1)-u\log(2\theta N).
  \label{eq:appendix-one-particle-potential}
\end{equation}
The equilibrium measure of $\mathbb P_N(\boldsymbol\ell)$ has the single band (see \cite{DD})
\begin{equation*}
  [A,B]
  =
  \left[\theta(\sqrt2-1)^2,\,
        \theta(\sqrt2+1)^2\right].
\end{equation*}
In particular $A>0$, so the band is separated from the hard wall at zero.
Guionnet and Huang prove optimal rigidity and identify every regular soft
edge with the corresponding edge of a continuous beta ensemble
\cite[Corollary~1.10 and Theorem~1.11]{GH}.  Their results are formulated
under the hypotheses collected in \cite[Section~1.1]{GH}.  We first reduce
to the compactly supported setting required there, and then verify the
remaining hypotheses for the truncated ensemble.
\begin{enumerate}[label=(\roman*)]
  \item Fix
  \[
    0<\widehat a<A<B<\widehat b;
    \qquad
    \text{for instance, }\widehat a=A/2,\quad \widehat b=2B.
  \]
  Set
  \[
    a(N):=\lfloor N\widehat a\rfloor,
    \qquad
    b(N):=(N-1)\theta+
      \left\lceil N\widehat b-(N-1)\theta\right\rceil .
  \]
  Thus \(a(N)=N\widehat a+O(1)\) and
  \(b(N)=N\widehat b+O(1)\), while
  \[
    a(N)\in\mathbb Z,\qquad
    b(N)-(N-1)\theta\in\mathbb Z,
    \qquad
    b(N)-a(N)-(N-1)\theta\in\mathbb Z_{>0}
  \]
  for all sufficiently large \(N\).  Let
  \[
    \mathcal E_N
    =\{a(N)<\ell_1<\cdots<\ell_N<b(N)\},
    \qquad
    \mathbb P_N^{\mathrm{tr}}
    =\mathbb P_N(\,\cdot\mid\mathcal E_N).
  \]
  On \(\mathcal E_N\), \eqref{eq:appendix-increasing-particles} gives
  \[
    \ell_1-a(N)\in\mathbb Z_{>0},\qquad
    b(N)-\ell_N\in\mathbb Z_{>0},
  \]
  and, for \(1\leq i<N\),
  \[
    \ell_{i+1}-\ell_i-\theta
    =\lambda_{N-i}-\lambda_{N-i+1}\in\mathbb Z_{\geq0}.
  \]
  Hence the truncated configurations satisfy exactly the three
  shifted-lattice conditions in \cite[Definition~1.1]{GH}.%

  Equivalently, on the finite state space one uses
  \[
    w_N^{\mathrm{tr}}(x)
    =w_N(x)\mathbf 1_{\{a(N)<x<b(N)\}},
    \qquad
    w_N^{\mathrm{tr}}(a(N))
    =w_N^{\mathrm{tr}}(b(N))=0.
  \]
  The one-particle large-deviation estimate for this Jack--Plancherel
  weight \cite[Section~6.3.1]{DD} gives constants \(c,C>0\) such that
  \[
    \mathbb P_N(\mathcal E_N^{\mathrm c})\leq Ce^{-cN}.
  \]
  Consequently
  \(\|\mathbb P_N-\mathbb P_N^{\mathrm{tr}}\|_{\mathrm{TV}}
  \leq Ce^{-cN}\).

  \item On the compact interval
  \([\widehat a,\widehat b]\subset(0,\infty)\), Stirling's
  formula gives, uniformly in \(u\),
  \begin{equation}
    V_N(u)=u\log\frac{u}{2\theta}-u
    +O\left(\frac{\log N}{N}\right).
    \label{eq:appendix-potential-limit}
  \end{equation}
  Thus the truncated potential converges uniformly to
  \(V(u)=u\log(u/(2\theta))-u\), and
  \eqref{eq:appendix-one-particle-potential}--%
  \eqref{eq:appendix-potential-limit} verify
  \cite[Assumption~1.2]{GH} for
  \(\mathbb P_N^{\mathrm{tr}}\).

  \item To verify
  \cite[Assumption~1.3]{GH},  note that whenever both $x-1$ and $x$ lie in the interior of the truncated state
  space,
  \begin{equation*}
    \frac{w_N^{\mathrm{tr}}(x)}
         {w_N^{\mathrm{tr}}(x-1)}
    =\frac{2\theta N}{x},\qquad
    \phi^+(z)=2\theta,\quad \phi^-(z)=z.
  \end{equation*}

  At the artificial endpoints, we use the argument of
  \cite[Section~6.2]{GH}, which deals with these compactly supported discrete ensembles in a slightly different setting. The Nekrasov expression right above \cite[Corollary 6.1]{GH},
  \[
    \psi_N^-(\xi)\mathbb E_{\widehat P_N}
    \left[\prod_{i=1}^N\left(1-\frac{\theta}{\xi-\ell_i}\right)\right]
    +\psi_N^+(\xi)\mathbb E_{\widehat P_N}
    \left[\prod_{i=1}^N\left(1+\frac{\theta}{\xi-\ell_i-1}\right)\right],
  \]
  where $\widehat P_N=\mathbb P_N^{\mathrm{tr}}$, may have poles at
  $a(N)+1$ and $b(N)$. Since the cutoffs lie strictly inside the voids,
  the same one-particle large-deviation estimate gives
  $\mathbb P_N^{\mathrm{tr}}(\ell_1=a(N)+1\text{ or }\ell_N=b(N)-1)
  =O(e^{-cN})$. The lattice constraint $|l_j-l_i|\ge \theta|j-i|$ bounds the product in
  either boundary residue by $\prod_{j=1}^{N-1}(1+1/j)=N$; hence both
  residues are $O(Ne^{-cN})$. Subtracting the two poles yields an
  analytic function with an exponentially small error on the contours
  used in \cite{GH}, so the rigidity and edge-universality arguments
  still apply.

  \item Since \(\widehat a<A<B<\widehat b\), the cutoffs lie
  strictly inside the two void regions and do not change the equilibrium
  measure.  Let $\mu$ denote this equilibrium measure in the $\ell/N$-coordinates and let $G_\mu(z):=\int(z-x)^{-1}\mu(\mathrm dx)$ be its Stieltjes transform.  With the notation of \cite[(1.8)]{GH},
  \[
    R_\mu(z)
    =z e^{-\theta G_\mu(z)}+2\theta e^{\theta G_\mu(z)},
    \qquad
    Q_\mu(z)
    =z e^{-\theta G_\mu(z)}-2\theta e^{\theta G_\mu(z)}.
  \]
  Since \(\mu\) has total mass one,
  \(G_\mu(z)=z^{-1}+O(z^{-2})\).  It is proved in \cite{GH} that under Assumptions~1.2 and~1.3 the analytic continuation of
  \(R_\mu\) is entire, and hence
  \(R_\mu(z)=z+\theta\).  Therefore
  \[
    Q_\mu(z)^2
    =R_\mu(z)^2-8\theta z
    =(z+\theta)^2-8\theta z
    =(z-A)(z-B),
  \]
  so, for the branch asymptotic to \(z\) at infinity,
  \[
    Q_\mu(z)=\sqrt{(z-A)(z-B)}.
  \]
  Thus \cite[Assumption~1.5]{GH} holds with \(H\equiv1\).
  Moreover, \(V_N'=V'+O(N^{-1})\) uniformly on a neighborhood of
  \([\widehat a,\widehat b]\), verifying
  \cite[Assumption~1.7]{GH}.
\end{enumerate}
Thus \cite{GH} applies to \(\mathbb P_N^{\mathrm{tr}}\).
The exponentially small total-variation error transfers its rigidity and
edge-convergence conclusions to the original Jack--Plancherel ensemble.  All
applications of \cite{GH} below are understood through this truncation.

\subsection{Equilibrium density and edge normalization}

By \cite[Lemma~6.11, specialized to $t=2$]{DD}, the equilibrium density in
the increasing coordinates is
\begin{equation*}
  \rho_\theta(x)
  =\frac1{\pi\theta}
  \operatorname{arccot}\left(
    \frac{x+\theta}{\sqrt{8\theta x-(x+\theta)^2}}
  \right),
  \qquad A<x<B.
\end{equation*}
The definitions of $\ell_i$ and of the shifted coordinates $y_i$ in
\eqref{eq:shifted-particle-coordinates} give the exact relation
\begin{equation}
  \frac{\ell_{N-i+1}}N
  =\theta y_i+\theta-\frac\theta N.
  \label{eq:appendix-ell-y-relation}
\end{equation}
Consequently, after reversing the labels and letting $N\to\infty$, the
change of variables is $x=\theta(y+1)$.  The pushforward density is therefore
$\rho(y)=\theta\rho_\theta(\theta(y+1))$, namely
\begin{equation*}
  \rho(y)
  =\frac1\pi
  \operatorname{arccot}\left(
    \frac{y+2}{\sqrt{8(y+1)-(y+2)^2}}
  \right),
  \qquad 2-2\sqrt2<y<2+2\sqrt2.
\end{equation*}
The right edge in the $y$-coordinates is therefore
$\mu_+=2+2\sqrt2$, which was also obtained from the decorated path
expansion in Section~\ref{subsec:decorated-path-expansion}.

Near $B$,
\begin{equation*}
  \rho_\theta(x)
  =
  \frac{\sqrt{B-x}}
  {\pi\theta^{3/2}2^{1/4}(1+\sqrt2)}
  \bigl(1+O(B-x)\bigr).
\end{equation*}
In the normalization of \cite[Theorem~1.11 and its right-edge
counterpart]{GH}, $s_B$ is defined by
$\rho_\theta(x)=(s_B/\pi)\sqrt{B-x}\,(1+O(B-x))$.  Hence
\[
  s_B=\frac1{\theta^{3/2}2^{1/4}(1+\sqrt2)}.
\]
Equation~\eqref{eq:appendix-ell-y-relation} then converts the
Guionnet--Huang normalization into
\begin{equation*}
  \mathfrak c_2N^{2/3}(y_i-\mu_+),
  \qquad
  \mathfrak c_2
  =\theta s_B^{2/3}
  =2^{-1/6}(1+\sqrt2)^{-2/3}.
\end{equation*}
The deterministic $-\theta/N$ term in
\eqref{eq:appendix-ell-y-relation} contributes only $O(N^{-1/3})$ on this
scale.

We finish this subsection by computing the constant that relates the
Perelomov--Popov measure to the ordinary counting measure.  If $\mu_y$
denotes the limiting measure of the $y$-particles, then
\begin{equation}
  \mu_y=(x\mapsto x/\theta-1)_*\mu,
  \qquad
  \int\frac{\mu_y(\mathrm dv)}{\mu_+-v}=\theta G_\mu(B),
  \label{eq:y-equilibrium-measure-shift}
\end{equation}
because $x=\theta(y+1)$ and $B=\theta(\mu_++1)$.  The square-root
behavior at the right edge implies that $G_\mu(B)<\infty$.  Since
$Q_\mu(B)=0$, the formula for $Q_\mu$ obtained above gives
\[
  B e^{-\theta G_\mu(B)}
  =2\theta e^{\theta G_\mu(B)}.
\]
Using $B=\theta(3+2\sqrt2)$ and
$z_c=2-\sqrt2$, we therefore obtain
\begin{equation}
  \exp\bigl(\theta G_\mu(B)\bigr)
  =\sqrt{\frac{B}{2\theta}}
  =\frac{1+\sqrt2}{\sqrt2}
  =z_c^{-1}.
  \label{eq:edge-PP-constant}
\end{equation}

\subsection{Rigidity input}

Define the classical locations $\gamma_1<\cdots<\gamma_N$ by
\begin{equation*}
  \frac{i-\tfrac12}{N}
  =\int_0^{\gamma_i}\rho_\theta(x)\,\mathrm dx,
  \qquad i=1,\ldots,N,
\end{equation*}
as in \cite[(1.5)]{GH}, and put
$\widehat i=\min\{i,N+1-i\}$.  Since both sides of $[A,B]$ are
voids, \cite[Corollary~1.10 and its right-edge counterpart]{GH} imply that,
for every $\varepsilon>0$, there is $c_\varepsilon>0$ such that
\begin{equation}
  \mathbb P\left(
    \exists i\in\{1,\ldots,N\}:
    \left|\frac{\ell_i}{N}-\gamma_i\right|
    >
    \frac{N^\varepsilon}
         {N^{2/3}\widehat i^{1/3}}
  \right)
  \leq 2\exp\bigl(-c_\varepsilon(\log N)^2\bigr).
  \label{eq:appendix-optimal-rigidity}
\end{equation}
Here we combined the left- and right-edge statements, taking a fixed positive
distance from the opposite edge in each application.

For every fixed $\delta\in(0,1)$, the square-root behavior of the density
and compactness away from the opposite edge give
\[
  B-\gamma_{N-i+1}\asymp_\delta(i/N)^{2/3},
  \qquad 1\leq i\leq(1-\delta)N.
\]
Combining this with \eqref{eq:appendix-ell-y-relation} and
\eqref{eq:appendix-optimal-rigidity}, we obtain the form used in
Section~\ref{sec:airy-identification}: for every
$\varepsilon>0$ and $\delta\in(0,1)$, there are constants
$c_\delta,C_{\varepsilon,\delta}>0$ such that, with probability at
least $1-2\exp(-c_\varepsilon(\log N)^2)$,
\begin{equation*}
  \mathfrak c_2N^{2/3}(y_i-\mu_+)
  \leq
  C_{\varepsilon,\delta}N^\varepsilon i^{-1/3}
  -c_\delta i^{2/3},
  \qquad 1\leq i\leq(1-\delta)N.
\end{equation*}

\subsection{Conditioning on a good exterior configuration}
\label{subsec:good-exterior-conditioning}

The purpose of this subsection is to prove the unconditional weak
level-repulsion bound stated below.  The argument relies on
\cite[Proposition~4.3]{GH}; we verify in detail that its good-boundary and
cutoff requirements are satisfied in the present setting.

Fix the sufficiently small parameters $\mathfrak a>0$ and $r_{GH}>0$  as in
\cite[Definition~4.1]{GH}, let $\Theta$ be the cutoff function in
\cite[(4.1)]{GH}, and choose $0<\mathfrak b<1/13$ sufficiently small that
\[
  L^{4/3}N^{-1/3}\ll s\ll\min\{L^{-2},N^{-\mathfrak a}\}
\]
is nonempty when $L:=\lfloor N^{\mathfrak b}\rfloor$.  For such $s$, define
\[
  \mathcal B_N
  :=
  \left\{
    \min_{1\leq k\leq L}(Y_{k,N}-Y_{k+1,N})
    \leq sL^{-1/3}N^{1/3}
  \right\}.
\]
We will prove
\begin{equation}
  \mathbb P_N(\mathcal B_N)\leq CsL^3+o(1).
  \label{eq:unconditional-weak-level-repulsion}
\end{equation}

As above, we first use the compactly
supported truncation; its error is $e^{-cN}$.  Condition on the exterior
particles $(Y_{L+1,N},\ldots,Y_{N,N})$.  In the normalized coordinates
$Y_{i,N}/(\theta N)$ the right endpoint is
$\mu_+=2+2\sqrt2$.  Reflecting these coordinates across the midpoint of
their limiting band turns this right edge into a left edge.  By
\eqref{eq:appendix-ell-y-relation}, the corresponding edge distance in
the scale of \cite[Section~4]{GH} is
\[
  \theta\left(\mu_+-\frac{Y_{i,N}}{\theta N}+\frac1N\right).
\]
Thus the first $L$ reflected particles are still indexed by
$Y_{1,N},\ldots,Y_{L,N}$, and their gaps are unchanged.  Let
\[
  \mathcal G_N
  :=\{(Y_{L+1,N},\ldots,Y_{N,N})\in R_L^*(\mathfrak a,r_{GH})\},
\]
where $R_L^*(\mathfrak a,r_{GH})$ denotes the reflected version of the good-boundary set
in \cite[Definition~4.1]{GH}.  Theorem~1.8 and Corollary~2.5 of
\cite{GH} give directly
\begin{equation}
  \mathbb P_N(\mathcal G_N^c)=o(1).
  \label{eq:good-exterior-high-probability}
\end{equation}

For a fixed good exterior configuration $\boldsymbol y$, denote the genuine
conditional law by $\mathbb P_{L,\boldsymbol y}^{\mathrm{cond}}$, and the
perturbed law in \cite[(4.1)]{GH} by
$\widetilde{\mathbb P}_{L,\boldsymbol y}
=\mathbb P_L^{\mathrm{dis}}$.  In the present coordinates their
Radon--Nikodym factor is
\[
  q_N(\boldsymbol Y)
  =\exp\left\{-\beta\sum_{i=1}^L
    \Theta\left(
      N^{2/3-\mathfrak a}\theta
      \left(\mu_+-\frac{Y_{i,N}}{\theta N}+\frac1N\right)
    \right)
  \right\}\leq1,
  \qquad
  \frac{\mathrm d\widetilde{\mathbb P}_{L,\boldsymbol y}}
       {\mathrm d\mathbb P_{L,\boldsymbol y}^{\mathrm{cond}}}
  =\frac{q_N}
  {\mathbb E_{L,\boldsymbol y}^{\mathrm{cond}}[q_N]}.
\]
Set
\[
  \mathcal H_N
  :=\left\{
    N^{2/3-\mathfrak a}\theta
    \left(\mu_+-\frac{Y_{1,N}}{\theta N}+\frac1N\right)>-1
  \right\}.
\]
Since $Y_{1,N}>\cdots>Y_{L,N}$ and $\Theta(u)=0$ for $u>-1$,
we have $q_N=1$ on $\mathcal H_N$.  Hence, for every event
$\mathcal A_N$,
\begin{equation}
  \mathbb P_{L,\boldsymbol y}^{\mathrm{cond}}
  (\mathcal A_N\cap\mathcal H_N)
  =
  \mathbb E_{L,\boldsymbol y}^{\mathrm{cond}}[q_N]\,
  \widetilde{\mathbb P}_{L,\boldsymbol y}
  (\mathcal A_N\cap\mathcal H_N)
  \leq
  \widetilde{\mathbb P}_{L,\boldsymbol y}(\mathcal A_N).
  \label{eq:perturbed-to-genuine-conditional}
\end{equation}

It remains to control $\mathcal H_N$ under the original ensemble.  Apply
right-edge rigidity with an exponent $\mathfrak a_0<\mathfrak a$.  With probability $1-o(1)$,
\[
  \frac{Y_{1,N}}{\theta N}
  \leq\mu_++C N^{-2/3+a_0},
\]
and therefore
\[
  N^{2/3-\mathfrak a}\theta
  \left(\mu_+-\frac{Y_{1,N}}{\theta N}+\frac1N\right)
  \geq-C\theta N^{\mathfrak a_0-\mathfrak a}+o(1)>-1.
\]
Thus
\begin{equation}
  \mathbb P_N(\mathcal H_N^c)=o(1).
  \label{eq:theta-cutoff-inactive}
\end{equation}

Uniformly on $\mathcal G_N$, \cite[Proposition~4.3,
equation~(4.5)]{GH} bounds the probability of each gap event by
$CsL^2$.  A union bound over the first $L$ gaps gives
\[
  \widetilde{\mathbb P}_{L,\boldsymbol y}(\mathcal B_N)
  \leq CsL^3.
\]
Averaging \eqref{eq:perturbed-to-genuine-conditional} over the exterior
particles and using
\eqref{eq:good-exterior-high-probability}--\eqref{eq:theta-cutoff-inactive}
proves \eqref{eq:unconditional-weak-level-repulsion}.  Proposition~4.3 of
\cite{GH}, and hence this estimate, holds for every $\beta>0$.

\subsection{The Airy edge for $\beta\geq1$}

The right-edge form of \cite[Theorem~1.11]{GH}, with the
normalization computed above, gives for every fixed $r$ and
$\beta\geq1$
\begin{equation*}
  \left(
    \mathfrak c_2N^{2/3}(y_i-\mu_+)
  \right)_{i=1}^r
  \Longrightarrow
  \left(\mathcal A_i^{(\beta)}\right)_{i=1}^r.
\end{equation*}
This input is used only in the final matching step in
Section~\ref{sec:airy-identification}.

{{
\bigskip
 		\footnotesize
        
\noindent Jiaming Xu, \textsc{Department of Mathematics, The Ohio State University, Columbus, OH 43210, USA.} 
     \\
        \textit{Email:} \texttt{jxu0800@gmail.com}
        }}
\end{document}